\documentclass[12pt,a4paper]{amsart}
\usepackage{tikz-cd}
\usepackage{amscd}
\usepackage{amssymb}
\usepackage[centering,text={15.5cm,22cm}]{geometry}
\usepackage{algorithm2e,setspace}
\usepackage{graphicx,color}
\usepackage{xcolor}
\usepackage{pgfplots}
\usepackage[all]{xy}
\usepackage{mathrsfs}
\usepackage{marvosym}
\usepackage{wrapfig,caption}
\usepackage{hyperref,url}
\usepackage{comment}
\usepackage{enumitem}

\definecolor{shadecolor}{rgb}{1,0.9,0.7}

\newtheorem{theorem}{Theorem}[section]
\newtheorem{lemma}[theorem]{Lemma}
\newtheorem{lemma-definition}[theorem]{Lemma-Definition}
\newtheorem{proposition}[theorem]{Proposition}
\newtheorem{corollary}[theorem]{Corollary}

\newtheorem{assumption}[theorem]{Assumption}

\newtheorem{terminology}[theorem]{Terminology}

\theoremstyle{definition}

\newtheorem{definition}[theorem]{Definition}

\newtheorem{example}[theorem]{Example}

\newtheorem{notation}[theorem]{Notation}

\theoremstyle{remark}
\newtheorem{remark}[theorem]{Remark}

\numberwithin{equation}{section}
\numberwithin{figure}{section}

\newcommand{\ZZ} {\mathbb{Z}}
\newcommand{\QQ} {\mathbb{Q}}
\newcommand{\RR} {\mathbb{R}}
\newcommand{\CC} {\mathbb{C}}

\newcommand{\PP} {\mathbb{P}}

\newcommand {\shA}  {\mathcal{A}}

\newcommand {\shL}  {\mathcal{L}}

\newcommand {\shO}  {\mathcal{O}}

\newcommand {\shS}  {\mathcal{S}}

\newcommand {\Aut}  {\operatorname{Aut}}

\newcommand {\Comp} {\operatorname{Comp}}

\newcommand{\coDelta}{%
  \begingroup\normalfont
  \hspace{.6pt}\includegraphics[height=\fontcharht\font`\B]{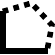}%
  \endgroup
}
\newcommand {\coker} {\operatorname{coker}}
\newcommand {\conv} {\operatorname{conv}}

\newcommand {\Disc} {\operatorname{Disc}}

\newcommand {\Diff} {\operatorname{Diff}}

\newcommand {\eps}  {\varepsilon}

\newcommand {\GL}  {\operatorname{GL}}

\newcommand {\Ham}  {\operatorname{Ham}}
\newcommand {\Hess}  {\operatorname{Hess}}

\newcommand {\Hom}  {\operatorname{Hom}}

\newcommand {\hra} {\hookrightarrow}

\newcommand {\id}  {\operatorname{id}}
\newcommand {\im}  {\operatorname{Im}}
\newcommand {\IIm}  {\operatorname{Im}}

\newcommand {\Int}  {\operatorname{Int}}

\renewcommand {\ker } {\operatorname{ker}}

\newcommand {\la}  {\leftarrow}

\newcommand {\lra}  {\longrightarrow}

\renewcommand {\max} {{\operatorname{max}}}

\newcommand {\mult} {\operatorname{mult}}

\newcommand {\ra}  {\to}

\newcommand {\Reg}  {\operatorname{Reg}}

\newcommand {\Sing} {\operatorname{Sing}}

\newcommand {\Smooth} {\operatorname{Smooth}}

\newcommand {\Sympl} {\operatorname{Symp}}

\newcommand {\tran}  {{\operatorname{tran}}}

\newcommand {\Type}  {{\operatorname{Type}}}

\newcommand {\wt} {\widetilde}

\def\mydate{\ifcase\month \or January\or February\or March\or
April\or May\or June\or July\or August\or September\or October\or
November\or December\fi \space\number\day,\space\number\year}

\begin{document}

%===========================================================

\title
[Tropically constructed Lagrangians in mirror quintic threefolds]
{Tropically constructed Lagrangians in mirror quintic threefolds}
%\mbox{\tiny -preliminary version-}}
\author{Cheuk Yu Mak and Helge Ruddat}

\date{\today}

% author one information
%\author{Helge Ruddat}
\address{\tiny JGU Mainz, Institut f\"ur Mathematik, Staudingerweg 9, 55099 Mainz, Germany}
%\curraddr{}
%\thanks{This work was partially supported by DFG research grant RU 1629/1-1}
\email{ruddat@uni-mainz.de}

% author two information
%\author{Bernd Siebert}
\address{\tiny DPMMS, University of Cambridge, CB3 0WB, UK}
%\curraddr{}
\email{cym22@dpmms.cam.ac.uk}
%\thanks{{\tiny The first author was funded by 
%National Science Foundation under agreement No. DMS-1128155 and by EPSRC (Establish Career Fellowship EP/N01815X/1).
%The second author was funded by DFG grant RU 1629/4-1.
%The authors also gratefully acknowledge the support of Hausdorff Research Institute for Mathematics
%(Bonn), through the Junior Trimester Program on Symplectic Geometry and Representation Theory.
%}}

\begin{abstract}
 We use tropical curves and toric degeneration techniques to construct closed embedded Lagrangian rational homology spheres in a lot of Calabi-Yau threefolds. The homology spheres are mirror dual to the holomorphic curves contributing to the GW invariants.
In view of Joyce's conjecture, these Lagrangians are expected to have special Lagrangian representatives and hence solve a special Lagrangian enumerative problem in Calabi-Yau threefolds.

 We apply this construction to the tropical curves obtained from the 2875 lines on the quintic Calabi-Yau threefold. 
 Each admissible tropical curve gives a Lagrangian rational homology sphere in the corresponding mirror quintic threefold and
the Joyce's weight of each of these Lagrangians equals to the multiplicity of the corresponding tropical curve. 

As applications, we show that disjoint curves give pairwise homologous but non-Hamiltonian isotopic Lagrangians and we check in an example that $>300$ mutually disjoint curves (and hence Lagrangians) arise. 
Dehn twists along these Lagrangians generate an abelian subgroup of the symplectic mapping class group with that rank.
\vspace{-1cm}
\end{abstract}
\maketitle

\setcounter{tocdepth}{1}
\tableofcontents

%%%%%%% INTRODUCTION %%%%%%%%%%%%%%%%%%%%%%%%%%%%%%%%%%%%%%%%%%%%%%%%%%%%%%%%%%%%%%%%%%%%%%%%%%%%%%%%%%%%%%

\section{Introduction}

Special Lagrangian submanifolds of Calabi-Yau threefolds have received much attention due to their role in mirror symmetry.
Based on Thomas and Yau \cite{Thomas01}, \cite{ThomasYau02}, Dominic Joyce \cite{Joyce15} conjectured that
a Lagrangian submanifold $L$ admits a special Lagrangian representative 
(after surgery at a discrete set of times under Lagrangian mean curvature flow) if 
$L$ is a stable object in the derived Fukaya category with respect to an appropriate Bridgeland stability condition.
Therefore, roughly speaking, special Lagrangians correspond to stable objects.
In \cite{Joyce02}, Joyce proposed a counting invariant for rigid special Lagrangians 
(i.e.~special Lagrangian rational homology spheres) 
so that each of these Lagrangians $L$ is weighted by $w(L):=|H_1(L,\ZZ)|$ when it is counted
and, under 
the conjectural correspondence between special Lagrangians and stable objects, Joyce's counting invariant is conjectured to be mirror to the
Donaldson-Thomas invariant.
One possible explanation of the weight $w(L)=|H_1(L,\ZZ)|$ is that objects in the Fukaya category are Lagrangians with local systems and $|H_1(L,\ZZ)|$ is exactly the number of rank one local systems
on $L$, giving $|H_1(L,\ZZ)|$ many different objects in the Fukaya category. (The original explanation in \cite{Joyce02} is different.)
%By analyzing the deformation of special Lagrangians, Joyce proposed that to get an invariant, each spherical Lagrangian rational homology sphere $L$ 
%should be weighted by $w(L):=|H_1(L,\ZZ)|$ when it is counted.

Even before counting, finding special Lagrangians is a challenging problem (\cite{Hit01}, \cite{Joyce05} etc).
The main source of examples is given by the set of real points.
Making a given Lagrangian special is hard. Even without the specialty assumption, there aren't many explicit methods to construct closed embedded Lagrangian submanifolds in 
Calabi-Yau threefolds in the literature, especially when the Calabi-Yau is assumed to be compact and the Lagrangians spherical. 
In this paper, we provide a new method to address the latter difficulty using toric degeneration techniques and tropical curves.

The Lagrangians are constructed by dualizing tropical curves that contribute to the Gromov-Witten invariant of the mirror.
Therefore, even though we do not show that the Lagrangians we construct are (Hamiltonian isotopic to) special Lagrangians,
their quasi-isomorphism classes in the Fukaya category are conjecturally mirror dual to the stable sheaves contributing to the Donaldson-Thomas invariants (via DT/GW correspondence).
It is expected that these Lagrangians would give the full set of stable objects in a fixed $K$-class in the Fukaya category with respect to a stability condition, and hence play an important role towards the enumerative 
study of stable objects in the Fukaya category.
In particular, we find that the weight $w(L)$ indeed coincides with the multiplicity of the corresponding tropical curve which is also how it enters the mirror dual Gromov-Witten count. We had communicated this result to Mikhalkin who then also confirmed it in his approach \cite{Mikhalkin}.

The idea of construction is motivated by Strominger-Yau-Zaslow's (SYZ) conjecture \cite{SYZ} and the construction of cycles in \cite{RuddatSiebert19}. 
Parallel results without connection to enumerative geometry have very recently been achieved independent
from us in the situation where the symplectic manifold is non-compact \cite{Matessi1}, \cite{Matessi2}, \cite{Hicks0,Hicks1}, a toric variety \cite{Mikhalkin}, \cite{Hicks2} or a torus bundle over torus \cite{SheridanSmith}.

Roughly speaking, if there is a Lagrangian torus fibration for a Calabi-Yau manifold and a tropical curve $\gamma$ in the base integral affine manifold such that all edges of $\gamma$ have weight one, then it is easy to construct for each edge $e$ of $\gamma$, a Lagrangian torus times interval $L_e$ lying above $e$ and for each trivalent vertex $v$ of   $\gamma$, a Lagrangian pairs of pants times torus $L_v$ lying above a small neighborhood of $v$. Moreover, these local pieces can be 
constructed in a way that can be patched together smoothly, resulting in a Lagrangian submanifold $L^{\circ}_{\gamma}$. Furthermore, if $\gamma$ hits the discriminant at the end points appropriately, then $L^\circ_{\gamma}$ can be closed up to a closed embedded Lagrangian $L_{\gamma}$, whose diffeomorphism type is determined by the combinatorial type of $\gamma$ \emph{and} the local monodromy at points where the discriminant is hit. 
We will explain this in more details in Section \ref{ss:LagrangianLift}. We call $L_{\gamma}$ a tropical Lagrangian over $\gamma$. 
The key point is that, this construction is straightforward only when we have been given a Lagrangian torus fibration. 
However, the only compact K\"ahler Calabi-Yau threefolds that knowingly admit a Lagrangian torus fibration are torus bundles over a torus.

Our actual construction starts with a family of smooth threefold hypersurfaces $M_t \subset \PP_\Delta$ in a toric $4$-orbifold $\PP_\Delta$ degenerating to $M_0=\partial \PP_\Delta$, the toric boundary divisor of $\PP_\Delta$ with the reduced scheme structure.
 Let $(\partial \PP_\Delta)_{\Sing}$ be the locus of singular points of $\partial \PP_\Delta$, $(\partial \PP_\Delta)_{\Smooth}:=\partial \PP_\Delta \setminus (\partial \PP_\Delta)_{\Sing}$,
 $\Disc:=M_t \cap (\partial \PP_\Delta)_{\Sing}$ be the discriminant, $\pi_\Delta: \PP_\Delta \to \Delta$ be the moment map
 and $\shA:=\pi_{\Delta}(\Disc)$.
 Suppose that $\PP_\Delta$ has at worst isolated Gorenstein orbifold singularities. The singularities are necessarily at the preimage of vertices of $\Delta$ under $\pi_\Delta$, 
 and thus $M_t$ is a smooth threefold for $|t| > 0$ small.

 Starting with a reflexive polytope $\Delta_X$, \cite{MGross05} exhibited a Minkowski summand $\Delta'$ so that $\Delta=\Delta_X+\Delta'$ has the property that $(\partial \Delta, \shA)$ is \emph{simple} (\cite{GrossSiebert06}, Definition 1.60). 
We equip $\partial \Delta\setminus\shA$ with an integral affine structure using the integral affine structure on $\pi_\Delta((\partial \PP_\Delta)_{\Smooth})$
 and the fan structure at the vertices, see Definition 3.13 in \cite{MGross05}, Example~1.18 in \cite{GrossSiebert06}.
 Therefore, we can define tropical curves $\gamma$ in $(\partial \Delta, \shA)$, see Definition \ref{def-trop-curve} below.
 We require $\gamma\cap \shA$ to be the set of univalent vertices of $\gamma$.
%\begin{definition}
%By a \emph{tropical curve} we refer to a connected embedded graph $\gamma$ in $(\partial \Delta, \shA)$ such that
% \begin{enumerate}
%  \item all edges of $\gamma$ are contained in straight lines of rational tangent direction,
%  \item every vertex is trivalent and the primitive directions of the edges adjacent to a vertex are integrally linear equivalent to $\{-e_1,-e_2,e_1+e_2\}$,
%  \item $\gamma$ may have edges that are non-compact in $(\partial \Delta, \shA)$ and the infinity point of such a non-compact edge is required to be a regular point of $\pi_\Delta(\Disc)$
%  (see Assumption \ref{a:singularModel}, \ref{a:assumptionOrbit} and the paragraph before Example \ref{ex:m3m4}),
% \end{enumerate}
%\end{definition}
Let $\Lambda$ be the local system of integral tangent vectors on $\partial \Delta \setminus \shA$.
When $\gamma$ is rigid and $\Lambda$ is trivializable over $\gamma$,
we can associate to $\gamma$ its multiplicity $\mult(\gamma)$ defined in \cite{MR16} (cf. \cite{MR19},\cite{NS06}), which we recall in Section \ref{section-weight-mult}.
The multiplicity depends on the directions of edges of $\gamma$ as well as the monodromy action around $\shA$ near the univalent vertices of $\gamma$.
We call a tropical curve $\gamma$ {\it admissible} if for each univalent vertex $v$, there is a neighborhood $O_v$ of $v$ such that $O_v \cap \shA$
is an embedded curve (rather than two-dimensional). 
The tropical lines that end on the ``internal edges of the quintic curves'' in a mirror quintic threefold are admissible, see Lemma \ref{lem-admissible} -- in the example of Section~\ref{section-computer},  more than half of the lines are admissible. 
Moreover, an admissible tropical curve determines a diffeomorphism type of a $3$-manifold in the way that the diffeomorphism type of a tropical Lagrangian over $\gamma$ is determined by $\gamma$. By slight abuse of terminology, we call the  diffeomorphism type determined by $\gamma$ a Lagrangian lift of $\gamma$ (see Section \ref{ss:LagrangianLift}).
We denote the $\epsilon$-neighborhood of $\gamma$ with respect to the Euclidean distance on $\Delta$ by $W_{\epsilon}(\gamma)$.
Our main theorem is:

\begin{theorem}\label{t:Construction}
%Let $W_{\epsilon}(\gamma)$ denote the $\epsilon$-neighborhood of $\gamma$ with respect to the Euclidean distance on $\Delta$.
 Let $\gamma$ be an admissible tropical curve in $\partial \Delta$.
 For any $\epsilon>0$, there exist a $\delta>0$ such that for all $0<|t|<\delta$, there is
 a closed embedded Lagrangian $L \subset M_t$ such that $\pi_\Delta(L) \subset W_{\epsilon}(\gamma)$ and $L$ is diffeomorphic to a Lagrangian lift of $\gamma$.
 Moreover, whenever $\mult(\gamma)$ is well-defined, we have $w(L)=\mult(\gamma)$.
\end{theorem}

\begin{remark}
For discussion of non-admissible tropical curves, see Remark \ref{r:trivalentVertex1}, \ref{r:trivalentVertex2}.
\end{remark}

\begin{remark}
 While our main examples are mirror quintics, Theorem \ref{t:Construction} applies to all admissible tropical curves that arise from the setup in \cite{MGross05} explained above.
 For example, the tropical curves are not necessarily simply connected. 
\end{remark}

\begin{remark}
While  $\mult(\gamma)$ is only defined when $\Lambda$ is trivializable over $\gamma$ \cite{MR16}, in view of Theorem \ref{t:Construction}, it is tempting to define $\mult(\gamma)$ by $w(L)$ when $\Lambda$ is not trivializable.
We believe that this definition will have application to enumerative problems in algebraic/tropical geometry.
\end{remark}

\begin{remark}
 One can easily generalize $w(L)=\mult(\gamma)$ to all dimensions (see Remark \ref{r:higherMult}). 
 However, it is pointed out to us by Joyce that we do not expect a special Lagrangian counting invariant in dimensions higher than $4$.
\end{remark}

\begin{comment}
In our forthcoming paper, we will apply this construction to tropical curves that
contribute to the Gromov-Witten invariant of the line class in quintic threefolds $M_t^{\vee}$.
In this case,
%Theorem \ref{t:Construction} $(1),(2)$ are always satisfied, and $(3),(4),(5)$ are also satisfied for generic cases.
we will also show that $\{L_i\}$ are homologous but pairwise non-Hamiltonian isotopic.
Computer-aided example will also be given to illustrate that various diffeomorphism types of Lagrangians can be obtained.

%It should be pointed out that when $\{\gamma_i\}_{i=1}^k$ contribute to the Gromov-Witten invariant of the line class in the corresponding quintic threefold $M_t^{\vee}$,
%then $(1),(2)$ above are satisfied.
%Generic choice of triangulation of $\partial \Delta$ can produce a lot of $\gamma_i$ such that $(3),(4),(5)$ are satisfied.
%A computer-aided example is given in the last section.
\end{comment}

\begin{proof}[Sketch of proof of Theorem \ref{t:Construction}]
The construction is divided into two parts: for the geometry away from the discriminant and near the discriminant.
Both constructions rely heavily on the fact we can isotope
$M_t$ symplectically to a nice symplectic hypersurface in local coordinates, as long as the isotopy is away from the discriminant
and does not produce new discriminant (see Lemma \ref{l:GoodDeformation}).

For the construction away from the discriminant, we isotope $M_t$ to a standard form (Lemma \ref{l:ExistenceStandardSymp}) in a chart such that $M_t$
admits a local Lagrangian torus fibration and we can construct a local Lagrangian from the torus fibration (Proposition \ref{p:standardLagModel}).
We have to deal with compatibility of standard forms (Lemma \ref{l:standardTransitionModel}), transition of symplectic charts (Corollary \ref{c:pStandardIndep})
and the trivalent vertices of $\gamma$ (Lemma \ref{l:Gluing3Legs}).
The outcome will be an embedded Lagrangian with toroidal boundaries such that the $\pi_\Delta$-image lies in a small neighborhood of $\gamma$.

Then we need to close up the Lagrangian with toroidal boundaries by Lagrangian solid tori near the discriminant, which is the essential part of the construction.
The basic idea is that we can deform $M_t$ to a particular $M$ such that we have complete control away from the discriminant.
We find an appropriate open subset $V$ of $M$ which is an exact symplectic manifold with contact boundary (Proposition \ref{p:fibrationAfterGoodDeform})
and we have complete control near the contact boundary of $V$.
We show that $V$ is a symplectic bundle over an annulus and we use our control near $\partial V$ to show that the boundaries of the fibers are standard contact $S^3$.
By a famous result of Gromov, each fiber is symplectomorphic to an open $4$-ball (Theorem \ref{t:filling}).
There is a Legendrian $T^2$ inside $\partial V$, which is an $S^1$-bundle over $S^1$ with respect to the symplectic $4$-ball fiber bundle structure on $V$.
This Legendrian $T^2$ can be filled by a Lagrangian solid torus in $V$ by a soft symplectic method (Proposition \ref{p:LagrangianConstructionSingLoci}), which gives the Lagrangian solid torus we need.

Once the Lagrangian is constructed, the statement that $w(L)=\mult(\gamma)$ follows from a simple calculation using \v{C}ech cohomology (see Subsection \ref{ss:MultiplicityWeight})
which was independently obtained in \cite{Mikhalkin} by a different argument after a presentation of our result given by the second author in 2017.
\end{proof}

\begin{comment}

\begin{corollary}\label{c:non-Ham}
 Let $k_{\max}$ be the maximum number of disjoint tropical curves satisfying Theorem \ref{t:Construction}.
 Then there are at least $k_{\max}$ many homologous but pairwise non-Hamiltonian isotopic embedded Lagrangian submanifolds in $M$ in the class $[L_i]$.
\end{corollary}

\begin{proof}[Proof of \ref{c:non-Ham}]
 By Theorem \ref{t:Construction}(1), $L_i$ are rational homology spheres.
 By \cite{FOOObook}, we have well-defined Floer cohomology $HF^*(L_i,L_i)=H^*(S^3)$ and $HF(L_i,L_j)=0$ for all $i \neq j$.
 Floer cohomology are invariant under Hamiltonian isotopy so the result follows.
\end{proof}

\end{comment}

\subsubsection*{Application to symplectic topology}

%When $M_t$ is a pencil of mirror quintics, all the $2875$ degree one
%tropical curves in $(\partial \Delta, \shA)$ are admissible so Theorem \ref{t:Construction} can be applied.
Let $\shS$ be the set of admissible tropical lines in $(\partial \Delta, \shA)$ associated to a pencil of mirror quintics.
We can show that the Lagrangians constructed by Theorem \ref{t:Construction} are homologous and non-Hamiltonian isotopic in the following sense:

\begin{theorem}\label{t:homologous}
 Let $\gamma \in \shS$ and $L_\gamma$ be a Lagrangian obtained by Theorem \ref{t:Construction}.
 For any $\gamma' \in \shS$, we can get a Lagrangian $L_{\gamma'}$ by Theorem \ref{t:Construction}
 such that $[L_\gamma]=[L_{\gamma'}] \in H_3(M,\ZZ)$.
 Moreover, if $\gamma \cap \gamma'=\emptyset$, then $L_\gamma$ is not Hamiltonian isotopic to $L_{\gamma'}$.
\end{theorem}

Theorem \ref{t:homologous}
gives a large number of pairwise homologous but non-Hamiltonian isotopic Lagrangian rational homology spheres, which is a 
rare application to symplectic topology in the literature.

When $L$ is diffeomorphic to a free quotient of a sphere by a finite subgroup of $SO(4)$, we can define Dehn twist along $L$, which is 
an element in the symplectomorphism group $\Sympl(M)$ of $M$.
It is easy to deduce from Theorem \ref{t:Construction} the following:

\begin{corollary}\label{c:Symp}
 Let $k_{\max}$ be the maximum number of disjoint tropical curves satisfying Theorem \ref{t:Construction} such that for each $i=1,\dots,k_{\max}$, the corresponding $L_i$
 is a spherical manifold.
 Then $\pi_0(\Sympl(M))$ contains an abelian subgroup isomorphic to $\mathbb{Z}^{k_{\max}}$.
\end{corollary}

Note that, in generic situations, most tropical curves are disjoint from the others so Corollary \ref{c:Symp} gives a large rank of abelian subgroup in $\pi_0(\Sympl(M))$.

By a computer-aided search for a particular symplectic mirror quintic, we found $354$ pairwise disjoint admissible tropical lines giving $312$ Lagrangian $S^3$ and $42$ Lagrangian $\RR\PP^3$ in the mirror quintic all of which are pairwise disjoint.
The total number of admissible tropical lines in our example is $1451$ out of which $1406$ have multiplicity one (giving Lagrangian $S^3$'s) 
and $45$ have multiplicity two (giving Lagrangian $\RR\PP^3$'s).
The total number of tropical lines in our example is however $2785$ out of which $2695$ have multiplicity one 
%(corresponding to $S^3$'s) 
and $90$ have multiplicity two, 
%(corresponding to $\RR\PP^3$'s)
so the weighted sum is indeed 2875. Their adjacency matrix has full rank, which implies that every tropical line intersects some other tropical line.
Inspection of the center of Figure~\ref{graph:575lines} gives an impression of the meeting of tropical lines, yet tropical lines also meet another across components of the degenerate Calabi-Yau unlike possibly expected. We don't know whether this is a general phenomenon or
due to possibly not having picked the most general deformation. We chose a random small perturbation of the subdivision given in Section \ref{section-subdivision}.

\subsubsection*{Structure of the paper}
In Section \ref{section-lines2Lag} we give some background of SYZ mirror symmetry and 
the tropical curves in the affine base.
We also explain the topology of the Lagrangians and derive some consequences, including Theorem \ref{t:homologous}, by assuming Theorem \ref{t:Construction}, which is 
proved in the subsequent sections.
In Section \ref{ss:toricBasic}, we review toric geometry from symplectic perspective.
%following \cite{AbreuOrbifold} and \cite{Abreu}.
In Section \ref{ss:geometricSetup}, we explain how to perform symplectic isotopy away from the discriminant for our pencil of hypersurfaces.
After that, we explain the construction of the Lagrangians away from the discriminant
and near the discriminant in Section \ref{s:AwayFromSing} and \ref{s:NearSingLoci}, respectively.
We conclude the proof of Theorem \ref{t:Construction} in Section \ref{ss:Concluding}.

\subsubsection*{Acknowledgments}
We thank Paul Seidel for suggesting to study mirror duals of lines and for bringing the authors together at the Institute for Advanced Study in Princeton.
Further thanks for helpful discussions goes to Mohammed Abouzaid, Ivan Smith, Travis Mandel, Johannes Walcher, Hans Jockers and Penka Georgieva. We also thank the two anonymous referees.

%We had useful discussions with Mohammed Abouzaid on his related work and he suggested Hyperk\"ahler rotation to produce the local Lagrangian at a vertex.
%We thank Ivan Smith for discussing the geometry of the discriminant in the affine base and Travis Mandel for discussing the multiplicity. 
%We are also grateful to Johannes Walcher and Hans Jockers for their interest and sharing of ideas.

{\small The first author was funded by 
National Science Foundation under agreement No. DMS-1128155 and by EPSRC (Establish Career Fellowship EP/N01815X/1).
The second author was funded by DFG grant RU 1629/4-1.
The authors thank the support of Hausdorff Research Institute for Mathematics
(Bonn), through the Junior Trimester Program on Symplectic Geometry and Representation Theory.}

%%%%%%% Lines on the Quintic and SYZ story %%%%%%%%%%%%%%%%%%%%%%%%%%%%%%%%%%%%%%%%%%%%%%%%%%%%%%%%%%%%%%%%%%%%%%%%%%%%%%%%%%%%%%

\section{From 2875 lines on the quintic to Lagrangians in the quintic mirror}\label{section-lines2Lag}
The toy model of the SYZ mirror symmetry conjecture is the following. Set $V=\RR^n$ and let $TV$ and $T^*V$ denote the tangent and cotangent bundle. Let $T_{\ZZ}V$ denote the local system on $V$ of integral tangent vectors (using the lattice $2\pi\ZZ^n$ in $\RR^n$). The quotient
$TV/T_{\ZZ}V$ is an $(S^1)^n$-bundle over $V$. Similarly, we can define $T^*_{\ZZ}V$ and another $(S^1)^n$-bundle $T^*V/T^*_{\ZZ}V$. We arrive at dual torus fibrations over $V$,
$$ X:=TV/T_{\ZZ}V \ra V \la T^*V/T^*_{\ZZ}V=:\check X $$
where the left one carries a natural complex structure 
with complex coordinates $(z_j)_{j=1,\dots,n}$ given by $(x, \alpha=\sum_j y_j \frac{\partial}{\partial x_j}) \mapsto (z_j=x_j+ \sqrt{-1} y_j)_{j=1,\dots,n}$ and $x_j$ the $j$th coordinate on $V$. The right one carries a natural symplectic structure inherited from the canonical one of $T^*V$.
Part of the conjecture of SYZ is that mirror symmetry is locally of this form. 
Unless talking about complex tori, in practice there are also singular torus fibers in these bundles for Euler characteristic reasons and we will get back to this.

Note that this toy model gives insight on how a complex submanifold ought to become a Lagrangian submanifold of the mirror dual (see Section 6.3 of \cite{DbraneBook}).
If $W$ is an integrally generated linear subspace of $V$, then $TW/T_{\ZZ}W$ is naturally a complex submanifold of $X$.
On the other hand, $W^\perp/(W^\perp\cap T^*_{\ZZ}V)$
as a subbundle of $T^*V/T^*_{\ZZ}V$ supported over $W$ is a Lagrangian submanifold of $\check X$.
To reach sufficient generality, one needs to run this construction for the situation where $W$ is a tropical variety, i.e.~a polyhedral complex. At a general point it still looks just like the above but then pieces are glued non-trivially when polyhedral parts meet another.
However, this doesn't produce a differentiable submanifold, let alone complex or Lagrangian. 
Improvements on the symplectic side can be made by thickening the tropical $W$ to an amoeba, see \cite{Matessi2}.
In this article, we are only interested in the situation where $W$ is one-dimensional, so a tropical curve,
and the focus will be put on constructing closed Lagrangian submanifolds in Calabi-Yau threefolds using tropical curves.
Whenever $\check X$ compactifies to a projective toric variety and the tropical curve attaches to the codimension two strata in the moment polytope in particular ways, 
Mikhalkin recently gave a construction of closed Lagrangian submanifolds in the projective toric variety \cite{Mikhalkin}. 
On the other hand, no Lagrangian torus fibration is known for any simply-connected compact Calabi-Yau threefold. This is the situation that we are interested in, which is also the subject of the SYZ conjecture.
Luckily, most Calabi-Yau threefolds permit degenerations to a reducible union of toric varieties, introducing the toric techniques we lay out for the quintic and its mirror dual in the next sections.

\subsection{The quintic threefold and its symplectic mirror duals} \label{section-quintic}
The most famous Calabi-Yau threefold is the quintic $X$ in $\CC\PP^4$. 
Its mirror dual $\check X$ is a crepant resolution of an anti-canonical hypersurface in the weighted projective space $\PP_{\Delta_{\check X}}$ associated to the lattice simplex 
$$\Delta_{\check X} = \conv\{e_1,...,e_4,-\sum_j e_j\}.$$
One finds $\PP_{\Delta_{\check X}}\cong\CC\PP^4/(\ZZ_5)^3$. As progress towards nailing the SYZ conjecture for the quintic, Mark Gross \cite[Theorem 4.4]{Gross01} gave a topological torus fibration on a space that is diffeomorphic to $X$ and Matessi and Casta\~no-Bernard \cite{CBM06} showed that this one can be upgraded to a piecewise smooth Lagrangian fibration for some symplectic structure and a similar approach works for $\check X$. 
Recently, Evans-Mauri \cite{EvansMauri} gave a Lagrangian fibration local model for parts of the fibration that are most difficult to deal with in dimension three. Whether this can be used for global compactifications of fibrations or tropical Lagrangian attachment problems presumably requires a similarly careful analysis as for the situations that we are going to consider. 
For recent progress on the SYZ topology, we refer to \cite{RZ1,RZ2,RuddatZharkov20}.

We are not working on the diffeomorphic model but on the actual symplectic quintic mirror $\check X$, in fact our construction applies to Lagrangians in each of the many possibilities resulting from different choices of a crepant resolution for the quintic mirror ($h^{1,1}(\check X)=101$). 
For our construction, it suffices to have a Lagrangian torus fibration \emph{locally} around the Lagrangian $L$ that we wish to construct from a tropical curve $\gamma$.
To say where $\gamma$ lives, we make use of the construction of the real affine base space of the torus fibration from \cite{Gross01}.

The Newton polytope of the quintic is the polar dual to $\Delta_{\check X}$, that is, the convex hull $\conv\{0,5e_1,...,5e_4\}$ translated by $(-1,-1,-1,-1)$ so that its unique interior lattice point $(1,1,1,1)$ becomes the origin. We call the resulting polytope $\Delta_X$. 
Choosing $a_m\in\RR_{\ge 0}$ for each $m\in \partial\Delta_X\cap\ZZ^4$ yields a cone in $\RR^4\oplus\RR$ generated by the set of $(m,a_m)$ and its 
boundary gives the graph
of a piecewise linear convex function $\varphi:\RR^4\ra \RR$.
We require that every face in the boundary is a simplicial cone and we assume that each $(m,a_m)$ generates a ray of this cone, in particular, $\varphi(m)=a_m$ for all $m$.
There are lots of $\varphi$ satisfying these properties and each one gives a toric projective crepant partial resolution
$$\operatorname{res}:\PP_{\Delta}\ra\PP_{\Delta_{\check X}}$$
where $\PP_{\Delta}$ is given by the fan in $\RR^4$ whose maximal cones are the maximal regions of linearity of $\varphi$. 
Equivalently, $\PP_{\Delta}$ is given by the polytope 
\begin{equation} \label{eq-Delta-from-am}
\Delta=\bigcap_{m\in\partial\Delta_X\cap\ZZ^4}\{n\in\RR^4| \langle n,m\rangle\ge -a_m\} = \{n\in\RR^4| \langle n,m\rangle\le \varphi(m) \hbox{ for all }m\in\RR^4\}
\end{equation}
\begin{lemma} \label{lema-isolated-orbi}
$\PP_{\Delta}$ has at worst isolated Gorenstein orbifold singularities.
\end{lemma} 
\begin{proof} 
If $\sigma$ is a maximal cone in the fan, i.e.~a maximal region of linearity of $\varphi$, then it is simplicial by assumption, hence generated by $m_1,...,m_4\in\partial\Delta_{X}$ say. 
Moreover, these generators are all contained in a single facet $F$ of $\Delta_X$ because the fan refines the normal fan of $\Delta_{\check X}$.
By the assumption that each $(m,\varphi(m))$ for $m\in\partial\Delta_{X}\cap \ZZ^4$ is a ray generator, we find $F\cap\sigma\cap \ZZ^4 = \{m_1,...,m_4\}$. 
So $\sigma$ is a cone over the elementary lattice simplex given by the convex hull of $m_1,...,m_4$, thus gives a 
terminal toric Gorenstein orbifold singularity and these have codimension four. 
\end{proof}
Let $w_j$ be the monomial associated to $e_j$. 
Consider the (singular) hypersurface in $\PP_{\check X}$ given by an anti-canonical section written as a Laurent polynomial on $(\CC^*)^4$,
$$h=\alpha_1w_1+\alpha_2w_2+\alpha_3w_3+\alpha_4w_4+\alpha_0+\alpha_5(w_1w_2w_3w_4)^{-1}$$
with $\alpha_j\in\CC^*$ and $-5\alpha_0^5\neq \prod_{i=1}^5\alpha_i$ (so that $h=0$ gives a submanifold of $(\CC^*)^4$, cf.~\cite[\S2]{Candelas}). 
The monomial exponents of $h$ are precisely the lattice points of $\Delta_{\check X}$.
The closure of $h=0$ in $\PP_\Delta$ misses the isolated orbifold points at zero-dimensional strata (Lemma~\ref{lema-isolated-orbi}) and gives a symplectic $6$-manifold $\check X$ 
with symplectic structure induced from $\PP_\Delta$.
Furthermore, $\check X$ is a Calabi-Yau manifold as it agrees with the crepant resolution under $\operatorname{res}$ of the anti-canonical hypersurface, the closure of $h=0$ inside $\PP_{\Delta_{\check X}}$.
By deforming the $a_m\in\RR_{\ge 0}$, one can study continuous deformations of the symplectic structure. 
The space of crepant symplectic resolutions acquires an interesting chamber structure with a point on a wall given by a set of $a_m$ that violates the simplicialness of $\varphi$. Just as a remark: the wall geometry is governed by the \emph{secondary polytope} of $\Delta_X$. 

\subsection{The real affine manifold and tropical curves}
\label{section-affine-charts}
Following \cite{MGross05}, we next explain how to give a real integral affine structure on a large open subset of $\partial\Delta$ where $\Delta$ is a lattice polytope obtained from a $\varphi$ as in Section \ref{section-quintic}. 
We split $\Delta$ as a Minkowski sum
$$\Delta=\Delta_{\check X}+\Delta'$$
where $\Delta'$ is the polytope associated to the piecewise linear function $\varphi'$ that takes value $\varphi(m)-1$ at $m\in\partial\Delta_X\cap\ZZ^4$. 
Indeed, this gives a decomposition as claimed because $\Delta_{\check X}$ is the polytope of the function $\varphi_{\check X}$ taking value $1$ on all of $\partial\Delta_X$, so $\varphi=\varphi_{\check X}+\varphi'$.
By the decomposition, every vertex $v$ of $\Delta$ is uniquely expressible as $v=v_{\check X}+v'$
for $v_{\check X}$ a vertex of $\Delta_{\check X}$ and $v'$ a vertex of $\Delta'$. 
We project a small neighborhood $W_v\subset \partial\Delta$ of $v$ onto the quotient of the affine four-space $[v+\RR^4]$ that contains $W_v$ by the affine line $[v+\RR v_{\check X}]$ resulting an affine three-space. The projection is thus injective and thereby gives a real affine chart for $W_v$. There is also an integral structure obtained by complementing $v_{\check X}$ to a lattice basis of $\RR^4$ to find a lattice for the quotient. We do this for each vertex $v$ of $\Delta$. 
Furthermore, for each facet $F$ of $\Delta$, its interior $\Int(F)$ carries a natural integral affine structure from the tangent space to the facet.
Combining the resulting charts $W_v$ with the interiors of facet $\Int(F)$ yields an atlas on the union of these charts for an integral affine structure, i.e.~transitions in $\GL_3(\ZZ)\ltimes \RR^3$. 
By choosing $W_v$ suitably, the complement of the union of charts can be made to be
$$\shA:=\pi_\Delta(\check X\cap \PP_\Delta^{[2]})=\pi_\Delta(\check X)\cap \partial\Delta^{[2]}$$
where $\PP_\Delta^{[2]}$ is the union of complex two-dimensional strata, $\partial\Delta^{[2]}$ the union of two-cells of $\Delta$ and
$\pi_\Delta:\PP_\Delta\ra\Delta$ is the moment map for the Hamiltonian $(S^1)^4$-action on $\PP_\Delta$.
We don't need $\Delta$ to be a lattice polytope for this construction. 
The affine structure is integral affine because $\Delta_{\check X}$ is a lattice polytope.
Let $\Lambda$ denote the local system of integral tangent vectors on $\partial\Delta\setminus\shA$ (we also used $T_\ZZ$ before).
\begin{definition} \label{def-trop-curve}
A \emph{tropical curve} in $(\partial\Delta,\shA)$ is a graph $\gamma$ (realized as a topological space) together with a continuous injection $h:\gamma\ra \partial\Delta$ such that
\begin{enumerate} 
\item a vertex of $\gamma$ is either univalent or trivalent,
\item $h(v)\in\shA \iff v$ is a univalent vertex of $\gamma$,
\item the image of the interior of an edge $e$ is a straight line segment in the affine structure of $\partial\Delta\setminus\shA$ of rational tangent direction,
\item For $v$ with $h(v)\in\shA$, the primitive tangent vector of the adjacent edge $e$ generates the image of $T_\nu-\id$ for $T_\nu$ the monodromy of $\Lambda$ along any non-trivial simple loop $\nu$ around $\shA$ in a small neighborhood of $v$.
\item For every trivalent vertex $v$, and $e_1,e_2,e_3\in \Lambda_v$ the primitive tangent vectors into the outgoing edges, we have $e_1+e_2+e_3=0$ and $e_1,e_2$ span a saturated sublattice of $\Lambda_v$.
\end{enumerate}
We consider two tropical curves $\Gamma_1,\Gamma_2$ the same if there exists a homeomorphism $\Gamma_1\ra\Gamma_2$ that commutes with $h_1,h_2$.
By slight abuse of notation, we also use $\gamma$ to refer to the image of $h$.
The following lemma guarantees that we can always satisfy (4) above as long as the tropical curve approaches $\shA$ from the right direction. The lemma directly follows from the aforementioned simplicity of $(\partial \Delta, \shA)$, c.f. \cite{Gross01}.
\begin{lemma} \label{lemma-simple}
Let $x\in \partial\Delta\setminus \shA$ be a point contained in a small neighborhood $V$ so that $(V,V \cap \shA)$ is homotopic to $(D,p)$
as a pair, where $D$ is an open disc and $p$ is a point in $D$. 
Let $\nu\in \pi_1(V\setminus \shA)=\pi_1(D\setminus p) \cong \ZZ$ be a generator and $T_\nu:\Lambda_x\ra\Lambda_x$ the monodromy of $\Lambda$ along $\nu$. 
In a suitable basis of $\Lambda_x\cong \ZZ^3$, 
 $T_{\nu}$ is given by $\left(\begin{smallmatrix} 1&0&0\\0&1&1\\0&0&1 \end{smallmatrix}\right)$.
In particular, the image of $T_\nu-\id$ is saturated of rank one, i.e.~generated by a primitive vector.
\end{lemma}
\end{definition}

\subsection{Katz's methods for finding lines on a quintic} \label{section-Katz-lines}
The quintic $X$ permits a flat degeneration to the union of coordinate hyperplanes simply by interpolation. If $X$ is given by the homogeneous quintic equation $f_5$ in the variables $u_0,...,u_4$ then we define the family of hypersurfaces in $\CC\PP^4$ varying with $t$ by
$$ u_0\cdot...\cdot u_4+tf_5=0$$
and denote by $X_0$ the fiber with $t=0$. 
Since $X_0$ is the union of five projective spaces, it contains infinitely many lines. 
However, only a finite number of them deforms to the nearby fibers, worked out by Katz in \cite{Katz86}. 
Assuming $f_5$ is general, the intersection of $X$ with each coordinate two-plane is a smooth complex quintic curve. There are ten of these.
\begin{theorem}[Katz]
A line in $X_0$ deforms into the nearby fiber if and only if it does not meet any coordinate line of $\PP^4$ but meets four of the $10$ quintic curves.
\end{theorem}
Note that it follows that a line which deforms needs to be contained in a unique irreducible component $H\cong\PP^3$ of $X_0$ and needs to meet the $4$ quintic curves that are contained in this component, namely the intersections of the four coordinate planes of $H$ with $X$.
A general quintic hosts $2875=5\cdot 575$ many lines and in the degeneration, each $H$ contains $575$ deformable lines, \cite{Katz86}.
On the dense algebraic torus $(\CC^*)^3$ of $\CC\PP^3$, we may apply the map $(\CC^*)^3\ra\RR^3$ given by $\log|\cdot|$ for each coordinate. Each line maps to an amoeba with four legs going off to infinity in the directions of the rays in the fan of the toric variety $\PP^3$.
Furthermore, these legs ``meet the amoeba of the quintic plane curves at infinity''. We are not going to make this more precise because we only use this idea as inspiration. There is a closely related theorem that was our main motivation combined with Katz's findings:
\begin{theorem}[\cite{MR16}]
The number of tropical lines in $\RR^3$ meeting $5$ general quintic tropical curves at tropical infinity each in one of the four directions of the rays of the fan of $\PP^3$ when counted with their tropical multiplicities agrees with the number of complex lines in $\PP^3$ meeting five general quintic plane curves.
\end{theorem}
So we may almost deduce from Katz's count of complex lines a count of tropical lines via this theorem. The only issue here is the attribute ``general''. Indeed, the quintic curves in Katz's situation are \emph{not} in general position. If they were, the count would be $2\cdot 5^4$ by standard Schubert calculus but this number is way bigger than $575$. Indeed, any pair of quintic curves meets each other in $5$ points which wouldn't happen if they were in general position. 
They meet each other because they arise from the same equation $f_5=0$ restricted to each coordinate plane.

We expect that in the more special position where the tropical quintics meet each other, after removing degenerate tropical lines (meaning those that move in positive-dimensional families, meet vertices of the discriminant curve or don't have the expected combinatorial type \mbox{$>\hspace{-6pt}-\hspace{-6pt}<$}), then one actually finds $575$ when counting these with multiplicity. 
We verify this below in a global example. 
Before going into its details, let us clarify why tropical lines in $\RR^3$ that meet tropical quintics at infinity relate to $(\partial\Delta,\shA)$ in the sense of Definition~\ref{def-trop-curve}. 
For $s\in[0,1]$, setting $\Delta_s=\Delta_{\check X}+s\Delta'$, we observe $\Delta_0=\Delta_{\check X}$ and $\Delta_1=\Delta$. 
In this sense, $\Delta$ is a deformation of $\Delta_{\check X}$ and note that $\Delta_s$ has the same combinatorial type for all $s>0$. 
Recall the notion of the discrete Legendre transform from \cite{GrossSiebert03,GrossSiebert06,Ruddat14}. 
Since $\Delta_X$ and $\Delta_{\check X}$ are polar duals, their boundaries are discrete Legendre dual, \cite[Example 1.18]{GrossSiebert06}. 
The subdivided boundary of $\Delta_X$ by means of $\varphi$ is the discrete Legendre dual to $\partial\Delta$.
For a $2$-cell $\tau$ in $\partial\Delta$, there are three possibilities for what its deformation $\bar\tau$ in $\Delta_0=\Delta_{\check X}$ can be, namely $0$-, $1$- or $2$-dimensional. 
These cases match with whether its dual (one-dimensional) face $\check\tau$ in the subdivision of $\partial\Delta_X$ lies in a $3$-, $2$- or $1$-cell of $\Delta_X$.
Most importantly, since the subdivision of $\Delta_X$ by $\varphi$ governs the composition $\PP_{\Delta}\stackrel{\operatorname{res}}\lra\PP_{\Delta_{\check X}}\ra\Delta_{\check X}$, the following holds.
\begin{lemma} \label{lem-admissible}
Let $\tau\subset\partial\Delta$ be a $2$-cell. 
Recall the monomials $w_1,...,w_4$. We set $w_5:=(w_1w_2w_3w_4)^{-1}$ and $i\in\bar\tau$ means that the vertex of $\Delta_{\check X}$ corresponding to $w_i$ is contained in $\bar\tau$.
The amoeba part $\shA\cap\tau$ is given by 
$$g_{\tau,0}:=\sum_{i\in\bar\tau} \alpha_iw_i$$
as an equation on the torus orbit dense in the stratum of $\PP_\Delta$ given by $\tau$. 

In particular, for $\tau$ deforming to an edge of $\Delta_{\check X}$, $g_{\tau,0}$ is a binomial. Also note that $\shA\cap\tau=\emptyset$ if $\bar\tau$ is a $0$-cell.
\end{lemma}

Note that $g_{\tau,0}$ is a binomial if and only if the corresponding amoeba is one-dimensional and hence tropical curves ending on it will be admissible (see also Remark \ref{r:justifyA6}).

If $\varphi$ is a unimodular subdivision, e.g.~as in Figure~\ref{graph:subdivision}, then most two-cells of $\partial\Delta$ have $\shA\cap\tau=\emptyset$, there are $5\cdot 10$ many two-cells of $\partial\Delta$ that deform to triangles in $\Delta_{\check X}$ but most interestingly for us, $10\cdot 30$ two-cells deform to edges, hence their amoeba is given by a binomial. 
These amoeba pieces arrange as $10$ plane quintic curves, e.g. as in Figure~\ref{fig-plane-quintic}. (One verifies that indeed the number of interior edges is $30$ here.) Each quintic curve is dual to the triangulation of a two-face of $\Delta_X$, e.g. consider the front face in the right hand part of Figure~\ref{graph:subdivision}. Each facet of $\Delta_X$ contains four triangle faces,
\begin{wrapfigure}[14]{r}{0.38\textwidth}
\captionsetup{width=.75\linewidth}
\begin{center}
 \includegraphics[width=0.35\textwidth]{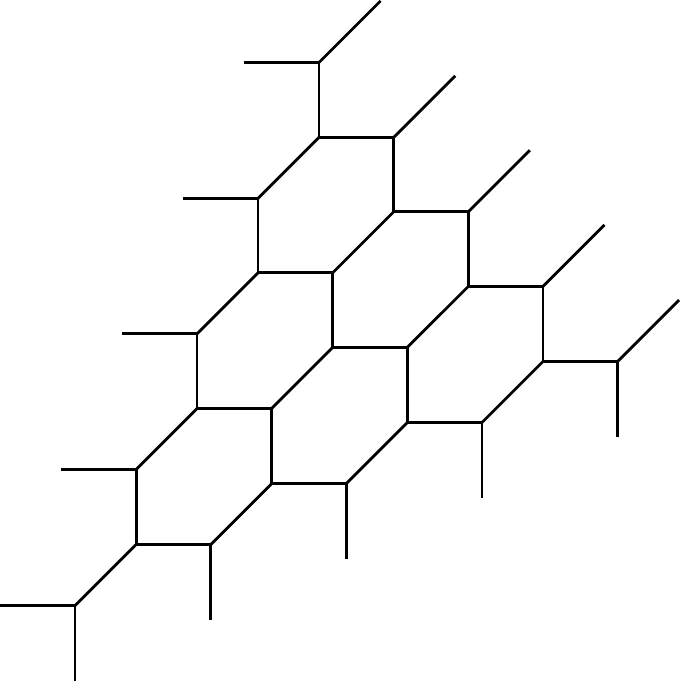}
\end{center}
 \caption{The tropical plane quintic curve that describes part of the discriminant $\shA$ of $\partial\Delta$.}
 \label{fig-plane-quintic}
\end{wrapfigure}
hence dually, four quintic curves arrange together as the boundary of a space tropical quintic surface in $\RR^3$. In particular, we can view them as lying at infinity and since they make up the discriminant in $\partial\Delta$, a tropical line in $\RR^3$ with ends on the four quintics thus gives a tropical curve in $\partial\Delta$. 
There are a lot of these, see Figure~\ref{graph:575lines} and most of them are admissible, i.e.~they meet one of the 30 inner edges of each quintic, rather than the $15$ outer ones. 
Also note that this configuration appears five times in the boundary of $\partial\Delta$.

\begin{figure}
  \centering
  \begin{tabular}{p{0.4\linewidth}cp{0.4\linewidth}}
 \includegraphics[width=0.4\textwidth]{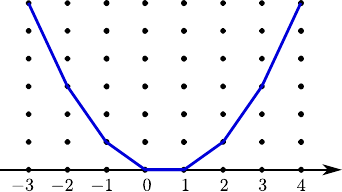}&\qquad &\includegraphics[width=0.4\textwidth]{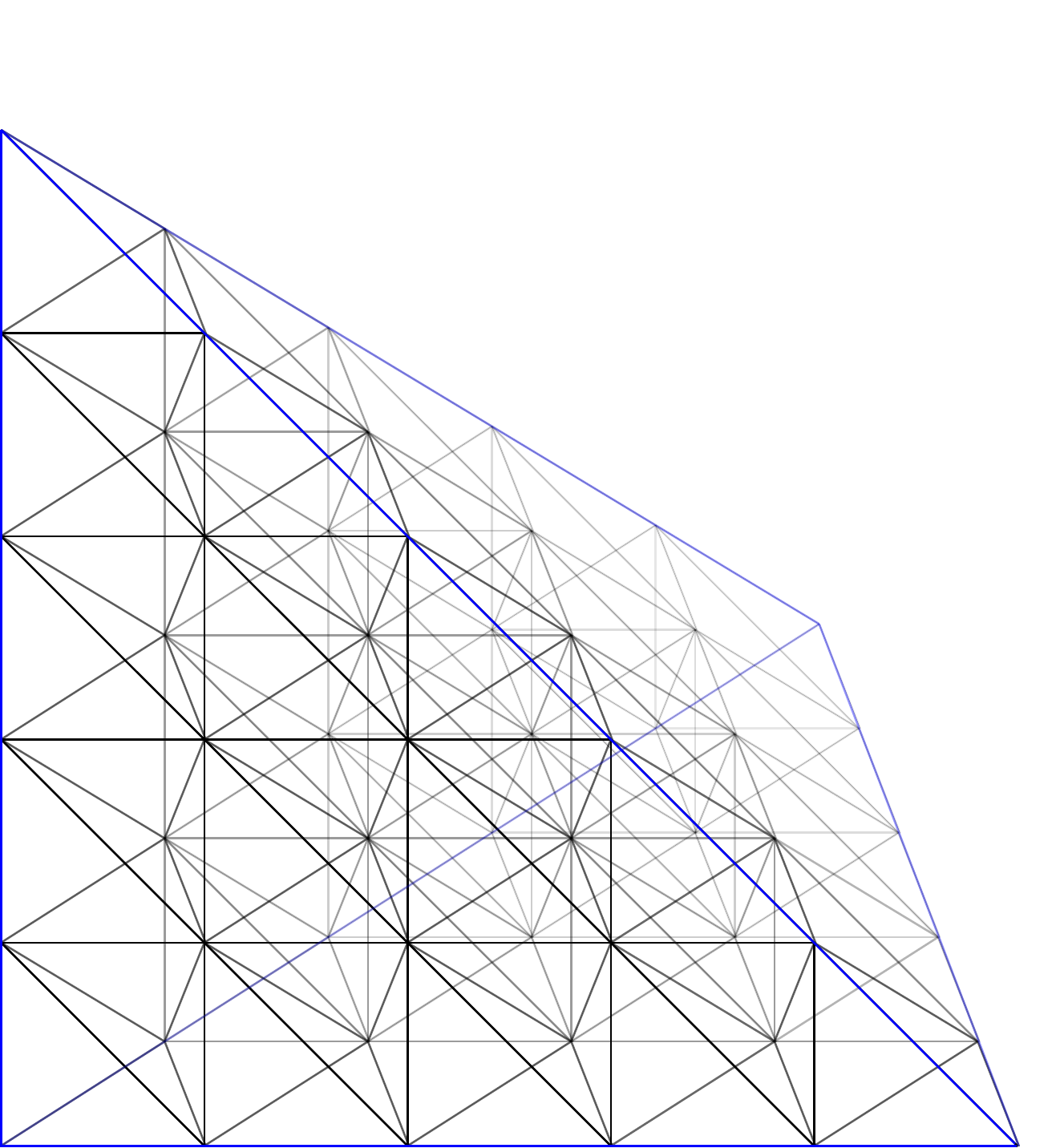}
  \end{tabular}
 \caption{Graph of $\varphi_0$ and Weyl-$A_3$ subdivision of a facet of the moment polytope of the quintic threefold (obtained from $\varphi_1$).}
 \label{graph:subdivision}
\end{figure}

\subsection{A very symmetric subdivision and resolution of the quintic mirror} \label{section-subdivision}

We next give an example for $\PP_\Delta$ that is even a manifold, see also \cite[p.\,122: Fig. 4.6]{Gross01}.
The subdivision of each facet of $\partial\Delta_X$ is obtained from the affine Weyl chambers of type $A_3$, cf. \cite[III,\S2]{KKMS}. 
Concretely, let $\varphi_0:\RR\ra\RR$ be the unique continuous convex function that is linear on each connected component of $\RR\setminus\ZZ$, changes slope by $1$ at each point in $\ZZ$ and is constantly zero on $[0,1]$, see Figure~\ref{graph:subdivision}. 
One finds $\varphi_0(n)=n(n-1)/2$ for $n\in\ZZ$ (``discrete parabola''). 
Now consider the piecewise affine function $\varphi_1:\RR^4\ra \RR$ given by
\begin{align*}
\varphi_1(x_1,x_2,x_3,x_4)= & \ \varphi_0(x_1)+\varphi_0(x_2)+\varphi_0(x_3)+\varphi_0(x_4)\\
&+\varphi_0(x_1+x_2)+\varphi_0(x_2+x_3)+\varphi_0(x_3+x_4)\\
&+\varphi_0(x_1+x_2+x_3)+\varphi_0(x_2+x_3+x_4)\\
&+\varphi_0(x_1+x_2+x_3+x_4).
\end{align*}
and finally define $\varphi$ as the unique piecewise linear function on $\RR^4$ that coincide with $\varphi_1$ on $\partial\Delta_X$. 
For $m\in\partial\Delta_X\cap\ZZ^4$, set $a_m=\varphi(m)$ and recall from Section \ref{section-quintic} that $\varphi$ is entirely determined from the set of $a_m$.

The induced subdivisions of any two facets are isomorphic and looks like what is on the right in Figure~\ref{graph:subdivision}.
One checks that each four-dimensional cone in the fan given by $\varphi$ is lattice-isomorphic to the standard cone $\RR_{\ge 0}^4\subset\RR^4$, so the resulting $\PP_\Delta$ is smooth.
In our explicit example below, we will use a slight perturbation replacing $a_m$ by $a_m+\eps_m$ for random $\eps_m$ to increase our chance of being in a generic situation. 
Plugging the perturbed $a_m$ into \eqref{eq-Delta-from-am} yields a slightly deformed $\Delta$ and while the complex manifold $\PP_\Delta$ doesn't change, as this slightly perturbs the symplectic form in a well-understood manner.
The best way to understand what $(\partial\Delta,\shA)$ looks like is by considering its discrete Legendre dual. The subdivision of $\Delta_X$ by $\varphi$ is five copies of the right hand side of Figure~\ref{graph:subdivision} glued along facets.
Therefore, after identification there are ten $2$-faces, each carrying a subdivision that is dual to that of a quintic curve in its most symmetric form show in Figure~\ref{fig-plane-quintic}.

\subsection{The findings of a computer search for the tropical lines} \label{section-computer}
As described in the previous section, we obtained a particular $\partial\Delta$ as a small perturbation of $\varphi$ that gave the very symmetric subdivision of $\Delta_X$. 
From Katz's work as described in Section \ref{section-Katz-lines}, we are looking for tropical lines in $\partial\Delta$ that meet the quadruple of tropical quintic curves where each tropical quintic is dual to the subdivision of one of the $10$ triangle faces of $\Delta_X$. 
We used a computer for this search following the pseudo code\footnote{For more details, the complete code with instructions and results, see\\ \url{https://www.staff.uni-mainz.de/ruddat/lines-in-quintic/lines.html} or look at the ancillary files of this third arxiv version}.\vspace{2mm}
\begin{center}
\RestyleAlgo{boxruled}
\begin{algorithm}[H]
%\setstretch{1.8}
%\SetAlgoLined
 \KwData{Unimodular regular subdivision of the convex hull $\Delta_X$ of $0,5e_1,...,5e_4$ in $\RR^4$ by piecewise affine height function that is integral at integral points.}
 \KwResult{Findings of all tropical lines.}
  compute the tropical quintic threefold inside $\RR^4$ associated to the height function\;
  compute the $5$ tropical quintic surfaces $S_i$ ($i=1,...,5$) at the five asymptotic directions of infinity (each sitting inside an $\RR^3$)\;
  compute for each of these quintic surfaces $S_i$ the quintic curves $C_{ij}$ ($j=1,...,4$) at the respective four directions of infinity. Each $C_{ij}$ has $45$ edges\;
 \For{$i\leftarrow 1,...,5$}{
  \For{each of the $45^4$ tuples $(a_1,...,a_4)$ with $a_j$ an edge of $C_{ij}$}{
   \For{each of the $3$ generic combinatorial types of a tropical line in $\RR^3$}{\vspace{1mm}
     Check if there exists a tropical line of the given type meeting $a_1,...,a_4$ and if so, record it.
   }
  }
 }
 From the recorded tropical lines, remove all those that are non-rigid (i.e.~they are part of a positive-dimensional family) or are special (they meet vertices of the tropical quintics $C_{ij}$). The remaining ones are the result of the search.
\caption{Tropical line search}
\end{algorithm}
\end{center}
\vspace{2mm}
After removing all lines that meet vertices of the quintics, that have only one internal vertex or are non-rigid (that is move in families), we did actually get the expected count --- when counting with multiplicity (which is remarkable in view of \cite{Vigeland,PaniVige,CuetoDeo}). That is, maybe surprisingly, the lines weren't all of multiplicity one. We give the definition of the multiplicity in Section \ref{ss:MultiplicityWeight}.
For each of the five facets of $\Delta_X$, the count with multiplicity of the tropical curves gave indeed $575$, so in total $2875$ as expected. We found $2695$ curves of multiplicity one and $90$ of multiplicity two. These $90$ did not evenly distribute over the $5$ facets: $15+16+18+20+21$. While $90$ is a number that hasn't appeared yet in the context of the quintic to our knowledge, one may speculate that relates to the count of real lines that was found to be $15$ in \cite{Solomon}: for rational curves on an elliptic surface, the presence of higher multiplicity tropical curves is implied from the Welschinger invariant to differ from the Gromov-Witten invariant, see e.g. \cite[\S4.2.2]{GOR15}.

The goal is to construct Lagrangian threefolds from these tropical curves. The remainder of this article carries this out for admissible curves. Recall that the requirement is that the tropical curve meets the discriminant amoeba $\shA$ in points where this amoeba is one-dimensional. By Lemma~\ref{lem-admissible}, this holds true if the tropical line meets the internal edges of the quintic curves, i.e.~no outer edges. A bit more than half the curves feature this: we get $1451$ admissible lines out of which $45$ have multiplicity two (multiplicity weighted account is $1496$). 
Interestingly, the admissible curves don't meet curves of other facets (unlike non-admissible ones), though possibly still other curves in their own facet. We found a set of $354$ admissible lines that are pairwise disjoint out of which $42$ have multiplicity two.
%\begin{figure}
%  \centering
%  \begin{tabular}{p{0.4\linewidth}cp{0.4\linewidth}}
% \includegraphics[width=0.4\textwidth]{AllAM.png}&\qquad &\includegraphics[width=0.4\textwidth]{AM.png}
%  \end{tabular}
% \caption{The upper triangular part of the adjacency matrix of the $2785$ tropical lines (left) and of the subset of $1451$ admissible lines (right).}
% \label{fig-incidence-matrices}
%\end{figure}

\begin{figure}
  \includegraphics[width=0.8\textwidth]{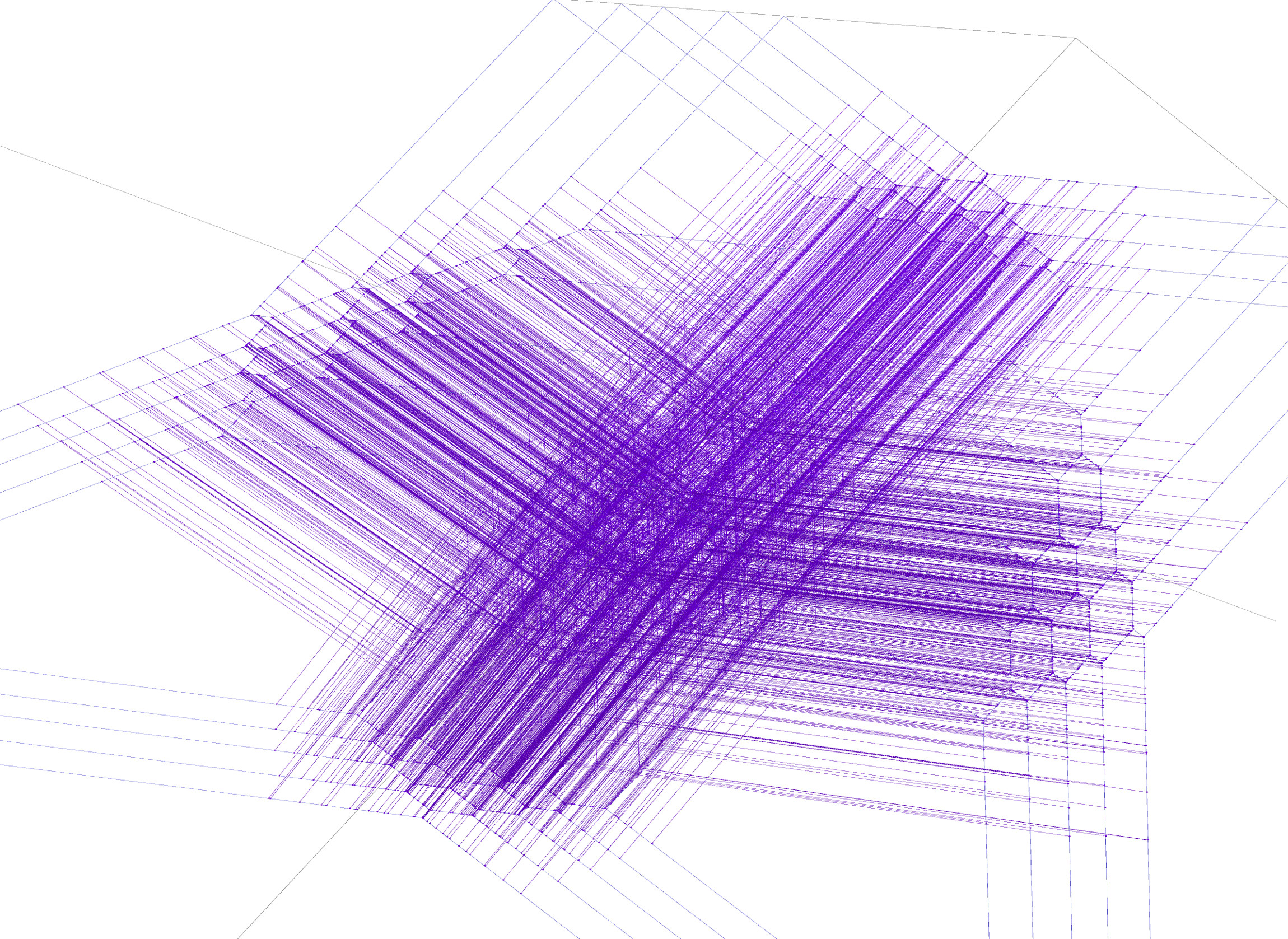}
 \caption{Tropical lines in $\RR^3$ meeting $4$ tropical quintics at infinity, of total multiplicity $575$.}
 \label{graph:575lines}
\end{figure}

\subsection{Lagrangian lift of a tropical curve}\label{ss:LagrangianLift}
In this section, we give the definition of the diffeomorphism type of a Lagrangian lift of a tropical curve $\gamma$ in $(\partial \Delta, \shA)$ to the Calabi-Yau given by $\partial \Delta$, e.g. the mirror quintic as before.
Using the integral affine structure on $\partial \Delta \setminus \shA$, we can define a Lagrangian torus bundle by
\begin{align}
\check X^\circ:=T^*(\partial \Delta \setminus \shA)/T^*_{\mathbb{Z}}(\partial \Delta \setminus \shA).
\end{align} 
Recall the notation $\Lambda=T_{\mathbb{Z}}(\partial \Delta \setminus \shA)$ and note that $\partial\Delta$ is an orientable topological manifold and so $\bigwedge^3\Lambda\cong\ZZ$. Fixing an orientation once and for all, we can talk about oriented bases of stalks of $\Lambda$.

For each edge $e$ of $\gamma$ and a point $x$ in the interior $e^\circ$ of $e$, we get the $2$-dimensional subspace $e^{\perp}$ of $T^*_x(\partial \Delta \setminus \shA)$ consisting of co-vectors that are perpendicular to the direction of $e$.
By Definition \ref{def-trop-curve}\,(3), every translation $e^{\perp}+a$ descends to an embedded $2$-torus in $\check X^\circ$.
A smooth family of these 2-tori over $x \in e^\circ$ defines a (trivial) torus bundle $L_{e^\circ}$ over $e^\circ$ 
and the total space $L_{e^\circ}$ is a Lagrangian submanifold in $\check X^\circ$. It extends over the vertices of $e$ that don't lie in $\shA$, and we let $L_{e}$ denote the extension.

\begin{remark}\label{r:absorbAlpha}
 Let $f: e^\circ \to \mathbb{R}$ be a smooth function such that it descends to a compactly supported function $f':e^\circ \to \mathbb{R}/2\pi\mathbb{Z}$.
 Given a smooth family of 2-tori over $x \in e^\circ$ as above, we can define a new family by fiberwise translating the 2-tori by $f(x)$.
 The resulting Lagrangian is a different embedding of a 2-torus times interval to  $\check X^\circ$.
 The function $f'$ being compactly supported corresponds to that the two embeddings coincides near the ends of $e^\circ$.  
\end{remark}

\begin{wrapfigure}[9]{r}{0.26\textwidth}
\begin{center}\vspace{-.4cm}
  \includegraphics[width=0.2\textwidth]{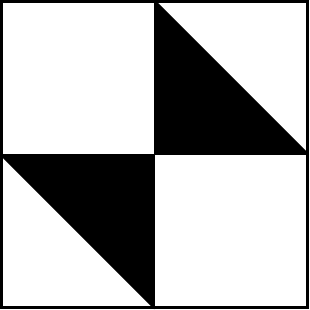}
 \captionsetup{width=.85\linewidth}
 \caption{A pair of pants.}
 \label{graph:angularPOP}
\end{center}
\end{wrapfigure}
For each trivalent vertex $v$ of $\gamma$, by Definition \ref{def-trop-curve}\,(5), we can identify the primitive tangent vector of the outgoing edges as
$e_1=(1,0,0)$, $e_2=(0,1,0)$ and $e_3=(-1,-1,0)$ with respect to a $\mathbb{Z}$-basis of $\Lambda_v\cong \ZZ^3$.
Let $\tilde{L}_v$ be the subset of $T_v \partial \Delta/(T_v \partial \Delta)_{\mathbb{Z}}=(\mathbb{R}/2\pi\mathbb{Z})^3$ consisting of all the points $(q_1,q_2,q_3)$ such that 
\begin{equation}
\begin{aligned}
\{ q_1,q_2\ge 0 \text{ and }& q_1+q_2 \le \pi \} \text{ or }\qquad\qquad\qquad\qquad\qquad\\ 
\{ q_1,q_2\le 0 \text{ and }& q_1+q_2 \ge -\pi\}.
\end{aligned}\label{eq:2triangles}
\end{equation}
Equipping $\tilde{L}_v$ with the subspace topology yields a finite CW complex of the same homotopy type as a pairs of pants times a circle. More explicitly, $\tilde{L}_v$ has a trivial circle factor given by the $q_3$-coordinate, and \eqref{eq:2triangles} defines two triangles in the $q_1,q_2$-coordinates and the vertices of the triangles are glued at $(q_1,q_2)=(0,0), (0,\pi), (\pi,0)$ respectively (see Figure \ref{graph:angularPOP}).

If we equip the two triangles in  \eqref{eq:2triangles} with opposite orientations, then the boundary of them is exactly given by the circles $C_1:=\{q_1=0\}$, $C_2:=\{q_2=0\}$ and $C_3:=\{q_1+q_2=\pi\}$.
For $i=1,2,3$, the product of $C_i$ with the circle in $q_3$-coordinate is exactly $e_i^{\perp}+a_i \subset T^*_v(\partial \Delta \setminus \shA)$ where $a_1=a_2=0$ and $a_3=\pi$.
For an appropriate choice of orientations, one can see that the boundary of $\tilde{L}_v$ cancels the boundary of $\cup_{i=1}^3 L_{e_i}$ lying above $v$, yet $\tilde{L}_v \cup \bigcup_{i=1}^3 L_{e_i}$ is only a Lagrangian cell complex instead of a manifold.
In Section \ref{ss:trivalent}, we explain how to replace the union of the triangles by a pairs of pants 
and obtain a Lagrangian pair of pants times circle $L_v$ that can be glued with $\cup_{i=1}^3 L_{e_i}$ smoothly.

Every univalent vertex $v$ of $\gamma$ lies in $\shA$ by Definition \ref{def-trop-curve}(2).
Let $\nu$ and $T_{\nu}$ be as in  Definition \ref{def-trop-curve}(4), so, by Lemma~\ref{lemma-simple}, $T_{\nu}=\left(\begin{smallmatrix} 1&0&0\\0&1&1\\0&0&1 \end{smallmatrix}\right)$ for a suitable basis.
The primitive direction of the edge $e$ adjacent to $v$ is by assumption given by $\pm (0,1,0)$ so the 2-tori in $L_e$ lying above $e$
are generated by $\partial_{q_1},\partial_{q_3}$. We can glue a solid torus $L_v$ to the toroidal boundary component of $L_e$ lying above $v$ to cap off this boundary component. Moreover, we require that the circle generated by $\partial_{q_3}$ is a meridian of  $L_v$. It is useful to observe that $\partial_{q_3}$ is characterized by being perpendicular 
to the invariant plane $\ker(T_{\nu}-\id)$.

\begin{definition}\label{d:LagLift}
The \emph{diffeomorphism type of a Lagrangian lift} of a tropical curve $\gamma$ is the diffeomorphism type of the closed $3$-manifold obtained by gluing $L_v$ and $L_e$ as above over all vertices $v$ and edges $e$ of $\gamma$.
\end{definition}

%\newpage
\subsection{Lagrangian weight versus tropical multiplicity}\label{ss:MultiplicityWeight}\label{section-weight-mult}
\begin{wrapfigure}[11]{r}{0.26\textwidth}
\begin{center}\vspace{-.3cm}
  \includegraphics[width=0.2\textwidth]{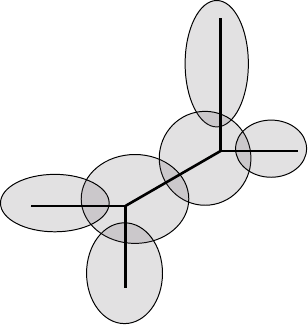}
 \captionsetup{width=.85\linewidth}
 \caption{\small An admissible \v{C}ech covering of a tropical curve}
 \label{graph:CechCover}
\end{center}
\end{wrapfigure}
Following Joyce, we define the weight of a Lagrangian rational homology sphere $L$ to be $w(L):=|H_1(L,\mathbb{Z})|$ and more generally $w(L):=|H_1(L,\mathbb{Z})_{\operatorname{tor}}|$.
Let $\gamma$ be a tropical curve in $(\partial \Delta, \shA)$.
In this subsection, we explain how $w(L_{\gamma})$ of a tropical Lagrangian $L_\gamma$ can be computed by a \v{C}ech covering of the corresponding tropical curve $\gamma$.
Since our Lagrangian $L_{\gamma}$ is homotopic to the Lagrangian cell complex $\tilde{L}_{\gamma}$ that is built by $\tilde{L}_{v}$ instead of $L_v$ at the trivalent vertices $v$ (see Section \ref{ss:LagrangianLift}),
it suffices to compute the first homology of $\tilde{L}_{\gamma}$.
For simplicity, we denote  $\tilde{L}_{\gamma}$ by $L_\gamma$ in this subsection. 
The universal coefficient theorem gives $(H_1(L_\gamma,\ZZ))_{\operatorname{tor}}=(H^2(L_\gamma,\ZZ))_{\operatorname{tor}}$, so we may compute $w(L_\gamma)$ via \v{C}ech cohomology.

A collection $\{U_{j}\}_{j=1}^m$ of open sets in $\partial \Delta$ that covers $\gamma$ is called {\it admissible} if
\begin{enumerate}
 \item $U_{j_1} \cap U_{j_2}  \cap U_{j_3}=\emptyset$ whenever $j_1,j_2,j_3$ are pairwise distinct, 
 \item for all $j$, $\gamma \cap U_j $ is connected and it contains exactly one vertex of $\gamma$ which is, by definition, either trivalent or univalent, and
 \item for $j_1 \neq j_2$, $\gamma \cap U_{j_1} \cap U_{j_2} $ (which may be empty) contains no vertex.
\end{enumerate}
For $\{U_{j}\}_{j=1}^m$ admissible, $H_i(U_j\cap\gamma,\ZZ)$ and $H_i(U_{j_1}\cap U_{j_2}\cap\gamma,\ZZ)$ are torsion free for all $i,j,j_1,j_2$ and therefore 
$$
w(L_\gamma)=\left|\coker\left(\bigoplus_j H^1(\pi^{-1}_{\Delta}(U_j)\cap L_\gamma,\ZZ)\stackrel{\Phi_\gamma}{\ra} \bigoplus_{i<j} H^1(\pi^{-1}_{\Delta}(U_i\cap U_j)\cap L_\gamma,\ZZ)\right)_{\operatorname{tor}}\right|
$$
where the map $\Phi_\gamma$ is the \v{C}ech map on $H^1$ (restriction with sign).
Recall from \cite{NS06,MR16,MR19} the definition of multiplicity $\mult(\gamma)$ of a tropical curve $\gamma$. 
Applicable for us is \cite[Equation (13)]{MR16} since we need to consider tropical curves with constraints on unbounded edges (i.e.~univalent vertices for us).
Let $\gamma^\circ$ be the interior of $\gamma$ and assume that we can trivialize $\Lambda$ on $\gamma^\circ$, i.e.~set $N:=\Gamma(\gamma^\circ,\Lambda)$ and $N\cong\ZZ^3$.
Furthermore, each univalent vertex $v$ of $\gamma$ gives a saturated rank two subspace $A_v$ in $N$ as the kernel of $T_\nu-\id$ near $v$. We view this as a constraint for the tropical curve $\gamma$ in $N_\RR$ in the sense of \cite{MR16}.
Given these constraints, \cite[Equation (13)]{MR16} provides a map of lattices $\Phi$ whose cokernel torsion gives the \emph{tropical multiplicity} $\mult(\gamma)$ of $\gamma$.
\begin{proposition} \label{p:multiplicity}
The \v{C}ech map $\Phi_\gamma$ is quasi-isomorphic to the tropical multiplicity computing map $\Phi$ from \cite[Equation (13)]{MR16} and thus $w(L_\gamma)=\mult(\gamma)$.
\end{proposition}
\begin{proof} 
The assertion follows if one shows that there is a natural isomorphism
$$H^1(\pi^{-1}_{\Delta}(U_i\cap U_j)\cap L_\gamma,\ZZ)\cong N/\ZZ v$$
whenever $U_i\cap U_j\neq\emptyset$ and $v$ is the primitive generator of the edge of $\gamma$ that meets $U_i\cap U_j$ and an isomorphism
$$H^1(\pi^{-1}_{\Delta}(U_i)\cap L_\gamma,\ZZ)\cong N$$
whenever $U_i$ contains a trivalent vertex and an isomorphism
$$H^1(\pi^{-1}_{\Delta}(U_i)\cap L_\gamma,\ZZ)\cong w^\perp/\ZZ v$$
whenever $U_i$ contains a univalent vertex of $\gamma$, $w^\perp=\ker(T_\nu-\id)$ and $v$ the primitive generator of the image of $T_\nu-\id$. 
Furthermore the restriction maps $H^1(\pi^{-1}_{\Delta}(U_j)\cap L_\gamma,\ZZ){\ra} H^1(\pi^{-1}_{\Delta}(U_i\cap U_j)\cap L_\gamma,\ZZ)$ are supposed to be the natural maps under these isomorphisms.
The isomorphisms and naturality of restriction maps are straightforward to be checked from the local descriptions of $L_v$ and $L_e$ given in Section \ref{ss:LagrangianLift}.
\end{proof}

\begin{remark}\label{r:higherMult}
 Proposition \ref{p:multiplicity} can be generalized to all dimensions for all tropical curves $\gamma$ 
 satisfying exactly the same set of conditions in Definition \ref{def-trop-curve}.
 The main reason is that, in higher dimensions,  
 $\pi^{-1}_{\Delta}(U_i\cap U_j)\cap L_\gamma$ and $\pi^{-1}_{\Delta}(U_i)\cap L_\gamma$ split as a product and there is a trivial factor 
 accounting for the extra dimensions.
Moreover, the universal coefficient theorem gives $(H_1(L_\gamma,\ZZ))_{\operatorname{tor}}=(H^2(L_\gamma,\ZZ))_{\operatorname{tor}}$ no matter what the dimension is
so the same \v{C}ech cohomology calculation applies to conclude that  $w(L_\gamma):=|(H_1(L_\gamma,\ZZ))_{\operatorname{tor}}|=\mult(\gamma)$.
\end{remark}

\subsection{Homology class of the Lagrangians}
\begin{figure}
  \centering
 \includegraphics[width=0.4\textwidth]{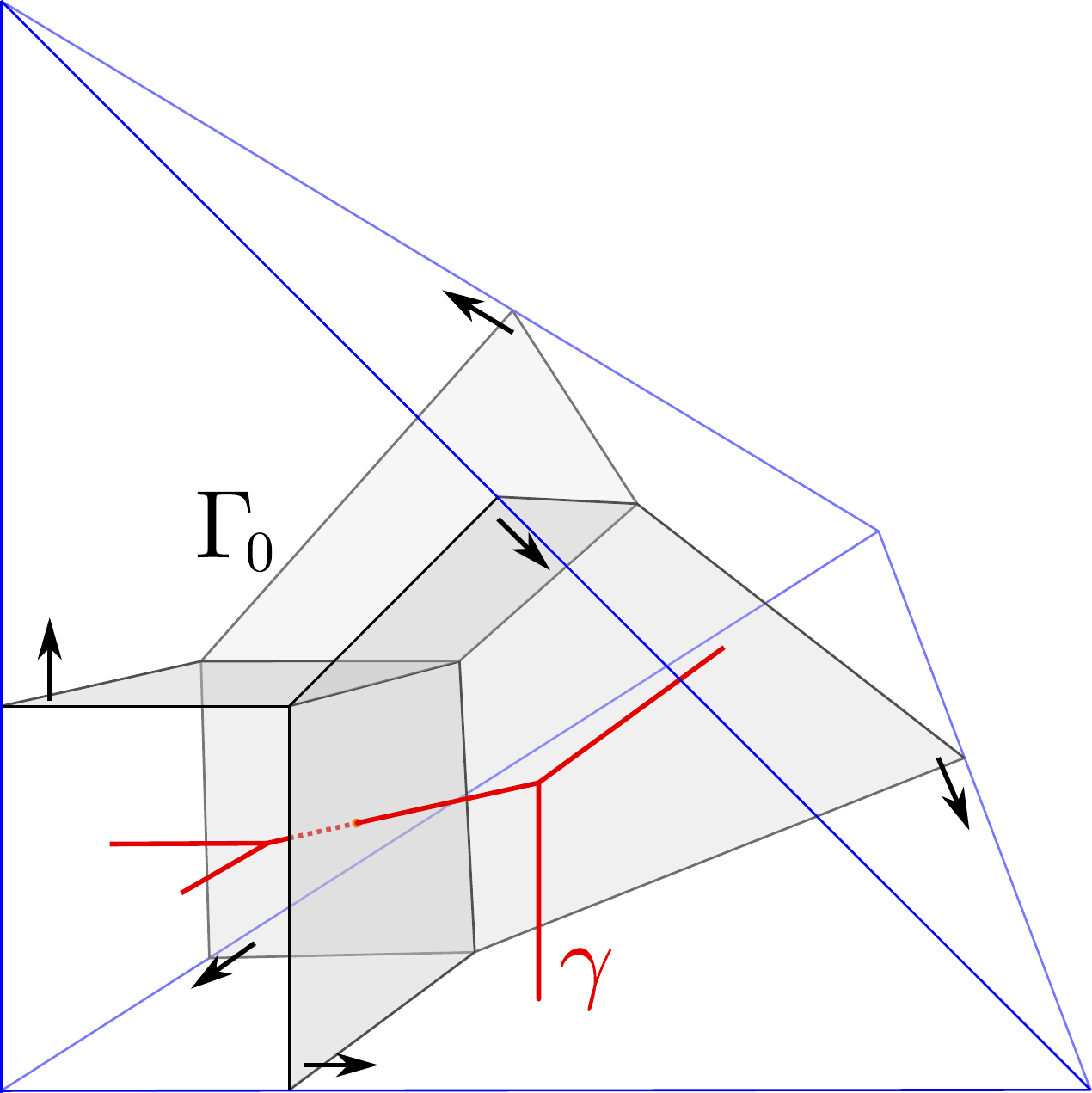}
 \caption{Intersection of a tropical line $\gamma$ and a tropical 2-cycle $\Gamma$ in $\partial\Delta_X$ where $\Gamma$ consists of five copies of the depicted cycle $\Gamma_0$, one copy for each facet of $\partial\Delta_X$. The line $\gamma$ however is contained in a unique facet of $\partial\Delta_X$.}
 \label{graph:tropicaltwocycle}
\end{figure}

Recall from Section~4 in \cite{Ruddat19} that a tropical 2-cycle in an affine manifold $B$ with singularities $\shA$ is simply a sheaf homology cycle representing a class in $H_2(B,\iota_*\bigwedge^2\Lambda)$ for $\iota:B\setminus\shA\ra B$ the inclusion of the regular part. 
Moreover, by $(0.6)$ in \cite{Ruddat19}, there is a homomorphism $r_{2}:H_2(B,\iota_*\bigwedge^2\Lambda)\ra H_3(\check X,\ZZ)/W_2$ with $W_2=\im(r_1)+\im(r_0)$ for similar maps $r_0,r_1$, c.f.~\cite{RZ1,RZ2}. 
For $\Gamma\in H_2(B,\iota_*\bigwedge^2\Lambda)$, we simply refer to any lift of $r_2(\Gamma)$ from $H_3(\check X,\ZZ)/W_2$ to $H_3(\check X,\ZZ)$ by $L_\Gamma$.

\begin{lemma} \label{lema-LGamma}
There is a tropical 2-cycle $\Gamma\subset \partial\Delta$ whose associated 3-cycle $L_\Gamma$ inside $M_t$ has intersection number $\pm1$ with each Lagrangian $L_\gamma$ constructed from a tropical line $\gamma$. Changing the orientation of $L_\gamma$ if needed, we can thus assume this intersection number is $+1$.
\end{lemma}
\begin{proof} 
Recall from Section~\ref{section-Katz-lines} that the subdivided boundary of $\Delta_X$, call it $\check B$, is discrete Legendre dual to $B:=\partial\Delta$.
In particular, $\check B$ and $B$ are homeomorphic with dual linear parts of their affine structures. 
This means the homeomorphism $\frak{D}:B\ra\check B$ identifies the local system $\check\Lambda_{\check B}$ of integral tangent vectors on $\check B\setminus \frak{D}(\shA)$ with the similar local system $\Lambda$ on $B\setminus\shA$. 
We remark that $\frak{D}(\shA)$ is contained in a neighbourhood of the union of 2-faces of $\partial\Delta_X$.
Making use of $\frak{D}$, in order to produce the desired tropical 2-cycle $\Gamma$, it therefore suffices to give a cycle for $H_2(\check B,\check\iota_*\bigwedge^2\check\Lambda_{\check B})$ where $\check\iota$ is the inclusion $\check B\setminus \frak{D}(\shA)\hra \check B$.
Since $\check B$ is orientable, we have an isomorphism $\check\iota_*\bigwedge^2\check\Lambda_{\check B}\cong \check\iota_*\Lambda_{\check B}$ for $\Lambda_{\check B}$ the dual of $\check\Lambda_{\check B}$.
We may thus give a suitable cycle $\Gamma$ representing a class in $H_2(\check B,\check\iota_*\Lambda_{\check B})$ in order to prove the lemma.

Recall that $\partial\Delta_X$ consists of $5$ tetrahedra.
Figure~\ref{graph:tropicaltwocycle} shows one such tetrahedron containing a union $\Gamma_0$ of six polyhedral disks.
The configuration $\Gamma_0$ can be described as a homeomorphic version of the compactification of the union of two-dimensional cones in the fan of $\PP^3$. 
Each of the five facets of $\Delta_X$ contains such a configuration $\Gamma_0$ and we may move the five copies of $\Gamma_0$ so that they fit together to a cycle $\Gamma$. 
That is, $\Gamma$ is actually a union of only $10$ disks, each of which is glued from $3$ disks that stem from different copies of $\Gamma_0$. 
The $10$ disks of $\Gamma$ are naturally in bijection with the edges of $\partial\Delta_X$, indeed we simply match a disk with the edge that it meets (transversely).

As the next step, we need to attach a section in $\Gamma(D,\check\iota_*\Lambda_{\check B})$ to each of the 10 disks $D$ so that the 10 sections satisfy the cycle-condition at the 1-cells where disks meet (three at a time).
We make use of the fact that the tangent space to a cell of the polyhedral decomposition of $\Delta_X$ is always monodromy-invariant for all monodromy transformations along loops
in $U\setminus (\frak{D}(\shA)\cap U)$ for $U$ a neighbourhood of the interior of the cell. In the case of a pair $(D,e)$ of a disk $D$ of $\Gamma$ and the corresponding transverse edge $e$ of $\Delta_X$, we may choose a primitive generator $v_D$ of the tangent direction to $e$ as the section of $\Gamma(D,\check\iota_*\Lambda_{\check B})$ that we associate with $D$. 
Making use of the existence of an orientation of $\check B$, the sign of $v_D$ and orientations of $D$ can be chosen so that the cocycle condition on $\{v_D\}_D$ is satisfied and we have thus produced a valid cycle $\Gamma$ as desired. 

It remains to show that $L_\Gamma$ satisfies the claimed intersection-theoretic property. For this purpose, we take the image of $\gamma$ along $\frak{D}$ and view $\gamma$ as a cycle in $H_1(\check B,\check\iota_*\check\Lambda)$
Theorem~7 in \cite{Ruddat19} says that the intersection number $L_\gamma.L_\Gamma$ agrees with the tropical intersection of $\gamma$ and $\Gamma$.
The tropical intersection number in turn is defined in item (3) of Theorem~6 in \cite{Ruddat19}. 
Note that $\gamma$ and $\Gamma$ have a unique point of physical intersection. 
We are left with verifying that the sections carried by $\gamma$ and $\Gamma$ at this point respectively pair to $\pm 1$.
The sections of $\gamma$ carried by the outer legs are precisely generators for the perp space of the 2-cells of $\Delta_X$ that they meet. The balancing condition then implies what the section at the central edge of $\gamma$ is. With this information and the knowledge that a disk $D$ of $\Gamma$ carries the section $v_D$ that is a generator for the tangent space to the edge of $\Delta_X$ that is met by $D$, it is easy to see from Figure~\ref{graph:tropicaltwocycle}
that the tropical intersection of $\gamma$ and $\Gamma$ is indeed $\pm1$ (and invariance of the intersection number under deforming the cycles being given by Theorem~6 in \cite{Ruddat19}).
\end{proof}

For a fixed tropical line $\gamma$, there are more than one $L_\gamma$ that can be constructed from Theorem \ref{t:Construction} due to the freedom of choices in the construction.
In particular, for each $L_\gamma$ and any integer $a$, we can construct another Lagrangian $(L_{\gamma})'$ by Theorem \ref{t:Construction}
such that the difference of their homology classes $[(L_{\gamma})']-[L_\gamma]$ is $a$ times the torus fiber class.
Using this freedom, we can prove the following.

\begin{proposition} \label{prop-homologous}
If $\gamma,\gamma'$ are two disjoint tropical lines, Lagrangians $L_\gamma, L_{\gamma'}$ can be constructed via Theorem \ref{t:Construction} so that they are homologous.
\end{proposition}
\begin{proof} 
We use the well-known fact that the vanishing cycle $\alpha\cong T^3$ of the quintic mirror degeneration is a primitive non-trivial homology class (it generates $W_0\cong\ZZ$ of the monodromy weight filtration) with $\alpha.\alpha=0$.
Using the cycle $L_\Gamma$ from Lemma~\ref{lema-LGamma}, we find the following intersection numbers
\begin{equation} \label{eq-int-numbers}
L_\Gamma.L_\gamma=1,\qquad L_\Gamma.\alpha=0,\qquad L_\gamma.\alpha=0,\qquad L_{\gamma}.L_{\gamma}=0
\end{equation}
where the middle ones follow from the fact that $\alpha$ can be supported in the complement of $L_\Gamma$ and $L_\gamma$ and the last one follows since the intersection pairing is anti-symmetric on $H_3(M_t)$.
We have equations \eqref{eq-int-numbers} similarly for $L_{\gamma'}$ in place of $L_\gamma$. 
Since the middle cohomology of the mirror quintic has rank four, we can complement $[\alpha],[L_\gamma],[L_\Gamma]$ to a basis of $H_3(M_t,\ZZ)$ by adding a fourth cycle $S$. 
Moreover since the restriction of the intersection pairing to the span of $L_\gamma,L_\Gamma$ is $\left(\begin{smallmatrix} 0&1\\-1&0\end{smallmatrix}\right)$, we can require $S$ to be in its orthogonal complement.
We write
$[L_{\gamma'}]=a[\alpha]+b[L_\gamma]+c[L_\Gamma]+dS$ and want to determine the coefficients $a,b,c,d$. 
From the analogue of \eqref{eq-int-numbers} for $L_{\gamma'}$, we find $d=0$ by pairing $L_{\gamma'}$ with $\alpha$ since necessarily $\alpha.S\neq 0$ for $[\alpha]$ being non-zero. 
Since $\gamma,\gamma'$ don't meet, $L_\gamma.L_{\gamma'}=0$ which yields $c=0$. 
Consequently, $1=L_\Gamma.L_{\gamma'}=b$
and hence $[L_{\gamma'}]=a[\alpha]+[L_\gamma]$ for some $a$.

As explained in Remark \ref{r:absorbAlpha}, for the construction of the Lagrangian torus bundle over an edge $e$ of $\gamma'$, there is a freedom given by translating 
the 2-tori fibers by a function on $e$. Note that $[\alpha]$ is exactly the fundamental class of the trace of the $2\pi$ translation by a 2-torus in a 3-torus fiber.
By applying the freedom in the construction and wrapping around $-a$ times, we can construct $L_{\gamma'}$ such that $[L_{\gamma'}]=[L_\gamma]$.
%\footnote{\red{Say that we can absorb copies of $\alpha$ into $L_\gamma$}}
\end{proof}

\begin{proof}[Proof of Theorem \ref{t:homologous}]
 The Lagrangians $L_{\gamma}$ and $L_{\gamma'}$ being homologous is the content of Proposition \ref{prop-homologous}.
 Since they are rational homology spheres, they have unobstructed Floer cohomology  over characteristic $0$ \cite{FOOO}
 and we have $HF(L_{\gamma},L_{\gamma})=HF(L_{\gamma'},L_{\gamma'})=H^*(S^3)$ by the degeneration of the spectral sequence in the second page.
 Moreover, when $\gamma \cap \gamma' =\emptyset$, we have $HF(L_{\gamma},L_{\gamma'})=0$.
 By Hamiltonian invariance of Floer cohomology, we conclude that $L_{\gamma}$ is not Hamiltonian isotopic to $L_{\gamma'}$.
\end{proof}

\begin{remark}
Theorem \ref{t:homologous} also works when $\gamma \cap \gamma' $ is a single point.
In this case, if  $L_{\gamma}$ is Hamiltonian isotopic to $L_{\gamma'}$, then 
$HF(L_{\gamma},L_{\gamma'})$ would be well-defined but one can see from the local model that
$L_{\gamma}$ intersects cleanly with $L_{\gamma'}$ along a circle so 
$HF(L_{\gamma},L_{\gamma'})$ is either $0$ or concentrated on consecutive degree. It gives a contradiction.

It is less clear what $HF(L_{\gamma},L_{\gamma'})$ is when $\gamma$ overlaps with $\gamma' $
along a codimension $0$ subset. These cases arise in our computer-aided search.
\end{remark}

\subsection{Symplectomorphism group}

\begin{proof}[Proof of Corollary \ref{c:Symp}]
 Each spherical Lagrangian submanifold $L_i$ gives rise to a symplectomorphism $\tau_{L_i}:M \to M$, called the Dehn twist along $L_i$, supported inside an arbitrarily small neighborhood of $L_i$, see \cite{Seidel00}, \cite{MWspherical}.
 Therefore, it is clear that $\{\tau_{L_i}\}_{i=1}^{k_{\max}}$ generates an abelian subgroup in $\Sympl(M)$ that descends to an abelian subgroup $G$ of
 $\pi_0(\Sympl(M)) =\Sympl(M)/\Ham(M)$ (the equality uses the fact that $\pi_1(M)$ is trivial).

 We recall from \cite{Seidel00} that each $\tau_{L_i}$ can be
 lifted canonically to a $\mathbb{Z}$-graded symplectomorphism because $c_1(M)=0$.
 Moreover, we know that $\tau_{L_i}(L_i)=L_i[-2]$ and $\tau_{L_i}(L_j)=L_j$ as $\mathbb{Z}$-graded Lagrangians, for all $i \neq j$.
 Therefore, $g \in G$ is completely determined by $({HF}(L_i,g(L_j)) )_{i,j=1}^{k_{\max}}$ and $G$ is isomorphic to $\mathbb{Z}^{k_{\max}}$.
\end{proof}

\begin{remark}
 If $L_i$ is a spherical Lagrangian with $|\pi_1(L_i)|=m$, then $(\tau_{L_1})_*A=A+m([L_i]\cdot A)[L_i]$ for $A \in H_3(M,\mathbb{Z})$.
 Since $[L_i]=[L_j]$ for all $i,j$ (Theorem \ref{t:Construction}(2)), the natural map $G \subset \pi_0(\Sympl(M)) \to \Aut(H_3(M,\mathbb{Z}))$ has a large kernel.
 It is less clear what the kernel of the natural map $G \subset \pi_0(\Sympl(M)) \to \pi_0(\Diff(M))$ is.
\end{remark}

%%%%%%% Toric geometry in symplectic coordinates %%%%%%%%%%%%%%%%%%%%%%%%%%%%%%%%%%%%%%%%%%%%%%%%%%%%%%%%%%%%%%%%%%%%%%%%%%%%%%%%%%%%%%

\section{Toric geometry in symplectic coordinates}\label{ss:toricBasic}
We review some material about complex toric orbifolds. The presentation below is extracted from \cite{AbreuOrbifold} and \cite{Abreu} (see also \cite{Guillemin}, \cite{LermanTolman} and \cite{CdS}).
Any projective complex toric orbifold $X$  is K\"ahler and can be equipped with a K\"ahler form $\omega_X$
such that, for $i=\sqrt{-1}$, the action of the real torus
$$T^n:=i\mathbb{R}^n/2\pi i \mathbb{Z}^n \subset \mathbb{C}^n/2\pi i \mathbb{Z}^n=:T^n_{\mathbb{C}}$$
is effective and Hamiltonian with respect to $\omega_X$.
The effective Hamiltonian action induces a moment map $\pi_\Delta:X \to \mathbb{R}^n$
with image $\Delta:=\pi_\Delta(X)$ being a {\bf simple} and {\bf rational} convex polytope.
It means that $\Delta$ is a convex polytope such that
\begin{itemize}
 \item there are precisely $n$ edges meeting at each vertex $p$;
 \item each edge meeting a vertex $p$ is of the form $\{p+rv_j|r\in[0,r_j]\}$ for some $v_j \in \mathbb{Z}^n$, $r_j\ge 0$ for $1 \le j \le n$;
 \item $\{v_j\}_{j=1}^n$ form a $\mathbb{Q}$-basis of the lattice $\mathbb{Z}^n$.
\end{itemize}
If the last bullet is replaced by that $\{v_j\}_{j=1}^n$ can be chosen to be a $\mathbb{Z}$-basis of the lattice $\mathbb{Z}^n$, then $\Delta$ is called a Delzant polytope and $X$ is a smooth manifold.

We call a face of codimension one of $\Delta$ a \emph{facet}.

\begin{definition}
A labeled polytope is a simple rational convex polytope $\Delta$ plus a positive integer $m$ (label) attached to each facet of $\Delta$.
\end{definition}
The label $m$ of a facet $F$ is the order of the orbifold structure group of the generic points in $(\pi_\Delta)^{-1}(F)$.
If not mentioned, we assume all labels to be $1$.

Lerman and Tolman \cite{LermanTolman} prove that a labeled simple rational convex polytope $\Delta$ determines a
unique (up to equivariant symplectomorphism) compact symplectic orbifold $(X,\omega_X)$ with effective Hamiltonian torus action and moment map image $\Delta$,
which is a generalization of Delzant's result on Delzant polytope and compact symplectic manifold $(X,\omega_X)$ with effective Hamiltonian torus action \cite{Delzant}.
They also prove that if $J_1$ and $J_2$ are torus invariant complex structures on $X$ that are  compatible with $\omega_X$
then $(X,J_1)$ and $(X,J_2)$ are equivariantly biholomorphic
(\cite[Theorem 9.4]{LermanTolman}, see also \cite[Section 2]{AbreuOrbifold}).
However, since there can be different torus invariant K\"ahler structures on $X$, we need to go into details about the transition between complex and symplectic coordinates.

\subsection{Complex coordinates}
Let $X^\circ:=\{x \in X\,|\, T^n \text{ acts freely on } x\}$.
There is a biholomorphic identification
$$X^\circ=\mathbb{C}^n/2\pi i\mathbb{Z}^n=\{u+iv \,|\, u\in \mathbb{R}^n, v \in \RR^n/2\pi\mathbb{Z}^n \}$$
such that $t \in T^n$ acts by $$t \cdot (u+iv)=u+i(v+t).$$
The K\"ahler form $\omega_X$ is given by $\omega_X:=2i\partial \bar{\partial} f_{\omega}$
for a potential $f_{\omega}(u,v)=f_{\omega}(u) \in C^{\infty}(X^\circ)$, depending only on $u$ (see \cite{Guillemin} or \cite[Exercise $3.5$]{Abreu} for the definition of $f_{\omega}(u)$).

\subsection{Symplectic coordinates}
Dually, we have the symplectic identification
$X^\circ=\Delta^\circ \times T^n$,
where $\Delta^\circ$ is the interior of $\Delta$.
The torus acts on $(p,q) \in \Delta^\circ \times T^n$ by
$$t \cdot (p,q)=(p,q+t)$$
and the symplectic form is $\omega_X:=dp \wedge dq$.
The complex structure $J$ is determined by a function $f_J(p,q)=f_J(p) \in C^{\infty}(X^\circ)$ according to  the following procedure.
Let $F_J:=\Hess_p(f_J)$ be the Hessian of $f_J$ in the $p$ coordinates ($f_J$ and $F_J$ are denoted by $g$ and $G$, respectively, in \cite{AbreuOrbifold}).
The complex structure in $(p,q)$ coordinates is given by
 \[
   J=
  \left[ {\begin{array}{cc}
   0 & -F_J^{-1} \\
   F_J & 0 \\
  \end{array} } \right].
\]
The transition maps between the complex and symplectic coordinates are given by
\begin{align} \label{eq:Transition}
\left\{
\begin{array}{ll}
 p=\frac{\partial f_{\omega}}{\partial u},&q=v,\\
 u=\frac{\partial f_J}{\partial p},& v=q. 
\end{array}
\right. 
\end{align}
There are restrictions for $f_{\omega}$ and $f_J$ to satisfy near infinity so that we have a well-defined K\"ahler structure on $X$.

A canonical choice of complex structure is given by Guillemin as follows.
The simple rational convex polytope $\Delta$ can be described by a set of inequalities of the form
\begin{align*}
 \langle p, \mu_r \rangle-\rho_r \ge 0 \text{ for } r=1,\dots,d
\end{align*}
where $d$ is the number of facets, each $\mu_r$ is a primitive element of $\ZZ^n$ and $\rho_r \in \RR$.
We define affine linear functions $l_r:\mathbb{R}^n \to \RR$, $r=1,\dots,d$,
\begin{align*}
l_r(p):= \langle p, m_r\mu_r \rangle-\lambda_r
\end{align*}
where $m_r$ is the label of the $r^{th}$ facet and $\lambda_r=m_r \rho_r$, so $p \in \Delta$ if and only if $l_r(p) \ge 0$ for all $r=1,\dots,d$.

\begin{theorem}[\cite{AbreuOrbifold}, \cite{Abreu}, \cite{Guillemin}]\label{t:Guillemin}
 The `canonical' compatible complex structure $J_{\Delta}$ on $\Delta^\circ \times T^n$ is given (in $(p,q)$-coordinates) by

  \begin{align}
   J_{\Delta}=
  \left[ {\begin{array}{cc}
   0 & -F_{J,can}^{-1} \\
   F_{J,can} & 0 \\
  \end{array} } \right] \label{eq:Jcan}
  \end{align}
%   J_{\Delta}=
%  \left[ {\begin{array}{cc}
%   0 & -F_{J,can}^{-1} \\
%   F_{J,can} & 0 \\
%  \end{array} } \right]
%\]
where $F_{J,can}=\Hess_p(f_{J,can})$ and
\begin{align}
f_{J,can}(p):=\frac{1}{2} \sum_{r=1}^dl_r(p) \log(l_r(p)) \label{eq:fJcan}.
\end{align}

\end{theorem}

\begin{remark}
Fixing $\omega_X$, all torus invariant complex structures $J$ on $X$ compatible with $\omega_X$ are classified in \cite[Theorem $2$]{AbreuOrbifold}.
\end{remark}

\begin{example}[Extending charts]\label{ex:standard}

We consider the following important non-compact example.
Let $X=\CC^n$ with moment polytope $\Delta=\RR^n_{\ge 0}$ and
$l_r(p_1,\dots,p_n)=p_r$.
We have symplectic coordinates $(p_j,q_j) \in X^{\circ} =(\CC^*)^n \subset X$.
Define $z_j=x_j+iy_j=\sqrt{2p_j}\exp(iq_j) \in \CC^*$, so that we have $\sum_{j=1}^n dx_j \wedge dy_j=\sum_{j=1}^n dp_j \wedge dq_j$.
We can extend the domain of $z_j$ from $\CC^*$ to $\CC$ and
thus provide a symplectic chart to $X$ and moment map $X\ra\RR_{\ge 0}^n$ is given by
$(z_1, \dots,z_n)\mapsto\frac{1}{2}(|z_1|^2, \dots, |z_n|^2)$.

For the complex coordinates, \eqref{eq:fJcan} yields $f_{J,can}=\frac{1}{2} \sum_{j=1}^n p_j \log(p_j)$ and
$\frac\partial{\partial{p_j}}f_{J,can}=\frac{1}{2}(1+\log(p_j))$, so the Hessian of $f_{J,can}$ is given by
  \[
   F_{J,can}=
  \left[ {\begin{array}{ccc}
   \frac{1}{2p_1} & 0 & 0 \\
   0 & \dots & 0\\
   0 & 0 & \frac{1}{2p_n}\\
  \end{array} } \right].
\]

We define $J_{\Delta}$ by Equation \eqref{eq:Jcan}.
Then a direct calculation gives
\begin{align*}
 J_{\Delta}(\partial_x)=\partial_y.
\end{align*}
Let $u_j=\frac{\partial f_{J,can}(p)}{\partial p_j}$, $v_j=q_j$ and $w_j=e^{u_j+iv_j}$ be the holomorphic coordinates on $(\CC^*)^n$ (see \eqref{eq:Transition}).
Then $u_j=\frac{1}{2}(1+\log(p_j))$ and $w_j=e^{\frac{1}{2}}\sqrt{p_j}e^{iq_j}=(\frac{e}{2})^{\frac{1}{2}}z_j$.
The holomorphic coordinates $(w_1, \dots,w_n)$ on $(\CC^*)^n$ naturally extend to holomorphic coordinates on $\CC^n$.

\end{example}

%When we define $z_j:=x_j+iy_j=\sqrt{2p_j}\exp(iq_j) \in \CC^*=(\RR^2)^*$, it is true in general that
%$dx \wedge dy=dp \wedge dq$ so we can obtain symplectic charts for points $p \in X$ where $T^n$ does not act freely.

%These symplectic charts are almost never preserving the complex structures.

\begin{lemma}[Integral linear transformation]\label{l:IntLinearTran}
 Let $\Delta_1$ be a labeled polytope and $\Delta_2=A\Delta_1+v$ where $A\in GL_n(\mathbb{Z})$ and $v \in \mathbb{Z}^n$.
 Let $X_1$ and $X_2$ be the canonical K\"ahler toric orbifold with moment polytope being $\Delta_1$ and $\Delta_2$, respectively.
 Then $X_1$ and $X_2$ are K\"ahler isomorphic.
\end{lemma}

\begin{proof} 
This follows from realizing that neither the definition of the symplectic nor complex structures needs coordinates, as the $\mu_r$ are intrinsic to the integral affine structure and hence are the $l_r$.
\end{proof}

\begin{example}[Transforming hypersurfaces]\label{ex:Transforming hypersurfaces}
Let $X$ be a toric manifold with moment image a Delzant polyhedron $\Delta$.
By picking a vertex $v$ and replacing $\Delta$ by $A(\Delta-v)$ for some $A \in GL_n(\mathbb{Z})$ (see Lemma \ref{l:IntLinearTran}), we can assume $l_r(p_1,\dots,p_n)=p_r$ for $r=1,\dots,n$
and the remaining facets of $\Delta$ are contained respectively in $l_r =0$ for $r=n+1, \dots, d$.
 Let $w_j=\exp(u_j+iv_j)=\exp(\frac{\partial f_{J,can}(p)}{\partial p_j}+iq_j) \in \CC^*$,
 which gives a $T_{\CC}^n$ equivariant identification between $(\CC^*)^n \subset \CC^n$ and $X^\circ$.
 We know that  (see \eqref{eq:fJcan})
 \begin{align}
  f_{J,can}:=\frac{1}{2} \sum_{j=1}^n p_j \log(p_j)+ R
 \end{align}
 where $R$ is the contribution from other facets. 
Assume now we are given a family of hypersurfaces via 
\begin{equation} \label{eq-hypersurface}
 w_1\dots  w_n=tg(w)
\end{equation}
for some polynomial $g$ in holomorphic coordinates and $t\in \CC$ a family parameter.
The logarithm of this hypersurface equation is transformed to
\begin{equation}
\label{eq-transformed-hypersurface}
\frac{n}{2}+\log\left(\sqrt{\prod_{j=1}^n p_{j}}\right)+i\left(\sum_{j=1}^n q_{j}\right)+\sum_{j=1}^n \frac{\partial R(p)}{\partial p_{j}}=\log(t)+\log(g(w(p,q)))
\end{equation}
in symplectic coordinates.
Notice that $R$ can be smoothly extended to the origin, so by exponentiating and setting $z_{j}:=\sqrt{2p_{j}}\exp(iq_{j})$, we may write this equation as
\begin{equation} \label{eq-zf}
 \prod_{j=1}^n z_{j}=tf(p,q)
\end{equation}
 where 
\begin{equation}
f(p,q)=g(w(p,q))h(p)
\label{eq-nonvanishing-factor-in-p}
\end{equation}
 for $h=\sqrt{2}^{n}\exp(-\frac{n}{2}-\sum_{j=1}^n \frac{\partial R(p)}{\partial p_{j}})$. 
 Most importantly later on, $h$ is a non-vanishing $C^\infty$-function depending only on $p$.
 \end{example}

 With the above example, we know how to transform a complex hypersurface defined by the equation $w_{1}\dots w_{n}=tg(w)$
 into a symplectic hypersurface in symplectic coordinates $(p,q)$ for a toric manifold $X$.
 To cover a large range of applications, we need an analogue for toric orbifolds.

\subsection{Isolated Gorenstein toric orbifold singularities} \label{section-orbifold-local-model}
Now consider a cone $\Delta\subset\RR^n$ generated by $v_1,...,v_n\in\ZZ^n$.
The ring $\CC[\Delta\cap\ZZ^n]$ is the coordinate ring of an Abelian quotient singularity $X_\Delta$ as follows. The ring is regular if an only if the $v_i$ form a lattice basis. Let $\sigma$ be the dual cone of $\Delta$. 
It is also integrally generated, so let $N$ be the sublattice generated by the primitive ray generators of $\sigma$ as a sublattice of $(\ZZ^n)^*$, the dual lattice $M=\Hom(N,\ZZ)$ contains the original lattice $\ZZ^n$ and the cone $\Delta$ is a standard cone when viewed with respect to $M$, i.e.~$\CC[\Delta\cap M]=\CC[w_1,...,w_n]$ where $w_j$ is the monomial given by the primitive generator of $(\RR_{\ge 0}v_j)\cap M$.
The subring $\CC[\Delta\cap\ZZ^n]\subseteq \CC[\Delta\cap M]$
is the ring of invariants of the group action $K=(\ZZ^n)^*/N$ that acts on a monomial $z^m$ via $g.z^m=\exp(2\pi i \langle g,m\rangle)z^m$, see \cite[\S2.2,\,page 34]{Fulton1993}.
We need this a bit more explicit and also want to make further assumptions.
We require the singularity to be isolated. 
Since then $K$ necessarily acts faithfully on the subring $\CC[w_1]=\CC[(\RR_{\ge 0}v_1)\cap M]$, we conclude that $K$ is cyclic, say $K$ is the group of $k$th roots of unity.
Let $\zeta$ be a primitive generator. The action is 
$$\zeta.(w_1,...,w_n)=(\zeta^{a_1}w_1,...,\zeta^{a_n}w_n)$$ 
for some integers $a_j$ with $\gcd(a_j,k)=1$ for all $j$ which is equivalent to the isolatedness of the singularity. One can check the following result.
\begin{lemma}\label{l:unbranched}
Under the given assumptions, the cover $\CC^n\ra \CC^n/K$ is unbranched away from the origin.
\end{lemma}
We want to further assume that the singularity is Gorenstein which is equivalent to the statement that the Gorenstein monomial $w_1w_2...w_n$ is invariant under $K$, that is $$\sum_j a_j\in k\ZZ.$$

We now address the symplectic coordinates. 
Let $T_0=(S^1)^n$ and consider the standard $T_0$-action on $\CC^n$ by 
  $\theta \cdot (w_1,\dots,w_n)=(e^{ \theta_1 i }w_1,...,e^{ \theta_n i }w_n)$.
Recall from Example~\ref{ex:standard} that the standard symplectic coordinates of the toric variety $\CC^n$ are
$z_j=e^{-\frac{1}{2}}w_j$ and $z_j=\sqrt{2p_j}\exp(iq_j)$ giving the moment map $\CC^n\ra \RR_{\ge 0}^n$.
Note that $K$ is a subgroup of $T_0$, so $T_1=T_0/K$ acts faithfully on the orbifold singularity $\CC^n/K$.
We claim that the moment map of $\CC^n$ factors through that of $\CC^n/K$, that is
$$
\xymatrix{
\CC^n \ar[d]\ar[r]& \CC^n/K \ar^{\pi}[d]  \\
\RR_{\ge 0}^n&\Delta\ar[l]
}
$$
where the bottom horizontal map is the real affine isomorphism given by the fact that $\Delta$ becomes a standard cone with respect to $M$.
The right vertical map is the moment map of the orbifold singularity. The diagram clearly commutes and since the symplectic structures can be defined using the moment maps, the diagram is compatible with symplectic structures. 
The only thing to check is that the complex structures used in the diagram coincide with the canonical ones obtained from the complex potential $f_{J,can}$ in Theorem~\ref{t:Guillemin}. 
By Example~\ref{ex:standard} this is true for the left vertical map.
Since the $\mu_j$ are actually the primitive generators of the rays of $\sigma$, and are therefore contained in $N$, we find that the potential
$f_{J,can}$ for $(\Delta,\ZZ^n)$ is identical with the one for $(\Delta,M)$ which gives the desired compatibility.

We finally want to consider the situation where the Gorenstein singularity appears locally at the vertex of a compact polytope $\Delta$.
Let $\PP_\Delta$ be the compact K\"ahler orbifold obtained from $\Delta$ and $\pi_\Delta:\PP_\Delta\ra\Delta$ the moment map.
Let $v\in \Delta$ be a vertex. Replacing $\Delta$ by $\Delta-v$ and invoking Lemma~\ref{l:IntLinearTran}, we may assume $v=0$.
Compared to the local study above, there is no difference for the complex structure, however, the compact polytope $\Delta$ gives a different symplectic structure on the local model $\CC^n$.

Consider a neighborhood $O_v$ of $v$ in $\Delta$ which is then also a neighborhood of $v$ in the cone $\RR_{\ge 0}\Delta$.
The two inverse images under the moments maps $\pi_\Delta^{-1}(O_v)$ and $\pi^{-1}(O_v)$ resulting from this are naturally symplectomorphic.
Assume now we have a family of hypersurfaces in $\CC^n/K$ as given by \eqref{eq-hypersurface}, i.e.
$$w_{1}\cdot...\cdot w_{n}=tg(w)$$
where we use the coordinates of $\CC^n$ and so $g(w)$ is now a $K$-invariant polynomial. 
By the Gorenstein assumption, the monomial $w_1w_2...w_n$ is $K$-invariant. 
The same analysis as in Example~\ref{ex:Transforming hypersurfaces} gives \eqref{eq-transformed-hypersurface} as the equation for the family of hypersurfaces in
symplectic coordinates with the only difference that now $f$ and $h$ are $K$-invariant.
 
\subsection{Corner charts in four-orbifolds} \label{section-corner-chart}
Let $\PP_\Delta$ be a four-dimensional Gorenstein projective toric orbifold with isolated singularities and moment polytope $\Delta$.
For each point $z$ of $\PP_\Delta$, we can choose a vertex $v$ of $\Delta$ lying in the face containing $\pi_\Delta(z)$.
Let 
\begin{align}
\coDelta=\Delta \backslash \cup_{v \not\in F} F \label{eq:cornerPolyhedron}
\end{align}
where $F$ are facets of $\Delta$. 
%be the partial compactification of the interior $\Delta^\circ$ by the facets of $\Delta$ that contain $v$.
If $\pi_\Delta^{-1}(v)$ is a smooth point of $\PP_\Delta$, then we can, by an integral affine linear transform, assume $v$ is the origin 
and the primitive edge directions emerging from $v$ coincide with the positive real axes in $\RR^4$. If $\Delta$ has integral points in its interior, after the transform, $(1,1,1,1)$ must be one of them (in fact the only one if $\Delta$ is reflexive).
We can give a symplectic chart $U \subset \RR^{8}$ to $\pi_\Delta^{-1}(\coDelta)$ as in Example~\ref{ex:standard}, which is $T_{\CC}^4$-equivariantly biholomorphic to $\CC^4$ (see Figure \ref{graph:cornerchart}).
More generally, if $\pi_\Delta^{-1}(v)$ is an orbifold point of $\PP_\Delta$, then we have just shown in Section \ref{section-orbifold-local-model}
that $\pi_\Delta^{-1}(\coDelta)$ is equivariantly symplectomorphic to the model $X_v=\CC^4/K$ with the symplectic structure induced from $\pi_\Delta$. The smooth case can be viewed like the situation $K=\{\id\}$.
In both cases, we call $X_v$ a {\bf symplectic corner chart} for $\PP_\Delta$ associated to the vertex $v$.
All mirror quintic threefolds are hosted inside a toric variety $\PP_\Delta$ of the type considered here.
\begin{figure}
 \includegraphics{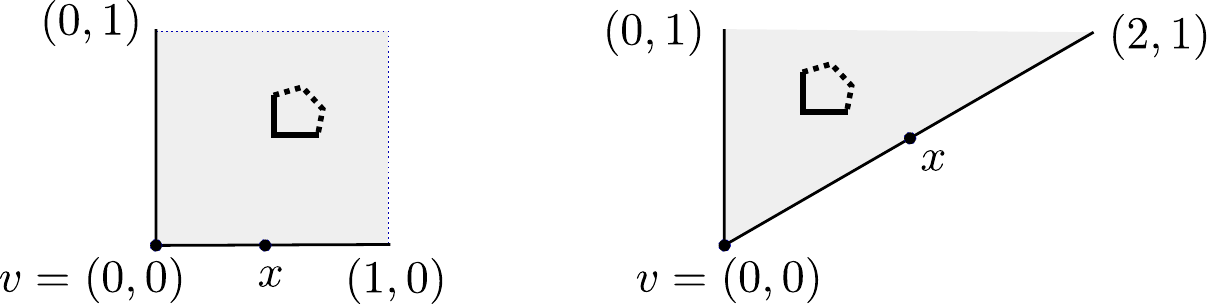}
 \caption{Left: two-dimensional analogue when $\pi_\Delta^{-1}(v)$ is a smooth point. 
 Right: two-dimensional analogue when $\pi_\Delta^{-1}(v)$ is an orbifold point. }
 \label{graph:cornerchart}
\end{figure}

\section{Geometric setup}\label{ss:geometricSetup}
%\begin{notation}\label{n:BAD}
%We will use the following notations
%\begin{align}
% B^2_r &= \{z \in \mathbb{C}| |z| \le \sqrt{2r}\} \\
% A^2_{a_0,a_1} &= \{(p,q) \in \RR \times \RR/2\pi\ZZ| a_0 \le p \le a_1\} \\
% D^2_{(p_0,q_0), \epsilon} &= \{(p,q) \in \RR \times \RR/2\pi\ZZ| |p-p_0|, |q-q_0| \le \epsilon\}
%\end{align}
%where $a_0< a_1$, $\epsilon<2\pi$ and $(p_0,q_0) \in \RR \times \RR/2\pi\ZZ$.
%We use $B^2$, $A^2$ and $D^2_{(p_0,q_0)}$ respectively
%to denote $B^2_r$, $A^2_{a_0,a_1}$ and $D^2_{(p_0,q_0), \epsilon}$ 
%for some appropriate choices of $r,a_0,a_1,\epsilon$ that are not specified.
%With these notations, a symplectic corner chart is given by $U=(B^2)^n$ where the $B^2$ in different factors of $U$ can have different sizes. 
%\end{notation}

%This section will gather ingredients for construction in the $3$-dimensional case with presentation in parallel to the $2$-dimensional case.
Let $\PP_\Delta$ be a complex projective toric orbifold of complex dimension four with moment polytope $\Delta$. 
Recall that $-K_{\PP_\Delta}=\partial \PP_\Delta$, we assume this is nef or equivalently (for a toric variety) that $\shO(-K_{\PP_\Delta})$ is generated by global sections (\cite{Oda}, Theorem~2.7). 
Let $\Delta_K$ denote the corresponding lattice polytope. We have a birational morphism $\PP_\Delta\ra \PP_{\Delta_K}$ that we will use to pull back an anti-canonical hypersurface.
We equip $\PP_\Delta$ with the canonical K\"ahler structure.
Set $\shL:=\shO(-K_{\PP_\Delta})$ and let $s_0 \in H^0(\PP_\Delta,\shL)$ such that $s_0^{-1}(0)=\partial \PP_\Delta$.

Let $C^{\infty}(\PP_\Delta,\shL)$ denote the vector space of $C^{\infty}$-sections of $\shL$.
For every $s  \in C^{\infty}(\PP_\Delta,\shL)$ and $t\in\CC$, we define
\begin{align}
 M_t^s:=\{s_0=ts\} \subset \PP_\Delta.
\end{align}
The total family of $M_t^s$ is denoted by $M^s$.
Let $(\partial \PP_\Delta)_{\Sing}$ denote the locus of
singular points of $\partial \PP_\Delta$ (we also used $ \PP_\Delta^{[2]}$ before).
%which is also the union of toric strata of complex codimension at least two.
We define the discriminant of $s$ via
\begin{align}
 \Disc(s):=s^{-1}(0) \cap (\partial \PP_\Delta)_{\Sing}.
\end{align}

%We are interested in the germ at $t=0$ so we always assume that $|t|>0$ is sufficiently small when necessary.
As explained in Section \ref{section-corner-chart},
a symplectic corner chart $U/K$ comes together with the quotient map $\Pi_{U/K}:U \to U/K$
and the diffeomorphism $\Phi_U:U \to \CC^4$. 
In a symplectic corner chart, we define 
\begin{align}
 \widetilde{M}_t^s:=&\Pi_{U/K}^{-1}(M_t^s \cap U/K )    \label{eq:tildeLocal1}\\
 =&\Phi_U^{-1}(\{w \in \CC^4|w_1w_2w_3w_4=tf(w)\}) \label{eq:tildeLocal2}
\end{align}
for some $K$-invariant function $f \in C^{\infty}(\CC^4,\CC)$.
The second equality comes from the fact that, with respect to a choice of trivialization, $s_0=w_1w_2w_3w_4h(w)$ for some non-vanishing $K$-invariant function $h$ on $\CC^4$.
It is clear that if $s \neq 0$ at the orbifold points of $\PP_\Delta$, then $M_t^s$ does not contain any orbifold point whenever $t\neq 0$.

When $s_1 \in H^0(\PP_\Delta,\shL)$, we get a family of complex subvarieties $M_t^{s_1}$ parametrized by $t \in \mathbb{C}$.
Let 
\begin{align}
 H^0(\PP_\Delta,\shL)_{\Reg}:=\{s_1 \in H^0(\PP_\Delta,\shL)| M_t^{s_1} \text{ is smooth for all $|t|>0$ small}\} \label{eq:RegSection}
\end{align}
When $M_t^{s_1}$ is a smooth manifold, it is a symplectic hypersurface in $\PP_\Delta$ and the symplectomorphism type is independent of $t$ by Moser's argument. 
For smooth but not necessarily holomorphic sections, we have the following sufficient condition to guarantee that
$M_t^s$ is symplectic (when $t$ is sufficiently close to $0$). 

%\hrule
%MAKE OUR PURPOSE CLEAR  (also note $s_1$ not $s^1$)
%\hrule
\begin{lemma}[Good deformation]\label{l:GoodDeformation}
 Let $s_1 \in H^0(\PP_\Delta,\shL)_{\Reg}$.
 Suppose we have a smooth family $(s^u)_{u \in [0,1]} \in C^{\infty}(\PP_\Delta,\shL)$ such that 
 \begin{itemize}
  %\item $s_1^0=s_1$,
  \item $s^u=s_1$ near $\Disc(s_1)$ for all $u$,
  \item $\Disc(s^u)=\Disc(s_1)$ for all $u$ 
 \end{itemize}
then there exist $\delta>0$ such that $M_{t}^{u}:=M_t^{s^u}$ is a smooth symplectic hypersurface in $\PP_\Delta$ for all $0<|t| < \delta$ and all $u$. 
\end{lemma}

\begin{proof}
For any regular neighborhood $N$ of $\partial \PP_\Delta$, there exists $\delta'>0$ such that $M_t^u \subset N$ for all $|t|< \delta'$ for all $u$.
This is because $M_t^u$ $C^0$-converges to  $\partial \PP_\Delta$ uniformly as $|t|$ goes to $0$.
Therefore, if for each point $x \in \partial \PP_\Delta$, we can find a neighborhood $O_x$ of $x$ such that $M_t^u \cap O_x$ is symplectic  for all $|t|>0$ small and all $u\in [0,1]$, 
then $M_{t}^{u}$ is symplectic for all $|t|>0$ small and all $u\in [0,1]$.

Since $M_t^u$ is independent of $u$ in a neighborhood $O_{\Disc}$ of $\Disc(s_1)$ (by the first bullet), 
we can take $O_x=O_{\Disc}$ if $x \in O_{\Disc}$.
Now we assume that $x \in \partial \PP_\Delta \setminus O_{\Disc}$.

First suppose $\pi_\Delta(x)$ lies in the interior of a $3$-cell.
There exists a symplectic corner chart $U/K$ and an open subset $V \subset U$ such that $x \in \Pi_{U/K}(V)$, $\Pi_{U/K}(V) \cap \Disc(s_1)= \emptyset$ and
\begin{align*}
 \widetilde{M}_t^u \cap V=\Phi_U^{-1}(\{w_1=tf^u(w)\})
\end{align*}
for some smooth family of functions $f^u:\Phi_U(V) \to \mathbb{C}$.
This is because we can assume $w_2,w_3,w_4$ are invertible in $\Phi_U(V)$ and absorbed by $f^u$.
Let $F^u_t=w_1-tf^u(w)$.
The differential is given by 
\begin{align*}
 DF^u_t=[1,0,0,0]-tDf^u
\end{align*}
Since $\ker(DF^u_t)=T\widetilde{M}^u_t$ and the first term of $DF^u_t$ dominates 
(say, with respect to the Euclidean norm in the chart) when $|t|$ small, $\widetilde{M}_t^u \cap V$ is symplectic  for all $|t|>0$ small and all $u\in [0,1]$.
Therefore, we can take $O_x=\Pi_{U/K}(V)$.

Now suppose $\pi_\Delta(x)$ lies in the interior of a $2$-cell.
There exists a symplectic corner chart $U/K$ and an open subset $V \subset U$ such that $x \in \Pi_{U/K}(V)$ and
\begin{align*}
 \widetilde{M}_t^u\cap V=\Phi_U^{-1}(\{w_1w_2=tf^u(w)\})
\end{align*}
 for some smooth family of functions $f^u:\Phi_U(V) \to \mathbb{C}$ such that $f^u(0,0,w_3,w_4) \neq 0$ 
 (by the second bullet and the assumption that $x \in \partial \PP_\Delta \setminus O_{\Disc}$).
 It is because we can assume $w_3,w_4$ are invertible in $\Phi_U(V)$ and $\Phi_U(\Pi_{U/K}^{-1}((\partial \PP_\Delta)_{\Sing}) \cap V)=\Phi_U(V) \cap \{w_1=w_2=0\}$.
 Therefore, there exists $c>0$ such that $\max \{|w_1|,|w_2|,|f^u(w)|\} >c$ for all points in $\Phi_U(V)$.
 %, where $|-|$ is the Euclidean norm in $\Phi_U(V)$.
 Let $F^u_t=w_1w_2-tf^u(w)$. 
  The differential is given by 
\begin{align*}
DF^u_t=[w_2,w_1,0,0]-tDf^u 
\end{align*}
Again, we want to show that the first term of $DF^u_t$ dominates for $w \in \{F^u_t=0\}$ for all $u$ when $|t|>0$ small.
 
Since $\|Df^u\|$ is bounded, the norm of the second vector is of order $|t|$.
 At points where $|w_1|>c$ or $|w_2|>c$, the first term clearly dominates when $|t|>0$ small.
 By the assumption, all other points satisfies $|f^u| >c$.
 As a result, for $w \in \{F_t^u=0\}$, we have
 $|w_1|^2+|w_2|^2 \ge 2|w_1w_2|=2|tf^u|>2|t|c$ so the norm of $[w_2,w_1,0]$ is of order at least $\sqrt{|t|}$ and hence dominates when $|t|>0$ small.
It implies that there exist $\delta>0$ such that $\widetilde{M}_t^u\cap V$ is a symplectic manifold for all $0<|t| < \delta$ and all $u$.

Similarly, when $\pi_\Delta(x)$ lies in the interior of a $1$-cell, we have
\begin{align*}
 \widetilde{M}_t^u \cap V=\Phi_U^{-1}(\{w_1w_2w_3=tf^u(w)\})
\end{align*}
for some $V$ and $f^u$.
There exists $c>0$ such that $\max \{|w_1w_2|,|w_2w_3|,|w_1w_3|,|f^u(w)|\} >c$.
At points where $|w_1w_2|>c$ or $|w_2w_3|>c$ or $|w_1w_3|>c$, the first term of $DF^u_t$, which is given by $[w_2w_3,w_1w_3,w_1w_2,0]$, dominates when $|t|>0$ small.
At points where $|f^u| >c$, we have $|w_1w_2|^2+|w_2w_3|^2+|w_1w_3|^2 \ge 3|w_1w_2w_3|^{\frac{4}{3}}>3|t|^{\frac{4}{3}}c$ so 
the norm of the first term of $DF^u_t$ is  of order $|t|^{\frac{2}{3}}$ and the second term of $DF^u_t$ is  of order $|t|$ so the first term dominates when $|t|>0$ small.

One can do the same analysis when $\pi_\Delta(x)$ is a vertex of $\Delta$. In this case, the norm of 
the first term and second term of $DF^u_t$ is  of order $|t|^{\frac{3}{4}}$ and $|t|$, respectively.
%By compactness, $\delta$ can be taken uniformly for all charts.
 
\end{proof}

We remark that $\Disc(s^u)=\Disc(s_1)$ implies that $s^u$ does not vanish at the orbifold points.
In view of Lemma \ref{l:GoodDeformation}, it is convenient to have the following definition.

\begin{definition}\label{d:admissible}
 Let $s_1 \in H^0( \PP_\Delta,\shL)_{\Reg}$. We say that $s\in C^\infty(\PP_\Delta,\shL)$ is \emph{$s_1$-admissible} if
 $s=s_1$ in a neighborhood of $\Disc(s_1)$ and $\Disc(s)=\Disc(s_1)$.
 
 We say that $s\in C^\infty(\PP_\Delta,\shL)$ is \emph{admissible} if
 it is $s_1$-admissible for some $s_1 \in H^0( \PP_\Delta,\shL)_{\Reg}$.
\end{definition}
%\begin{minipage}{\linewidth}
\noindent
\begin{minipage}{.5\linewidth}
\begin{corollary}\label{c:firstApproximation} 
 For $s_1 \in H^0( \PP_\Delta,\shL)_{\Reg}$ and any regular neighborhood $N$ of $(\partial \PP_\Delta)_{\Sing}$, there is a symplectic hypersurface $M \subset \PP_\Delta$
 such that $M$ is symplectic isotopic to $M_t^{s_1}$ for some $|t|>0$ small, and 
 $(\partial \PP_\Delta \setminus N) \subset M \subset (\partial \PP_\Delta  \cup N)$, see Figure~\ref{graph:good-defo}.
\end{corollary}
\emph{Proof.}
 Let $N' \subset N$ be a smaller neighborhood of $(\partial \PP_\Delta)_{\Sing}$.
 Let $\chi: \PP_{\Delta}  \to \RR$ be a smooth function\hfill that\hfill has\hfill values\hfill $1$\hfill in\hfill $N'$
\vspace{7pt}
\end{minipage}
\begin{minipage}{.5\linewidth}
%\captionsetup{width=.85\linewidth}
%\centering
 \begin{center}
 \includegraphics[width=0.4\textwidth]{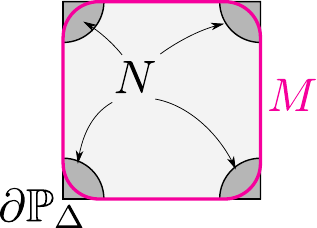} 
 \captionof{figure}{\label{graph:good-defo}The symplectic model $M$ close to the boundary of $\PP_\Delta$.}
 \end{center}
\end{minipage}
%\end{minipage}
 and $0$ outside $N$.
 Then $s:=\chi s_1$ is an $s_1$-admissible section.
 Moreover, $s^u:=(1-u)s_1+us$ is a smooth family of $s_1$-admissible sections so we can apply Lemma \ref{l:GoodDeformation}.
 Let the resulting family be $M^u_t$.
 By Moser's argument, $M_t^0=M_t^{s_1}$ is symplectic isotopic to $M:=M_t^1$ when $0<|t|< \delta$.
 
 It follows from the definition of $s$ that 
 for $x \in (\PP_\Delta \setminus N)$, we have 
$$x \in M \quad \Leftrightarrow \quad s_0(x)-ts(x)=0 \quad \Leftrightarrow \quad s_0(x)=0 \quad \Leftrightarrow \quad x \in \partial \PP_\Delta \setminus N.$$
This gives the assertion.\qed

An important consequence of Corollary \ref{c:firstApproximation} is that we can transfer the Lagrangian torus fiber bundle structure of 
$\partial \PP_\Delta \setminus (\partial \PP_\Delta)_{\Sing}$ to a Lagrangian torus fiber bundle structure in a large open subset of $M$, and hence a large open subset of $M_t^{s_1}$.
%This is helpful for constructing Lagrangian.
%The following sections will be focused on local models near $(\partial \PP_\Delta)_{\Sing}$.

\begin{lemma}\label{l:symplecticIsotopy}
 If $(s^u)_{u \in [0,1]}$ is a family of $s_1$-admissible sections such that, for some open subset $V \subset \PP_\Delta$, $s^u|_{\PP_\Delta \setminus V}$ is independent of $u$ 
 then there exists $\delta >0$ such that for all $0< |t|< \delta$, there is a symplectomorphism $\phi_{V,t}:M^{s^0}_t \to M^{s^1}_t$
 such that $\phi_{V,t}|_{M^{s^0}_t \setminus V}$ is the identity.
\end{lemma}

\begin{proof}
By Lemma \ref{l:GoodDeformation}, $M^{s^u}_t$ is a family of symplectic hypersurfaces for $0<|t|< \delta$.
By assumption, $M^{s^u}_t \cap (\PP_\Delta \setminus V)$ is independent of $u$.
The existence of $\phi_{V,t}$ follows from a standard application of Moser's argument.
\end{proof}

Outlook: recall $\gamma$, $W_{\epsilon}(\gamma)$ and $M_t$ from Theorem~\ref{t:Construction}.
In its proof, for all $\epsilon>0$, we will construct
a family of  admissible sections $(s^u)_{u \in [0,1]}$ such that $M^{s^0}_t=M_t$ and
$M^{s^1}_t \cap \pi_{\Delta}^{-1}(W_{\epsilon}(\gamma))$ contains a Lagrangian $L$ which is diffeomorphic to a Lagrangian lift of $\gamma$ for all $|t|>0$ small.
Moreover,  for $V:= \pi_{\Delta}^{-1}(W_{\epsilon}(\gamma))$, $s^u|_{\PP_\Delta \setminus V}$ will be independent of $u$.
We can apply Lemma \ref{l:symplecticIsotopy} to get a symplectomorphism $\phi_{V,t}: M_t \to M^{s^1}_t $ and $\phi_{V,t}^{-1}(L)$
will be our desired Lagrangian in $M_t \cap \pi_{\Delta}^{-1}(W_{\epsilon}(\gamma))$.
%The existence of $L$ implies Theorem \ref{t:Construction}.

\section{Away from discriminant}\label{s:AwayFromSing}

This section gives the construction of Lagrangians away from the discriminant.
In Subsection \ref{ss:StandardLag}, we give a local Lagrangian model and explain how to glue these Lagrangian models away from the discriminant.
We will complete our Lagrangian construction away from the discriminant after the discussion in Subsection \ref{ss:trivalent}, which  concerns trivalent vertices of a tropical curve.
We conclude the proof of Theorem \ref{t:Construction} in Subsection \ref{ss:Concluding}.
%In this section, we assume that the symplectic corner charts $U/K$ have $K=\{id\}$ (i.e. $\PP_{\Delta}$ is a toric manifold). The generalization to the orbifold case can be carried out by passing to the local cover as in the previous section without any difficulties, see Remark~\ref{rem-less-local}. 
For simplicity of notation, in the rest of the paper, we only consider $M^s_t$ for $t\in \RR$, $t>0$ instead of $t \in \mathbb{C}^*$.

\subsection{Standard Lagrangian model}\label{ss:StandardLag}

There are four tasks to be completed in this subsection, which will be accomplished in the subsequent four sub-subsections, respectively.
Firstly, points on a tropical curve $\gamma$ can lie in different strata of $\partial \Delta$ so we want to enumerate all possibilities and 
describe the neighborhood of points in different strata.
Then, for each point $x \in \gamma$ and a neighborhood $O_x \subset \Delta$ of $x$, we want to isotope $M^s_t \cap \pi^{-1}_{\Delta}(O_x)$
to a standard form for constructing a local Lagrangian in $M^s_t \cap \pi^{-1}_{\Delta}(O_x)$.
After that, we explain how to glue the 
local Lagrangians in $M^s_t \cap \pi^{-1}_{\Delta}(O_{x_1})$ and $M^s_t \cap \pi^{-1}_{\Delta}(O_{x_2})$ when $O_{x_1} \cap O_{x_2} \neq \emptyset$.
Finally, since the local Lagrangians are constructed with respect
to a symplectic corner chart, we will deal with the transition of symplectic corner charts so that all the local Lagrangian models in different symplectic corner charts can be glued together.

In sub-subsections \ref{sss:NbhdPointTrop}, \ref{sss:Local Lagrangian models at points in different strata} and \ref{sss:Gluing local Lagrangians}, we work inside a single symplectic corner chart $U/K$ with moment map image $\coDelta:= \pi_{\Delta}(U/K)$.
There is an induced moment map $\pi_{\wt{\Delta}}:U \to \mathbb{R}^4$ and we denote the image by $\wt{\coDelta}$.
Recall from Subsection \ref{section-orbifold-local-model} that the images $\coDelta$ and $\wt{\coDelta}$ are related by a rational linear affine transformation (in particular, a bijective map)  so there are corresponding subsets $\wt{x}$, $\wt{\gamma}$, $\wt{\shA}$, $\wt{O_x}$, etc in $\wt{\coDelta}$.
On the other hand, subsets in $U/K$ (e.g. $M^s_t \cap U/K$) can be lifted to $K$-invariant sets in $U$ (e.g. $\wt{M}^s_t$) that are compatible with the moment maps.
For the simplicity of notations, we omit all the $\wt{(-)}$ in Subsections \ref{sss:NbhdPointTrop}, \ref{sss:Local Lagrangian models at points in different strata} and \ref{sss:Gluing local Lagrangians} and work $K$-equivariantly in $U$
($U/K$ will also be denoted by $U$). 
By possibly adding a translation, we identify $\wt{\coDelta}$ with an open subset of $\mathbb{R}_{\ge 0}^4$ which contains the origin.

\subsubsection{Neighborhood of a point in a tropical curve}\label{sss:NbhdPointTrop}

We define a function $\Type: \partial \Delta \to \{0,1,2,3\}$ such that $\Type(x)=n$ if $x$ is in the interior of an $n$-cell.
In other words, $\Type$ specifies the stratum that $x$ lies in.
%For a symplectic corner chart $U/K$ and a point $x \in \widetilde{\Delta}$, we define $\Type(x):= \Type((\beta_{U/K}^*)^{-1}(x))$.

%Let $v$ be a vertex of $\Delta$ and $U/G$ be a symplectic corner chart at $v$.
%Let $\coDelta:= \pi_{\Delta}(U/G)$ which is the partial compactification of $\Delta^{\circ}$ by the facets containing $v$ (see Section \ref{section-corner-chart}).
%We have an induced moment map $\pi_{\wt{\Delta}}:U \to \mathbb{R}^4$ and we denote the image by $\wt{\coDelta}$.
%By an affine integral linear transformation, we identify $\wt{\coDelta}$ with an open subset of $\mathbb{R}_{\ge 0}^4$ which contains the origin.

For a neighborhood $O_0 \subset \coDelta \setminus \shA$ of the origin such that $O_0 \cap \partial \coDelta$ is contractible,
the integral affine structure on $O_0 \cap \partial \coDelta$ is inherited from the $C^0$-embedding
$O_0 \cap \partial\coDelta \hookrightarrow \coDelta/\RR (1,1,1,1)$ (see the definition of the chart $W_v$ in Section~\ref{section-affine-charts}). 
%This integral affine structure descends to an integral affine structure in $\Pi_{U/K}(O_0) \cap \partial \Delta$.

Let $x \in \partial \coDelta$ and $c(r):[-1,1] \to \partial \coDelta \backslash \shA$ be a straight line (regarded as a closed segment in a tropical curve $\gamma$ that contains $x$) in a small 
neighborhood of $x$ such that $c(0)=x$.
We have the following situations using that $\partial \Delta$ is simple.

\begin{enumerate}[label=(\Alph*)]
 \item 
If $\Type(x)=0$,
then $(\Type(c(-1)), \Type(c(0)),\Type(c(1)))$ can only take values (modulo the symmetry $r \mapsto -r$) $(2,0,2)$, $(1,0,3)$, $(2,0,3)$ and $(3,0,3)$.
%Explicitly, $\IIm(c)$ is given by (up to symmetry of the coordinate axes)
%\begin{enumerate}[label=A(\roman{*})]
% \item $\{(0,0,r,  r)|0 \le r \le b_1 \} \cup \{(-r,-r,0,0)|-b_0 \le r<0\}$ for some $b_0,b_1\in \RR_{>0}$, or
% \item $\{(0,0,0,  r)|0 \le r \le b_1 \} \cup \{(-r,-r,-r,0)|-b_0 \le r<0\}$ for some  $b_0,b_1\in \RR_{>0}$, or
% \item $\{(0,0,m_1r,  m_2r)|0 \le r \le b_1  \} \cup \{(-m_1r,-m_1r,0,(m_2-m_1)r)|-b_0 \le r<0\}$ for some positive integers $m_1>m_2$ and $b_0,b_1\in \RR_{>0}$, or
% \item $\{(0,m_1r,m_2r,m_3r)|0 \le r \le b_1 \} \cup \{(-m_1r,0,(m_2-m_1)r),(m_3-m_1)r)|-b_0 \le r<0\}$ for some positive integers $m_1>m_2,m_3$ and $b_0,b_1\in \RR_{>0}$
%\end{enumerate}
\item 
If $\Type(x)=1$,
then $(\Type(c(-1)), \Type(c(0)),\Type(c(1)))$ can only take values (modulo the symmetry $r \mapsto -r$) $(1,1,1)$, $(2,1,3)$ and $(3,1,3)$.
%Explicitly, $\IIm(c)$ is given by (up to symmetry of the coordinate axes)
%\begin{enumerate}[label=B(\roman{*})]
% \item $\{(0,0,0,  r)|b_0 \le r \le b_1 \} $ for some $0<b_0<b_1$, or
%\item $\{(0,0,m_1r, m_2r +a)| 0 \le r \le b_1 \} \cup \{(-m_1r,-m_1r,0,(m_2-m_1)r+a)|-b_0 \le r<0\}$ for some $m_1 \in \ZZ_{>0}$, 
%$ m_2 \in \ZZ$ and $a,b_0,b_1\in \RR_{>0}$, or
% \item $\{(0,m_1r,m_2r, m_3r +a)|0 \le r \le b_1\} \cup \{(-m_1r,0,(m_2-m_1)r,(m_3-m_1)r+a)|-b_0 \le r<0\}$ for some $m_1,m_2 \in \ZZ_{>0}$ such that $m_2 <m_1$, 
% $ m_3 \in \ZZ$ and $a,b_0,b_1\in \RR_{>0}$
%\end{enumerate}
\item
If $\Type(x)=2$,
then $(\Type(c(-1)), \Type(c(0)),\Type(c(1)))$ can only take values $(2,2,2)$ and $(3,2,3)$.
%Explicitly, $\IIm(c)$ is given by (up to symmetry of the coordinate axes)
%\begin{enumerate}[label=C(\roman{*})]
% \item $\{(0,0,m_1r+a_1,  m_2r+a_2)|b_0 \le r \le b_1 \} $ or some $m_1,m_2 \in \ZZ$, 
%with $ m_2 \le m_1$ and $a_1,a_2,b_0,b_1\in \RR_{>0}$, or
% \item $\{(0,m_1r,m_2r+a_2, m_3r +a_3)|0 \le r \le b_1  \} \cup \{(-m_1r,0,(m_2-m_1)r+a_2,(m_3-m_1)r+a_3)|-b_0 \le r<0\}$
% for some $m_1 \in \ZZ_{>0}$, $ m_2,m_3 \in \ZZ$ and $a_2,a_3,b_0,b_1\in \RR_{>0}$
%\end{enumerate}
\item
If $\Type(x)=3$, then $\Type(c(r))=3$ for all $r$.
\end{enumerate}

\begin{remark}\label{r:discontinuity}
From the enumeration above, we can see that 
for any straight line \linebreak[4]$c\colon[-1,1] \to \partial \Delta \backslash \shA$, if $r=r_0$ is a discontinuity of $\Type(c(r))$, then 
$\Type(c(r_0))<\Type(c(r))$ for all $r$ close to but not equal to $r_0$.
\end{remark}

\subsubsection{Local Lagrangian models at points in different strata}\label{sss:Local Lagrangian models at points in different strata}

For each point $x \in \gamma$ and each open subset $c \subset \gamma$ containing $x$, we want to isotope 
$M^s_t \cap \pi^{-1}_{\Delta}(O_x)$ to another $K$-invariant hypersurface so that we can build a $K$-invariant Lagrangian in $M^s_t \cap \pi^{-1}_{\Delta}(O_x)$ 
whose $\pi_{\Delta}$-image is close to $c$.
First we describe a class of symplectic manifolds in $\pi^{-1}_{\Delta}(O_x)$ to which we would like to isotope $M^s_t \cap \pi^{-1}_{\Delta}(O_x)$.

\begin{definition}\label{d:xStandard}
For a point $x=(x_1,\dots,x_4) \in (\partial \coDelta) \setminus \shA$ and an $s_1$-admissible section $s \in C^{\infty}(\PP_{\Delta},\shL)$,
we say that $M^s$ is {\bf $x$-standard with respect to $U$} if there is a neighborhood $O_{x} \subset \coDelta$ of $x$
that does not meet any facet that does not contain $x$, and furthermore
such that $ M_t^s \cap \pi_{\Delta}^{-1}(O_{x})$ is given by
\begin{align}\label{eq:SympStandardForm}
   \left(\prod_{j,x_j =0} \sqrt{p_j}\right)e^{i(q_1+q_2+q_3+q_4)}=tc
\end{align}
 for some constant $c \in \CC$.
 If $\Type(x)=0,1,2$, we require $c \neq 0$.
 
 For a point $x \in \partial \Delta \setminus \shA$, we say that $M^s$ is {\bf $x$-standard} if there is a symplectic corner chart $U$ such that $M^s$ is $x$-standard with respect to $U$.
\end{definition}

Since the $K$ action on $U$ acts only on the $q_i$ coordinates and the Gorenstein coordinate $q_1+q_2+q_3+q_4$ is invariant under the $K$ action, $M^s$ is $K$-invariant if $M^s$ is $x$-standard with respect to $U$.
To see that this is a sensible notion, we at least need to observe the following:

\begin{lemma}\label{l:noDisc}
 If $M^s$ is $x$-standard with respect to $U$, then $ (M_{t \neq 0}^s \cap \pi_{\Delta}^{-1}(O_{x})) \cap (\partial \mathbb{P}_{\Delta})_{\Sing}=\emptyset$.
 In other words, $M_t^s \cap \pi_{\Delta}^{-1}(O_{x})$ is disjoint from the discriminant for all $t >0$.
\end{lemma}

\begin{proof}

Notice that, we can rewrite Equation \eqref{eq:SympStandardForm} as
\begin{align*}
 z_1z_2z_3z_4=\sqrt{16p_1p_2p_3p_4}e^{i(q_1+q_2+q_3+q_4)}=4tc\prod_{j,x_j \neq 0} \sqrt{p_j}.
\end{align*}
To prove the lemma, it suffices to show that the zero locus of $4tc\prod_{j,x_j \neq 0} \sqrt{p_j}$ does not intersect with 
$(\partial \mathbb{P}_\Delta)_{\Sing} \cap \pi_{\Delta}^{-1}(O_{x})$.
When $\Type(x)=0,1,2$, by Definition \ref{d:xStandard}, we have $c \neq 0$.
Moreover, inside $\pi_{\Delta}^{-1}(O_{x})$, we have $p_j >0$ when $x_j \neq 0$.
All together implies that $4tc\prod_{j,x_j \neq 0} \sqrt{p_j}$ never vanishes in $\pi_{\Delta}^{-1}(O_{x})$.

Since $(\partial \mathbb{P}_\Delta)_{Sing}=\pi^{-1}_{\Delta}(\hbox{codim-2-strata})$, when $\Type(x)=3$, 
$(\partial \mathbb{P}_\Delta)_{\Sing} \cap \pi_{\Delta}^{-1}(O_{x})= \emptyset$.
\end{proof}

The next lemma addresses that we can always isotope $ M_t^s \cap \pi_{\Delta}^{-1}(O_{x})$ to an $x$-standard one through admissible sections that are $K$-invariant in $U$.

\begin{lemma}\label{l:ExistenceStandardSymp}
Let $s$ be an  $s_1$-admissible section.
 Let $x \in \partial \Delta  \backslash \shA$ be a point and $N_{x}$ be a neighborhood of $x$ in $\Delta$ such that $N_{x} \cap \shA =\emptyset$.
 %such that $N_x \cap \shA= \emptyset$.
 Then there is a symplectic corner chart $U$ containing $x$ and a family of $s_1$-admissible section $(s^u)_{u \in [0,1]}$ such that $s^0=s$, for all $u$,
  $s^u=s$ outside $\pi_{\Delta}^{-1}(N_{x})$, $s^u$ is $K$-invariant in $U$ 
 and $M_t^{s^1}$ is $x$-standard with respect to $U$.
 %\footnote{Simplicity of $(\partial \Delta, \shA)$ is used. We may add this into the statement.}
\end{lemma}

\begin{proof}
 If $\Type(x)=0$, then $x$ is a vertex and we take the symplectic corner chart $U$ associated to $x$.
 Since $x \notin \shA$, there exists a neighborhood $O_x \subset N_x$ of $x$ such that  $\pi_{\Delta}^{-1}(O_{x})$ is contractible
 and, by \eqref{eq-zf}, $M_t^{s} \cap \pi_{\Delta}^{-1}(O_{x})$ is given by
 \begin{align}
   z_1z_2z_3z_4=\sqrt{16p_1p_2p_3p_4}e^{i(q_1+q_2+q_3+q_4)}=tf \label{eq:CornerForm}
\end{align}
 for some $f \in C^{\infty}(\pi_{\Delta}^{-1}(O_{x}),\CC^*)$.
 Since $\pi_{\Delta}^{-1}(O_{x})$ is contractible, $f$ is null-homotopic.
 For any subset $O \subset O_x \setminus \{x\}$, since $f$ is null-homotopic, we know that $f|_{\pi_{\Delta}^{-1}(O)}$ 
is null-homotopic and it descends to a null-homotopic function in the quotient by $K$.
Therefore, for any neighborhood $O_{x}' \subset O_x$  of $x$, we can deform $f$, through $K$-invariant non-vanishing functions inside $\pi_{\Delta}^{-1}(O_{x})$,
 to a function which is constant in $\pi_{\Delta}^{-1}(O_{x}')$. 
 Moreover, the deformation can be chosen to be compactly supported.
 There is no new discriminant created during the deformation because it is through non-vanishing functions (cf. Lemma \ref{l:noDisc}).
 The deformation is constant near the discriminant $\shA$ because $N_x \cap \shA= \emptyset$.
 Since the deformation is compactly supported, it can patch with $f$ outside a compact set in $\pi_{\Delta}^{-1}(O_{x})$
 to give a family of $K$-invariant $s_1$-admissible sections with required properties.
 
 If $\Type(x)=1$, let $\delta$ be the $1$-cell in $\Delta$ containing $x$.
 By simplicity of $(\partial \Delta, \shA)$ (see introduction and \cite{GrossSiebert06}, Definition 1.60), there is a vertex $v$ in $\delta$ which can be connected to $x$ by a path $\delta'$ in $\delta$ that does not intersect with  $\shA$.
 Let $U$ be  the symplectic corner chart associated to $v$.
 Without loss of generality, we assume $x_1 \neq 0$ and $x_j =0$ for $j=2,3,4$.
 %, where $(\beta_{U/K}(x))_j$ is the $j^{th}$-coordinate of $\beta_{U/K}(x)$.
 Notice that $\delta' \cap \shA= \emptyset$ implies that there exists a neighborhood $N_{\delta'} \subset \partial \Delta$ of $\delta'$
 such that $\pi_{\Delta}^{-1}(N_{\delta'}) \cap \Disc(s_1)=\emptyset$, $\pi_{\Delta}^{-1}(N_{\delta'})$
 is contractible and $ \pi_{\Delta}^{-1}(N_{\delta'}) \cap M_t^{s}$ is given by
Equation \eqref{eq:CornerForm}
for some $f \in C^{\infty}(\pi_{\Delta}^{-1}(N_{\delta'}), \CC^*)$.
It implies that for a neighborhood $O_x \subset N_x \cap N_{\delta'}$ of $x$, $f|_{\pi_{\Delta}^{-1}(O_x)}$ is null-homotopic 
even though $\pi_{\Delta}^{-1}(O_x)$ is not contractible. In other words,
\begin{align}\label{eq:homotopyClass}
 (f|_{\pi_{\Delta}^{-1}(O_x)})_*:\pi_1(\pi_{\Delta}^{-1}(O_x)) \to \pi_1(\mathbb{C}^*)=\mathbb{Z} \text{ is zero}
\end{align}
and the same is true when $f|_{\pi_{\Delta}^{-1}(O_x)}$ is descended to the quotient by $K$.
On the other hand, for $O_x$ not containing $v$, $p_1>0$ in $\pi_{\Delta}^{-1}(O_x)$, 
so it gives a map  $\sqrt{p_1}|_{\pi_{\Delta}^{-1}(O_x)} \to \mathbb{R}_{>0} \subset \mathbb{C}^*$.
Moreover, $(\sqrt{p_1})_*:\pi_1(\pi_{\Delta}^{-1}(O_x)) \to \pi_1(\mathbb{C}^*)$ is also zero.
Therefore, there is no topological obstruction to deform $f|_{\pi_{\Delta}^{-1}(O_x)}$ to $\sqrt{p_1}|_{\pi_{\Delta}^{-1}(O_x)}$ inside $\pi_{\Delta}^{-1}(O_x)$ via $K$-invariant $\mathbb{C}^*$-valued functions.
Most notably, $\mathbb{C}^*$-valued functions are non-vanishing functions.
Similar to the previous case, we can assume the deformation is compactly supported and it gives a family of $s_1$-admissible sections with required properties by patching with 
$f$ outside $\pi_{\Delta}^{-1}(O_x)$.
 
If $\Type(x)=2$, we use simplicity of $(\partial \Delta, \shA)$ again to find a vertex $v$ and a path 
such that it lies inside a $2$-cell of $\partial \Delta$, connects $v$ and $x$, and does not intersect with $ \shA$.
Let $U$ be  the symplectic corner chart associated to $v$.
The equation of $M_t^{s}$ is again locally given by Equation \eqref{eq:CornerForm} for some $f$.
Moreover, $f$ is again null-homotopic.
If $x_1,x_2 \neq 0$ and $x_3=x_4=0$, we can deform $f$ to $\sqrt{p_1p_2}$ inside 
$\pi_{\Delta}^{-1}(O_x)$ for some small neighborhood $O_x$ of $x$.
It gives our desired family of $s_1$-admissible sections as in the previous case.

If $\Type(x)=3$, then we can take $O_x$ such that it does not intersect $0,1,2$-cells of $\Delta$.
Therefore, we can do any compactly supported deformation of the corresponding $f$ without creating/destroying discriminant loci 
(i.e.~we allow deformation of $f$ via functions that vanish somewhere). It is instructive to compare it with the proof of Lemma \ref{l:noDisc}.
The outcome is: the lemma is trivially true when $\Type(x)=3$.
\end{proof}

We are now ready to give the local Lagrangians in $ M_t^s \cap \pi_{\Delta}^{-1}(O_{x})$ when $M_t^s $ is $x$-standard.

\begin{proposition}[Standard Lagrangian model]\label{p:standardLagModel}
 Let $M^s \subset U$  be $x$-standard for some $x \in \partial \coDelta$ and $O_x$ be a neighborhood of $x$ such that \eqref{eq:SympStandardForm} holds.
 Let $W$ be a rationally generated $2$-dimensional affine plane in $\RR^4$ containing $x$ and $(1,1,1,1) \in TW$. 
 Let $c:=W \cap  O_x \cap \partial \coDelta$ (regarded as a straight line segment in $\partial \coDelta$).
 %Let $c:(-1,1) \to \partial \coDelta$ be a proper straight line in $O_{x} \cap \partial \coDelta$ such that $c(0)=x$.
 %Let $T_c:=\{\alpha \in \RR^4| \langle \alpha, c'(r) \rangle=\langle \alpha, (1,1,1,1)\rangle=0\}$, which is necessarily two dimensional because $c'(r)$ is not a multiple of $(1,1,1,1)$.
 %Let $\{v_a=(a_1,a_2,a_3,a_4), v_b=(b_1,b_2,b_3,b_4)\}$ be a basis of $T_c$.
 Then there is a family of proper $K$-invariant (possibly disconnected) Lagrangian submanifolds $L_t$ in $\pi_{\Delta}^{-1}(O_{x}) \cap M_t^s $, for $t > 0$, such that
 \begin{enumerate}[label=(\Roman{*})]
  \item $W^\perp \subset T_{(p,q)}L_t$ for all $(p,q) \in L_t$, and
  \item $\pi_{\Delta}(L_t) \subset W $.
 \end{enumerate}
Moreover, every family of proper $K$-invariant Lagrangian submanifolds $L'_t$ in $\pi_{\Delta}^{-1}(O_{x}) \cap M_t^s$ satisfying $(I),(II)$
can be given one of the following parametrizations (either Case A or Case B):

Let $W^\perp_T$ be the quotient of $W^\perp$ by the lattice $W^\perp \cap (2\pi\ZZ)^4$.
Under the natural identification between $W^\perp_T$ and the $T^2$ subgroup of $(\RR/2\pi\ZZ)^4$ in the $q$-variables, the cyclic group $K$ is either contained in $W^\perp_T$ or $K \cap W^\perp_T=\{0\}$.

{\bf Case A.} If $K$ is contained in $W^\perp_T$, then $L_t$ is connected and there
exists an $\RR^4$-valued function $P(r,t)$ and 
an $(\RR/2\pi\ZZ)^4$-valued function $Q(r,\theta_1,\theta_2,t)$, for $(r,\theta_1,\theta_2) \in (0,1) \times (\RR/2\pi\ZZ)^2$, parametrizing $L_t$.
%every proper Lagrangian submanifold in $\pi_{\Delta}^{-1}(O_{x}) \cap M_t^s $ satisfying $(I),(II)$
In this case, $L'_t$ is given by 
\begin{align}
 \{(p,q) \in \pi_{\Delta}^{-1}(O_{x}) \cap M_t^s | p=P(r,t),q=Q(r,\theta_1,\theta_2,t)+H(r)\} \label{eq:xStandardL}
\end{align}
for some $H(r) \in C^{\infty}((0,1),(\RR/2\pi\ZZ)^4)$ satisfying $\sum_j H_j(r)=0$ for all $r\in (0,1)$, where $H_j$ is the $j^{th}$ component of $H$.

{\bf Case B.} If $K \cap W^\perp_T=\{0\}$, then $L_t$ has $|K|$ connected components and there exists $P(r,t)$ and $Q(r,\theta_1,\theta_2,t)$
as above parametrizing one of the connected components so that the other connected components are parametrized by 
$p=P(r,t)$ and $q=Q(r,\theta_1,\theta_2,t)+\kappa$ for $\kappa \in K$.
%Moreover, every proper $K$-invariant Lagrangian submanifold in $\pi_{\Delta}^{-1}(O_{x}) \cap M_t^s $ satisfying $(I),(II)$
In this case, one of the components of $L'_t$ can be parametrized by  $p=P(r,t)$ and $q=Q(r,\theta_1,\theta_2,t)+H(r)$ for some $H(r)$ as above and the other components are obtained by adding $\kappa \in K$ in the $q$ coordinates.

Furthermore, in either Case A or Case B, the family of $\{P(r,t)|r \in (0,1)\}$ Hausdorff converges to $c$ when $t$ approaches $0$.
\end{proposition}

\begin{definition} \label{def-Lag-standard}
A proper Lagrangian submanifold $L_t$ in $\pi_{\Delta}^{-1}(O_{x}) \cap M_t^s $ satisfying Proposition \ref{p:standardLagModel} $(I)$, $(II)$
is called {\bf $c$-standard}.
\end{definition}
Before giving the proof, it would be helpful to have an intuitive understanding of what $L_t$ looks like.
For fixed $(r,t)$, $\{Q(r,\theta,\theta_2,t)| \theta, \theta_2 \in \RR/2\pi\ZZ\}$ is a $2$-torus lying inside $L_t$ with 
$p$-coordinates being $p=P(r,t)$ so when $K$ is contained in $W^\perp_T$, $L_t$ is a $2$-torus bundle over the curve $\{p=P(r,t)\}$ and when $K \cap W^\perp_T=\{0\}$, $L_t$ has $|K|$ connected components and each of them is a $2$-torus bundle over the curve.
Moreover, condition $(I)$ describes the tangent directions of the $2$-torus. 
Condition $(II)$ implies that the curve $\{p=P(r,t)\}$ is a subset of $W \cap O_x$, which Hausdorff converges to $c$ when $t$ approaches to $0$.

Also note that the function $H(r)$ in \eqref{eq:xStandardL} plays exactly the same role as $f$ in Remark~\ref{r:absorbAlpha}.

\begin{proof}
 We have enumerated the possibilities of $c$ in Section \ref{sss:NbhdPointTrop}.
 Existence of $L_t$ is a simple case by case calculation.
 %We are going to consider the cases of type $A$ and leave the remainings to the readers.
 %Without loss of generality, we assume $t>0$.

 For cases of type $A$, we have $x=(0,0,0,0)$ and 
 $\pi_{\Delta}^{-1}(O_{x}) \cap M_t^s $ is given by
\begin{align*}
   (\sqrt{p_1p_2p_3p_4})e^{i(q_1+q_2+q_3+q_4)}=tc_pe^{ic_q}
\end{align*}
for some constants $c_p>0$ and $c_q \in \RR/2\pi \ZZ$.
%Let $v_a=(m_1-m_2)\partial_{q_1}+m_2 \partial_{q_2}-m_1 \partial_{q_3}$
%and $v_b=(m_1-m_3) \partial_{q_1}+m_3 \partial_{q_2}-m_1 \partial_{q_4}$.
%For $u_1,u_2 \ge 0$, let $p_1=m_1u_2$, $p_2=m_1u_1$, $p_3=m_2u_1+(m_1-m_2)u_2$ and $p_4=m_3u_1+(m_1-m_3)u_2$ so that we have
%$\sum_j a_jp_j=\sum_j b_jp_j=0$.
%Since $p_1p_2p_3p_4=t^2c_p^2$ is a hyperbola, one can see (or show by direct calculation) that for each $t>0$, 
In particular, a point $(p,q) \in \pi_{\Delta}^{-1}(O_{x}) \cap M_t^s$ has to satisfies $p_1p_2p_3p_4=t^2c_p^2$.
Notice that, for each fixed $t>0$, $H_t:=\{p_1p_2p_3p_4=t^2c_p^2\}$ is a hyperbola so $H_t \cap  W \cap \coDelta$ is a smoothly embedded curve.
More rigorously, let $g:\RR_{\ge0}^4 \to \RR_{\ge 0}$ be $g(p)=p_1p_2p_3p_4$.
For each $p \in  W \cap (\partial \coDelta)$, the ray $\{p+\lambda(1,1,1,1)| \lambda \ge 0\}$ lies in $W$
and the function $g_p(\lambda):=g(p+\lambda(1,1,1,1))$ is a {\bf strictly} monotonic increasing function on the ray
because, for $\nu:=\sum_{j=1}^4 \partial_{p_j}$, we have $\nu(g)>0$ over $\RR_{>0}^4$.
Since $g_p(0)=0$, for each fixed $t>0$, there is exactly one $\lambda>0$ such that $g_p(\lambda)=t^2c_p^2$.
It means that for each $p \in W \cap \partial \coDelta$, there is at most one $p' \in W \cap \coDelta$
such that $p' \in H_t$ and $p'=p+\lambda(1,1,1,1)$ for some $\lambda>0$.
Since $W \cap \partial \coDelta$ is a continuous curve, $H_t \cap  W \cap \coDelta$ is a smoothing of it and hence a smoothly embedded curve.
We define $P_t := H_t \cap  W \cap O_x$, which is smooth because it is an open subset of a smooth curve.
%We call this family of curves $P(u,t)$.
It is clear that $P_t$ Hausdorff converges to $c$ when $t$ goes to $0$.

For each $t>0$ and $p \in P_t$, we can pick $2$-tori $Q_{p,t} \subset \pi^{-1}_{\Delta}(p)$ such that $Q_{p,t}$ varies smoothly with respect to $p$,
$Q_{p,t}$ is parallel to $W^{\perp}$ and 
$e^{i(q_1+q_2+q_3+q_4)}=e^{ic_q}$ for all $q \in Q_{p,t}$.
This family of $2$-tori gives a submanifold $L_t \subset \pi_{\Delta}^{-1}(O_{x}) \cap M_t^s$.
The fact that $P_t \subset W$ and $TQ_{p,t} = W^\perp$ for all $p \in P_t$ implies that  $L_t$ is a Lagrangian submanifold.

When $K \subset W^\perp_T$, it is easy to see that \eqref{eq:xStandardL} gives all proper $K$-invariant Lagrangian satisfying $(I),(II)$. 

On the other hand, when $K \cap W^\perp_T=\{0\}$, we replace $L_t$ by its $K$-orbit. It is also easy to see that any other proper $K$-invariant  Lagrangian satisfying $(I),(II)$ is given by adding a function $H(r)$ to the $q$-coordinates of all the components simultaneously.

%\begin{align}
% (m_1u_2)(m_1u_1)(m_2u_1+(m_1-m_2)u_2)(m_3u_1+(m_1-m_3)u_2)=t^2(c_p)^2 \label{eq:findP}
%\end{align}
%For a fixed $u_1$, there is a unique solution of $u_2$ in terms of $t$ for Equation \eqref{eq:findP}.
%Notice that, when $t$ approaches to $0$, $(p_1,p_2,p_3,p_4)$ approaches to $(0,m_1u_1,m_2u_1, m_3u_1)$, which is a subset of $\IIm(c)$.
%Let the solution be $u_2=S(u_1,t)$.

%We have a proper Lagrangian in $\pi_{\Delta}^{-1}(O_{x}) \cap M_t^s$ given by
%\begin{align*}
%L_t:=&\{(p,q) |p_1= m_1S(u_1,t), p_2=m_1u_1, p_3=m_2u_1+(m_1-m_2)S(u_1,t), \\
%&p_4=m_3u_1+(m_1-m_3)S(u_1,t), q_1=(m_1-m_2)\theta+(m_1-m_3)\theta_2+c_q, \\
%&q_2=m_2\theta+m_3\theta_2, q_3=-m_1\theta, q_4=-m_1\theta_2, \theta, \theta_2 \in \RR/2\pi\ZZ, u_1>0\} 
%\end{align*}
%Similarly, for a fixed $u_2$, there is a unique solution of $u_1$ in terms of $t$ for Equation \eqref{eq:findP}.
%This time, when $t$ approaches to $0$, $(p_1,p_2,p_3,p_4)$ approaches to $(m_1u_2,0,(m_1-m_2)u_2, (m_1-m_3)u_2)$, which is also a subset of $\IIm(c)$.

%If we let the solution be $u_1=S(u_2,t)$, then we can rewrite $L_t$ as
%\begin{align*}
%L_t:=&\{(p,q) |p_1= m_1u_2, p_2=m_1S(u_2,t), p_3=m_2S(u_2,t)+(m_1-m_2)u_2, \\
%&p_4=m_3S(u_2,t)+(m_1-m_3)u_2, q_1=(m_1-m_2)\theta+(m_1-m_3)\theta_2+c_q, q_2=m_2\theta+m_3\theta_2, \\ 
%&q_3=-m_1\theta, q_4=-m_1\theta_2, \theta, \theta_2 \in \RR/2\pi\ZZ, u>0\} 
%\end{align*}

For cases of type $B$, we have $x=(0,0,0,a)$ for some $a>0$ and we 
need to consider the set of $p \in W \cap O_x$ that solves $\sqrt{p_1p_2p_3}=tc_p$.
This time, we can take $g(p)=p_1p_2p_3$ for $p \in \RR_{\ge 0}^4$ and 
$g_p(\lambda)=g(p+\lambda(1,1,1,1))$ for $p \in W \cap \{p_1p_2p_3=0\}$, and $\lambda \ge 0$.
Let $\nu=\partial_{p_1}+\partial_{p_2}+\partial_{p_3}+\partial_{p_4}$ and we have $\nu(g)>0$ over $\RR_{> 0}^4$.
Similar to the previous case, it means that for each $p \in W \cap \{p_1p_2p_3=0\}$, there is at most one $p' \in W \cap \coDelta$
such that $p' \in H_t$ and $p'=p+\lambda(1,1,1,1)$ for some $\lambda>0$.
The rest of the argument is the same.

For cases of type $C$ or $\Type(x)=3$, we need to consider $p \in W \cap O_x$
that solves $\sqrt{p_1p_2}=tc_p$ and $\sqrt{p_1}=tc_p$, respectively.
The rest of the argument is the same.

\begin{comment}
For case $B(iii)$ , we have $x=(0,0,0,a)$ and  $\pi_{\Delta}^{-1}(O_{x}) \cap M_t^s $ is given by
\begin{align*}
   (\sqrt{p_1p_2p_3})e^{i(q_1+q_2+q_3+q_4)}=tc_pe^{ic_q}
\end{align*}
for some constants $c_p>0$ and $c_q \in \RR/2\pi \ZZ$
We can take the same $v_a,v_b$.
For $u_1,u_2 \ge 0$, let $p_1=m_1u_2$, $p_2=m_1u_1$, $p_3=m_2u_1+(m_1-m_2)u_2$ and $p_4=m_3u_1+(m_1-m_3)u_2+a$ so that we have
$\sum_j a_j(p_j-x_j)=\sum_j b_j(p_j-x_j)=0$.
The equation $\sqrt{p_1p_2p_3}=tc_p$ becomes
\begin{align}
 (m_1u_2)(m_1u_1)(m_2u_1+(m_1-m_2)u_2)=t^2(c_p)^2 \label{eq:findP2}
\end{align}
For a fixed $u_1$, there is a unique solution of $u_2$ in terms of $t$ for Equation \eqref{eq:findP2}.
Let the solution be $u_2=S(u_1,t)$.
We have a proper Lagrangian in $\pi_{\Delta}^{-1}(O_{x}) \cap M_t^s$ given by
\begin{align*}
L_t:=&\{(p,q) |p_1= m_1S(u_1,t), p_2=m_1u_1, p_3=m_2u_1+(m_1-m_2)S(u_1,t), \\
&p_4=m_3u_1+(m_1-m_3)S(u_1,t)+a, q_1=(m_1-m_2)\theta+(m_1-m_3)\theta_2+c_q, \\ 
&q_2=m_2\theta+m_3\theta_2, q_3=-m_1\theta, q_4=-m_1\theta_2, \theta, \theta_2 \in \RR/2\pi\ZZ, u_1>0\} 
\end{align*}
\end{comment}

%Now consider another Lagrangian $L$ satisfying $(I),(II)$.
% It is clear from the construction that the $p$-coordinates of $L$ are determined.
% Therefore, $L$ is a torus bundle over $\{P(u,t)|u \in (0,1)\}$ with respect to $\pi_{\Delta}|_L$ and the
% torus fibers are tangential to $W^\perp$.
% We have a freedom to parallel translate the torus fibers fiberwise, which is exactly given by adding the function $H(u)$.
\end{proof}

\subsubsection{Gluing local Lagrangians}\label{sss:Gluing local Lagrangians}

In the previous sub-subsection, we explained how to construct local Lagrangian when $M^s$ is $x$-standard.
Now, suppose $c:[-1,1] \to \partial \coDelta$ (again, $\im(c)$ is regarded as a closed segment of a tropical curve $\gamma$)
has the property that $\Type(c(r))$ is discontinuous at $r=0$ and $M^s$  is $c(0)$-standard with respect to $U$.
Then $M^s$  is not $c(r)$ standard with respect to $U$ for any $r$ close to but not equal to $0$.
Therefore, we need to generalize Proposition \ref{p:standardLagModel} and explain how to glue the local Lagrangian models together.

\begin{definition}\label{d:xTransitionStandard}
%\footnote{Picture for $f_p$}
Let $U$ be a symplectic corner chart and $\coDelta=\pi_{\Delta}(U)$.
Let $c^\circ \subset  \partial \coDelta \setminus \shA$ be an open straight line segment. 
Let $c:[0,1] \to \partial \coDelta \setminus \shA$ be a straight  line such that
$c((0,1))=c^{\circ}$ and $\Type(c(0)) \le \Type(c(1))=\Type(c(r))$ for $r \in (0,1]$.
Given an admissible section $s \in C^{\infty}(\PP_{\Delta},\shL)$,
we say that $M^s$ is {\bf $c^\circ$-transition-standard with respect to $U$} if
$M^s$ is $c(0)$-standard and $c(1)$-standard with respect to $U$, and there is 
%a neighborhood $O_{c(0)} \subset \coDelta$ of $c(0)$,
%a neighborhood $O_{c(1)} \subset \coDelta$ of $c(1)$ and
a neighborhood $O_{c^{\circ}} \subset \coDelta \setminus \shA$ of $c^{\circ}$ such that $c^{\circ}$ is proper inside $O_{c^{\circ}}$ and
$\pi_{\Delta}^{-1}(O_{c^{\circ}}) \cap M_t^s $ is given by
\begin{align}
   \left(\prod_{j,c(0)_j =0} \sqrt{p_j}\right)e^{i(q_1+q_2+q_3+q_4)}=tf_p\left(\prod_{j,c(0)_j =0, c(1)_j \neq 0} p_j\right)e^{if_q(p)} \label{eq:TranStandardSymp}
\end{align}
 for some function $f_q \in C^{\infty}(U,\RR/2\pi\ZZ)$ depending only on $p$ (in particular, $K$-invariant), and some $f_p \in C^{\infty}(\RR_{\ge 0},\RR_{>0})$ is 
 such that $\frac{u}{f_p^2(u)}$ is a monotonic increasing function 
 and $f_p$ is an interpolation from  $c_p$ to  $c_p'\sqrt{u}$ for some constants $c_p,c_p'>0$.
 In \eqref{eq:TranStandardSymp}, $c(k)_j$ is the $p_j$-coordinate of $c(k)$ for $k=0,1$, and, whenever $\Type(c(0))=\Type(c(1))$,
 $\prod_{j,c(0)_j =0, c(1)_j \neq 0} p_j$ (which is a product over the empty set) is interpreted as $1$.
 
 We say that $M^s$ is {\bf $c^\circ$-transition-standard} if $M^s$ is $c^\circ$-transition-standard with respect to some symplectic corner chart.
% For a point $x \in \coDelta$, we call $M^s$ {\bf $x$-transition-standard with respect to $U$} if $M^s$ is $c$-transition-standard with respect to $U$ for some $c$ such that
% $x \in \IIm(c)$.
% For a point $x \in \Delta$, we call $M^s$ $x$-transition-standard if there is a symplectic corner chart $U$ such that $M_t^s$ is 
% $x$-transition-standard with respect to $U$.
\end{definition}

\begin{remark}
Note, for $M^s$ to be $c(0)$-standard and $c(1)$-standard, simultaneously, it is necessary for $f_p$ to be an interpolation from $c_p$ to $c_p'\sqrt{u}$.
 The monotonicity of $\frac{u}{f_p^2(u)}$ is imposed to achieve Lemma \ref{l:standardTransitionModel} below.
\end{remark}

\begin{lemma}\label{l:ExistenceTranStandardSymp}
Let $s$ be an  $s_1$-admissible section.
 Let $c:[0,1] \to \partial \coDelta \backslash \shA$ be a straight line such that $\Type(c(0)) \le \Type(c(1))=\Type(c(r))$ for all $r \in (0,1]$.
 Let $c^{\circ}:=c((0,1))$ and $N_{c}$ be a neighborhood of $\im(c)$ in $\coDelta \backslash \shA$.
 Then there is a symplectic corner chart $U$ and a family of $s_1$-admissible section $(s^u)_{u \in [0,1]}$ such that $s^0=s$, for all $u$,
  $s^u=s$ outside $\pi_{\Delta}^{-1}(N_{c})$,  $s^u$ is $K$-invariant in $U$,
 and $M^{s^1}$ is $c^{\circ}$-transition-standard with respect to $U$.
 %\footnote{\red{ Simplicity of $(\partial \Delta, \shA)$ is used. We may add this into the statement.}}
\end{lemma}

\begin{proof}
The proof is in parallel to Lemma \ref{l:ExistenceStandardSymp}.
We give the details when $\Type(c(0))=1 < \Type(c(1))$ and leave the remaining to the readers.

Let $x=c(0)$.
We pick a vertex $v$ and the corresponding symplectic corner chart $U$ as in the proof of  Lemma \ref{l:ExistenceStandardSymp}.
We can find a neighborhood $O_c$ of $\im(c)$ such that
$ M_t^s \cap \pi_{\Delta}^{-1}(O_c)$ is given by \eqref{eq:CornerForm}
for some $f \in C^{\infty}(\pi_{\Delta}^{-1}(O_c),\CC^*)$ and $(f)_*:\pi_1(\pi_{\Delta}^{-1}(O_c)) \to \pi_1(\CC^*)$ is the zero map (see \eqref{eq:homotopyClass}).
Say $x_j \neq 0$ exactly when $j=1$ and $c(1)_j \neq 0$ exactly when $j=1, \dots, n_c$ (here $n_c \in \{2,3\}$).
Let $g(r)=\prod_{j,c(0)_j =0, c(1)_j \neq 0} c(r)_j=c(r)_2 \dots c(r)_{n_c}$, where $c(r)_j$ is the $j^{th}$-coordinate of $c(r)$ for $r \in [0,1]$.
Note that, $g(0)=0$ and $g(r)$ is strictly increasing.

Inside $\pi_{\Delta}^{-1}(O_c)$, there is no topological obstruction to deform $f$ through $K$-invariant non-vanishing functions, 
to $\sqrt{p_1} f_p(p_2 \dots p_{n_c})e^{if_q(p)}$ for some $f_p\in C^{\infty}(\RR_{\ge 0},\RR_{>0})$  such that $\frac{u}{f_p^2(u)}$ is a monotonic increasing
 function, and there are constants $c_p,c_p'>0$ such that $f_p(u)=c_p$ near $u=0$
 and $f_p(u)=c_p'\sqrt{u}$ near $u=g(1)$. 
The conditions on $f_p$ near $u=0$ and $u=g(1)$ imply that $M^s$ is $c(0)$-standard and $c(1)$-standard simultaneously.
 Moreover, there exists $O_{c^{\circ}} \subset O_c$ such that $c^\circ$ is proper inside $O_{c^{\circ}}$ and 
 $\pi_{\Delta}^{-1}(O_{c^{\circ}}) \cap M_t^s $ satisfies \eqref{eq:TranStandardSymp}.
 Therefore, the result follows.

\end{proof}

A simple but crucial observation is that we can extend the `standard region' by a further isotopy without destroying the previously established standard region in the following sense.

\begin{lemma}\label{l:gluing2lines}
Let $c_1,c_2:[0,1] \to \partial \coDelta \backslash \shA$ be two straight lines
as in Lemma \ref{l:ExistenceTranStandardSymp} such that $c_1(0)=c_2(0)$ or $c_1(0)=c_2(1)$ or $c_1(1)=c_2(1)$.
Suppose we have applied Lemma \ref{l:ExistenceTranStandardSymp} to $c_1$ and denote the resulting $s^1$ as $s$.
Let $N_2$ be a neighborhood of $\im(c_2)$ in $\coDelta \backslash \shA$.
Then there is a family of $s_1$-admissible sections  $(s^u)_{u \in [0,1]}$ such that $s^0=s$, for all $u$,
  $s^u=s$ outside $\pi_{\Delta}^{-1}(N_2)$,  $s^u$ is $K$-invariant in $U$,
 and $M^{s^1}$ is simultaneously $c_1^{\circ}$-transition-standard 
and $c_2^{\circ}$-transition-standard with respect to $U$.
\end{lemma}

\begin{proof}
We want to apply (the proof) of Lemma \ref{l:ExistenceTranStandardSymp} to $c_2$.
The key point is that, inside $\pi_{\Delta}^{-1}(O_{c_2})$, there is 
 no topological obstruction to deform $f$ through $K$-invariant non-vanishing functions to a function as in Lemma \ref{l:ExistenceTranStandardSymp}, and in addition to that, we are free to choose the deformation to be trivial 
inside $\pi_{\Delta}^{-1}(O_{c_1})$ for some small neighborhood $O_{c_1}$ of $\im(c_1)$.
In this case, the corresponding section $s^1$ will make 
$M^{s^1}$ simultaneously $c_1^{\circ}$-transition-standard 
and $c_2^{\circ}$-transition-standard.
\end{proof}

\begin{lemma}\label{l:standardTransitionModel}
 Let $M^s$  be $c^{\circ}$-transition-standard with respect to $U$. 
 Let $O_{c^{\circ}} \subset \coDelta$ be a neighborhood  of $c^\circ$ such that \eqref{eq:TranStandardSymp} holds.
 Let $W$ be the rationally generated $2$-dimensional plane in $\RR^4$ that contains $c^\circ$ and $(1,1,1,1) \in TW$.
 Then there exists a family of proper $K$-invariant (possibly disconnected) Lagrangian submanifold $L_t$ in $\pi_{\Delta}^{-1}(O_{c^\circ}) \cap M_t^s $, for $t>0$, such that
 \begin{enumerate}[label=(\Roman{*})]
  \item $W^\perp \subset T_{(p,q)}L_t$ for all $(p,q) \in L_t$, and
  \item $\pi_\Delta(L_t) \subset W$.
 \end{enumerate}
 %where $a_j,b_j$ are defined as in Proposition \ref{p:standardLagModel}
 Moreover, $\pi_{\Delta}(L_t)$ Hausdorff converges to $c^\circ$ when $t$ approaches $0$.
\end{lemma}

\begin{proof}
 Similar to Proposition \ref{p:standardLagModel}, since $f_q$ only depends on $p$, it suffices to show that for each $t>0$,
 the set of $p \in W \cap O_{c^\circ}$ that solves
 \begin{align}
  \prod_{j,c(0)_j =0} \sqrt{p_j}=tf_p\left(\prod_{j,c(0)_j =0, c(1)_j \neq 0} p_j\right) \label{eq:EmbeddedCurve}
 \end{align}
 is an open subset of a smoothly embedded curve.
 
 We only consider case that $\Type(c(0))=0$ and $\Type(c(1))=3$. The other cases can be dealt similarly.
 In this case
 $\pi_{\Delta}^{-1}(O_{c}) \cap M_t^s $ is given by
\begin{align*}
   (\sqrt{p_1p_2p_3p_4})e^{i(q_1+q_2+q_3+q_4)}=tf_p(p_2p_3p_4)e^{if_q(p)}
\end{align*}
Notice that $\nu:=\sum_{j=1}^4 \partial_{p_j}$ satisfies $\nu(\frac{p_1p_2p_3p_4}{f_p^2(p_2p_3p_4)})>0$
for all $p \in \coDelta \backslash \partial \coDelta$ because we assumed that $\frac{u}{f_p^2(u)}$ is monotonic increasing
and $\partial_{p_1}(\frac{p_1p_2p_3p_4}{f_p^2(p_2p_3p_4)})= \frac{p_2p_3p_4}{f_p^2(p_2p_3p_4)}>0 $ for all $p \in \coDelta \backslash \partial \coDelta$.
The rest of the argument is the same.
\end{proof}

\begin{definition} \label{def-Lag-tran-standard}
A proper $K$-invariant Lagrangian submanifold $L_t$ in $\pi_{\Delta}^{-1}(O_{c^{\circ}}) \cap M_t^s $ satisfying Lemma \ref{l:standardTransitionModel} $(I)$, $(II)$
is called {\bf $c^{\circ}$-transition-standard}.
\end{definition}

We summarize the steps taken so far.

\begin{proposition}\label{p:GluingModels}
 Let $c:[0,1] \to \partial \coDelta \backslash \shA$ be a straight line and $N_{c}$ be a neighborhood of $\im(c)$ in $\coDelta \setminus \shA$.
 Let $c^{\circ}=c((0,1))$.
 Then for any $s_1$-admissible section $s$, there is a family of $s_1$-admissible section $(s^u)_{u \in [0,1]}$ such that $s^0=s$, for all $u$,  
  $s^u=s$ outside $\pi_{\Delta}^{-1}(N_{c})$, $s^u$ is $K$-invariant,
 and for all $x \in \im(c)$, $M^{s^1}$ is either $x$-standard with respect to $U$ or there exists an open line segment 
 $c^{\circ}_x \subset c^{\circ}$ containing $x$ such that $M^{s^1}$ is $c^{\circ}_x$-transition-standard with respect to $U$.
 
 Moreover, there is a neighborhood $O_{c^{\circ}} \subset  \coDelta \backslash \shA$ of $c^\circ$ and a family of proper $K$-invariant
 Lagrangian $L_t$ in $M_t^{s^1} \cap \pi_{\Delta}^{-1}(O_{c^{\circ}})$, for $t>0$, such that 
 $c^\circ$ is proper inside $O_{c^{\circ}}$, $L_t$  is a $2$-torus bundle (or union of $|K|$ disjoint $2$-torus bundles) with respect to $\pi_{\Delta}$ 
 and $\pi_{\Delta}(L_t)$ Hausdorff converges to $c^{\circ}$ as $t$ goes to $0$.
\end{proposition}

\begin{proof}
 The function $\Type(c(r))$ is discontinuous at finitely many points, say at $0 \le r_1< \dots < r_k \le 1$.
 By extending $c$ slightly, we assume $r_1>0$ and $r_k <1$.
 Pick a $d_j \in (r_j,r_{j+1})$ for $j=1,\dots,k-1$.
 Let $d_0=0$ and $d_k=1$.
 For $j=1,\dots,k$, let $c_{j}^+(r)=c|_{[r_j,d_j]}(r)$ and $c_{j}^-(-r)=c|_{[d_{j-1},r_j]}(r)$.
 By re-parametrizing the domain of $c_j^{\pm}$, we can assume that they satisfy the assumption of Lemma \ref{l:ExistenceTranStandardSymp}.
 We can apply Lemma \ref{l:ExistenceTranStandardSymp} and \ref{l:gluing2lines} to the neighborhoods of $\{\IIm(c_j^{\pm})\}_{j=1}^k$ to find
 a family of $s_1$-admissible section $(s^u)_{u \in [0,1]}$ such that $s^0=s$, for all $u$,
  $s^u=s$ outside $\pi_{\Delta}^{-1}(N_{c})$, $s^u$ is $K$-invariant,
 and $M^{s^1}$ is $(c_j^{\pm})^\circ$-transition-standard with respect to $U$ for all $(c_j^{\pm})^\circ$, where $(c_j^{\pm})^\circ$ is the set of interior points of 
 $\im(c_j^{\pm})$.
 
 For each $c_j^{\pm}$, we obtain a $K$-invariant Lagrangian $(L_t)_{c_j^{\pm}}$ by Lemma \ref{l:standardTransitionModel} such that $(I)$ and $(II)$ are satisfied.
 By definition of transition-standard, $M^{s^1}$ is $x$-standard with respect to $U$ for $x=d_0,\dots,d_k, r_1,\dots,r_k$.
 Therefore, by Proposition \ref{p:standardLagModel}, there exists neighborhoods $O_{d_j}$ of $d_j$ such that 
 $(L_t)_{c_j^{+}} \cap \pi_\Delta^{-1}(O_x)$ and $(L_t)_{c_{j+1}^{-}} \cap \pi_\Delta^{-1}(O_x)$ are given by \eqref{eq:xStandardL} for some appropriate $P,Q,H$.
 By interpolating the $H$, we can concatenate the $K$-invariant Lagrangians $(L_t)_{c_j^{+}} \cap \pi_\Delta^{-1}(O_x)$ and $(L_t)_{c_{j+1}^{-}} \cap \pi_\Delta^{-1}(O_x)$ 
 (for all $j=1,\dots,k-1$), so the result follows.

\end{proof}

\subsubsection{Transition between symplectic corner charts}
 
Since tropical curves considered in Theorem \ref{t:Construction} are not necessarily contained in a single $\coDelta$, we now want to explain the transition between different symplectic corner charts and then how the Lagrangians from  different 
symplectic corner charts can be glued together.
The key conclusion we want to draw is that being $x$-standard is independent of choice of symplectic corner 
charts when $x$ is suitably far away from $\shA$ (see Corollary \ref{c:pStandardIndep}).

Let $U_0$, $U_1$ be symplectic corner charts at vertices $v_0$ and $v_1$ of $\Delta$, respectively.
We assume that $v_0$ and $v_1$ are connected by a $1$-cell $\delta$ in $\Delta$. 
Recall that $\Delta=\Delta_X+\Delta'$ and thus both $v_0$ and $v_1$ decompose accordingly, say as
$v_0=v_0^X+v_0'$ and $v_1=v_1^X+v_1'$. 
From the description of the monodromy in \cite[Proposition 3.15]{MGross05} which involves the vector $v_0^X-v_1^X$, we see that
$v_0^X=v_1^X$ holds if and only if $\delta$ does not meet the discriminant $\shA$.
Let us assume that the reflexive polytope $\Delta_X$ has been translated so that its unique interior lattice point coincides with the origin.
By the Gorenstein assumption, the monoid $\big(\RR_{\ge 0}(\Delta-v_0)\big) \cap\ZZ^4$ is Gorenstein, i.e.~the ideal of integral points in its interior is generated by a single element and this element is $-v^X_0$ (the Gorenstein character). 
A similar statement holds if we replace $v_0$ by $v_1$. We conclude the following consequence from this observation.
\begin{lemma} 
\label{lem-character-is-indep}
Let $U_0/K_0$ and $U_1/K_1$ be corner charts at $v_0$ and $v_1$ respectively connected by an edge $\delta$ so that $\delta\cap \shA=\emptyset$, then the Gorenstein characters of $U_0/K_0$ and $U_1/K_1$ agree on the overlap $U_1/K_1\cap U_0/K_0$.
\end{lemma}

\begin{lemma}
\label{c:pStandardIndep}
Let $x \in \partial \Delta$ and $\Pi$ be the cell of $\partial \Delta$ whose interior contains $x$.
Let $v_0,v_1$ be vertices of $\Pi$ such that there exists a union of $1$-cells $\{\delta_i\}_i$ in $\Pi$ 
connecting $v_0$ and $v_1$ and $\delta_i \cap \shA=\emptyset$ for all $i$.
Let $U_0/K_0, U_1/K_1$ be the symplectic corner charts at $v_0$ and $v_1$, respectively.
If $M^s$ is $x$-standard with respect to $U_0$ then $M^s$ is $x$-standard with respect to $U_1$.
\end{lemma}
\begin{proof} 
It suffices to assume that $v_0$ and $v_1$ are the endpoints of a single edge $\delta$ with $\delta\cap\shA=\emptyset$. 
By the standardness-assumption on $U_0$, there is a neighbourhood $O_{x}$ of $x$ such that the hypersurface $M_t^s \cap \pi_{\Delta}^{-1}(O_{x})$ in the coordinates of $U_0$ is given by
\begin{align} \label{standard-for-U0}
  \left(\prod_{j,x_j =0} \sqrt{p_j}\right)e^{i(q_1+q_2+q_3+q_4)}=tc.
\end{align}
The Gorenstein character $e^{i(q_1+q_2+q_3+q_4)}$ descends to $U_0/K_0$ and by Lemma~\ref{c:pStandardIndep}, it agrees with the one in $U_1$.
The coordinate $p_j$ measures the distance from the corresponding facet of $\Delta$ that contains $x$. In other words, the map $p\mapsto p_j$ is given by pairing with a dual vector that is an inward normal to the facet.
This is true for both $U_1$ and $U_0$. 
Note that the interior of the facet corresponding to $x_j=0$ is necessarily contained in both $\pi_{\Delta}(U_0/K_0)$ and $\pi_{\Delta}(U_1/K_1)$ because both charts contain $x$ and $x$ lies in that facet.
Since the coordinate transformation of the $p$-coordinates from $U_0$ to $U_1$ is affine $\QQ$-linear and identifies the respective placements of the polytope, it follows that, if $x_j=0$, the respective coordinates $p_j$ for $U_0$ and $U_1$ are constant multiples of one another. If we transform \eqref{standard-for-U0} from the coordinates of $U_0$ to the coordinates of $U_1$, the left hand side takes the same shape up to multiplication by a constant which we can absorb into the constant $c$ on the right, so we see that $M^s$ is also $x$-standard with respect to $U_1$.
\end{proof}

When $\Type(x)=3$, we can remove the assumption that $v_0$ and $v_1$ are connected by a union of $1$-cells, in the following sense:

\begin{lemma}\label{l:pStandardIndepType3}
 Let $\Pi$ and $x$ be as in Lemma~\ref{c:pStandardIndep} but we assume that $\Type(x)=3$ (so $\dim(\Pi)=3$).
 Let $v_0,v_1$ be vertices of $\Pi$ and $U_0,U_1$ be the corresponding corner charts.
 If $M^s$ is $x$-standard such that the equation \eqref{eq:SympStandardForm} holds for $c=0$ with respect to $U_0$, then 
 the same is true with respect to $U_1$.
\end{lemma}

\begin{proof}
 The equation \eqref{eq:SympStandardForm} holds for $c=0$ implies that $M^s_t 
 \cap \pi_{\Delta}^{-1}(O_{x})$ coincides with $\pi_{\Delta}^{-1}(O_{x} 
 \cap \Pi)$, which is independent of coordinates.
 Therefore, it is true with respect to $U_0$ if and only if it is true with respect to $U_1$.
\end{proof}

\subsection{Trivalent vertex}\label{ss:trivalent}

In this subsection, we construct a local Lagrangian modeled on a trivalent vertex of a tropical curve $\gamma$.
Near the trivalent vertex, $\gamma$ is contained in a $2$-dimensional plane so we start our construction in $T^*T^2$.

\begin{lemma}\label{l:resolvingConormal}
 In $T^*T^2$, there is a Lagrangian pair of pants $L$ such that outside a compact set, $L$ coincides with the union of the negative co-normal bundles of a $(1,0)$ and $(0,1)$ curve, and
 the positive co-normal bundle of a $(1,1)$ curve.
\end{lemma}
\begin{proof}
 Let $r_i,\theta_i$ be the polar coordinates of $\mathbb{R}^2 \backslash \{0\}$ for $i=1,2$. For a symplectic form on $(\mathbb{R}^2 \backslash \{0\})^2$ we use $\omega:=\sum_i d(\log(r_i)) \wedge d\theta_i$.
 Now, $T^*T^2$ is symplectomorphic to $((\mathbb{R}^2 \backslash \{0\})^2,\omega)$ by the identification $(p_i,q_i)=(\log(r_i),\theta_i)$, where the $q_i$ are the base coordinates of $T^*T^2$.
 In the complex coordinate $z_j=r_j e^{i\theta_j}$, the holomorphic pair of pants $H=\{(z_1,z_2)\mid z_1+z_2=1\}$ is given by
 $$ r_1\cos(\theta_1)+r_2\cos(\theta_2)=1,\quad r_1\sin(\theta_1)+r_2\sin(\theta_2)=0.$$
 To obtain a Lagrangian pair of pants, we use Hyperk\"ahler rotation. 
 Concretely, by transforming $\theta_1\mapsto \theta_2, \theta_2\mapsto-\theta_1$ keeping $r_1,r_2$ fixed, we know that
 $$L:=\left\{(r_1,\theta_1,r_2,\theta_2) \in (\mathbb{R}^2 \backslash \{0\})^2\,\left|\, {r_1 \cos(\theta_2)+r_2 \cos(\theta_1)=1,}\atop{-r_2 \sin(\theta_1)+r_1 \sin(\theta_2)=0}\right.\right\}$$
 is diffeomorphic to a pair of pants. 
 The three punctures corresponds to $r_1=0$, $r_2=0$ and $r_1=r_2=\infty$ respectively. We next check that $L$ is Lagrangian.
 
 The tangent space of $L$ is spanned by
 \begin{align}\label{e:tangent1}
  &\cos(\theta_1)\partial_{r_1}-\frac{\sin(\theta_1)}{r_1}\partial_{\theta_2}-\cos(\theta_2)\partial_{r_2}+\frac{\sin(\theta_2)}{r_2}\partial_{\theta_1},\\ \label{e:tangent2}
  &\sin(\theta_1)\partial_{r_1}+\frac{\cos(\theta_1)}{r_1}\partial_{\theta_2}+\sin(\theta_2)\partial_{r_2}+\frac{\cos(\theta_2)}{r_2}\partial_{\theta_1}
 \end{align}
 as can be checked by applying these to the defining equations of $L$. Computing $\omega(\eqref{e:tangent1},\eqref{e:tangent2})$ gives zero, hence $L$ is Lagrangian.

 Let $\pi:T^*T^2 \to \mathbb{R}^2$ be the projection $\pi(p_i,q_i)=(p_1,p_2)$ which is a Lagrangian torus fiber bundle. 
 Note that $\pi(L)=\pi(H)$ which is an amoeba with three legs asymptotic to the negative $p_1$ axis, the negative $p_2$ axis and the line
 $\{p_1=p_2|p_1>0\}$.
 More precisely, when $r_1 >0$ is sufficiently small, $\theta_1$ is close to $0$ and $r_2$ is close to $1$.
 The situation is similar when $r_2>0$ is sufficiently small.
 When $r_1,r_2$ are sufficiently large, we consider the equation $r_1^2+r_2^2+2r_1r_2\cos(\theta_1+\theta_2)=1$ obtained by sum of squares of two defining equations of $L$.
 It implies that $1 \ge (r_1-r_2)^2$ and $\cos(\theta_1+\theta_2)$ is close to $-1$ when $r_1,r_2$ large, which in turn implies $\frac{r_1}{r_2}$ is close to $1$
 and $\theta_1+\theta_2$ is close to $-\pi$. 
 To complete the proof, it suffices to deform $L$ to another Lagrangian $L'$ such that the three legs of $\pi(L')$ completely coincide with the asymptotic lines outside a compact set.
 
 We now explain the deformation procedure.
 One can check that $\alpha:=p_idq_i$ is exact when restricted to $L$ by showing that $\int_{c_i} \alpha=0$, where $c_i$ are simple closed loops
 wrapping around the asymptotes $r_i=0$ for $i=1,2$.
 Define 
 $$E_1:=\left\{(p_1,q_1,p_2,q_2)\left|q_1=0, p_2=0\right.\right\}$$
 which is the co-normal bundle of $\{q_1=p_1=p_2=0\} \subset \{p_1=p_2=0\}$ when we identify $\{p_1=p_2=0\}$ with the zero section of $T^*T^2$.
 In particular $E_1$ is a Lagrangian.
 The projection $\pi_1:L \to E_1$ defined by $\pi_1(p_1,q_1,p_2,q_2)=(p_1,0,0,q_2)$ is injective and submersive near the end corresponding to $p_1=-\infty$.
 By locally identifying a neighborhood of the zero section of $T^*E_1$ with an open subset of $T^*T^2$,
 $L$ can be identified as a section of $T^*E_1\ra E_1$ near $p_1=-\infty$. 
 Since we checked that $L$ is exact for $\alpha$, one can find a Hamiltonian isotopy to move this end of $L$ to $E_1$.
 For the end of $L$ corresponding to $p_2=-\infty$ and $p_1=p_2=\infty$, we can take $E_2:=\{(p_1,q_1,p_2,q_2)|q_2=0, p_1=0\}$
 and $E_3:=\{(p_1,q_1,p_2,q_2)|p_1=p_2,q_1=-\pi-q_2\}$ to substitute $E_1$,
 and $\pi_2(p_1,q_1,p_2,q_2)=(0,q_1,p_2,0)$ and $\pi_3(p_1,q_1,p_2,q_2)=(p_1,-\pi-q_2,p_1,q_2)$ to substitute $\pi_1$, respectively. 
 This completes the proof.
\end{proof}

\begin{comment}
It is instructive to see how $L$ in the previous lemma intersects the toric fibers.
For fixed $(r_1,r_2)$, we want to solve 
\begin{align*}
 &r_1 \cos(\theta_2)+r_2 \sin(\theta_1)=1\\
 & r_2 \cos(\theta_1)+r_1 \sin(\theta_2)=0
\end{align*}
A routine calculation gives
\begin{align*}
 &r_1 \cos(\theta_2)=\frac{r_1^2-r_2^2+1}{2}\\
 &r_2 \sin(\theta_1)=\frac{r_2^2-r_1^2+1}{2}\\
 & r_2 \cos(\theta_1)=\pm \sqrt{r_1^2-(\frac{r_1^2-r_2^2+1}{2})^2} \\
 &r_1 \sin(\theta_2)=\mp \sqrt{r_1^2-(\frac{r_1^2-r_2^2+1}{2})^2} 
\end{align*}
\end{comment}

By multiplying Lemma \ref{l:resolvingConormal} with a trivial $T^*S^1$ factor, we have the following.

\begin{corollary}\label{c:PoP}
  In $T^*T^3$, there is a Lagrangian pair of pants times circle $L$ such that outside a 
  compact set, $L$ coincides with the union of the negative co-normal bundles of a $(1,0,0)$-curve times a $(0,0,1)$-curve, of a $(0,1,0)$-curve times a $(0,0,1)$-curve, and
 the positive co-normal bundle of a $(1,1,0)$ curve times a $(0,0,1)$-curve.

 By applying backward Liouville flow for the standard Liouville structure on $T^*T^3$, we can assume $L$ to lie inside a small open neighborhood of the union of the zero section $T^3$ and the negative/positive co-normal bundles, and the neighborhood is as small as we want.
 \end{corollary}

\begin{lemma}\label{l:Gluing3Legs}
Let $M^s$ be $x$-standard for some $x \in \partial \coDelta \backslash \shA$ so that 
$ \pi_{\Delta}^{-1}(O_x') \cap M_t^s$ is given by Equation \eqref{eq:SympStandardForm} for
a small neighborhood $O_x'$ of $x$.
Let $c_j:[0,1) \to O_x' \cap \partial \coDelta$ for $j=1,2,3$ be proper straight lines such that $c_j(0)=x$ for all $j$. 
Assume the directions of $c_j$ is integral linearly equivalent to $\{e_1,e_2,-e_1-e_2\}$
 with respect to the integral affine structure on $\partial \coDelta \backslash \shA$.
 Then there exists a small neighborhood $O_x \subset O_x'$ of $x$, small  neighborhoods $O_{c_j} \subset O_x'$ of $\im(c_j)$ and
 a family of proper Lagrangian pair of pants times circle $L_t \subset \pi_{\Delta}^{-1}(O_x') \cap M_t^s$, for $t>0$,
such that $L_t \cap \pi_{\Delta}^{-1}(O_{c_j})$ is $\im(c_j)$-standard outside  $\pi_{\Delta}^{-1}(O_x)$ for $j=1,2,3$.
\end{lemma}

\begin{proof}
By Proposition \ref{p:standardLagModel}, we can construct
$\im(c_j)$-standard Lagrangian $L_j$ in $\pi_{\Delta}^{-1}(O_x') \cap M_t^s$.
The set of $p$-coordinates of $L_j$ are determined by condition $(II)$ in  Proposition \ref{p:standardLagModel}.
Let the set of $p$-coordinates of $L_j$ be $P_j$.
Notice that $\cap_{j=1,2,3} P_j$ is a singleton given by the unique element in $\pi_{\Delta}(M^s_t)$ such that $p_1-x_1=\dots=p_4-x_4$.
Let $p^*$ be the unique element in $\cap_{j=1,2,3} P_j$ and $T_{p^*}:= \pi_{\Delta}^{-1}(p^*) \cap M_t^s$ be the Lagrangian $T^3$ in $M_t^s$.

The assumptions of the directions of $c_j$ implies that, for some choice of coordinates in $T_{p^*}$, the intersection pattern of $L_j$ with $T_{p^*}$
is exactly given by
$(1,0,0)$-curve times $(0,0,1)$-curve, $(0,1,0)$-curve times $(0,0,1)$-curve and $(1,1,0)$ curve times $(0,0,1)$-curve.
We can do a Hamiltonian perturbation of $L_j$ such that, with respect to a choice of Weinstein neighborhood of $T_{p^*}$, $L_j$ coincides with the 
negative co-normal bundles of a $(1,0,0)$-curve times $(0,0,1)$-curve, $(0,1,0)$-curve times $(0,0,1)$-curve, and
 the positive co-normal bundle of a $(1,1,0)$ curve times $(0,0,1)$-curve.

We can also adjust $T_{p^*} \cap L_j$ by parallel translate the $2$-tori using $H(u)$ in Proposition \ref{p:standardLagModel} if necessary.
Therefore, we can apply Corollary \ref{c:PoP} to glue the $L_j$ together and obtain a proper Lagrangian pair of pants times circle $L_t$.
It is clear that $L_t \cap \pi_{\Delta}^{-1}(O_{c_j})$ is $\im(c_j)$-standard outside  $\pi_{\Delta}^{-1}(O_x)$ for some small neighborhood $O_x$ of $x$.
\end{proof}

\subsection{Assembling local Lagrangian pieces away from the discriminant}

We apply the results in the previous two subsections and conclude the construction of the Lagrangian away from the discriminant.

\begin{terminology}\label{term:SolidTori}
 A solid torus is a manifold diffeomorphic to  $S^1 \times \{z\in \CC| |z| \le 1\}$.
 An open solid torus is a manifold diffeomorphic to the interior of a solid torus.
\end{terminology}

Let $\gamma$ be an admissible tropical curve (see the assumption of Theorem \ref{t:Construction}).
Let $N$ be a neighborhood of $\gamma$ and $B' \subset B \subset N$ be small open tubular neighborhoods of the ends of $\gamma$ such that the closure $\overline{B}'$ of $B'$ lies inside $B$.
In particular, we can write $B=\cup_e B_e$ and $B'= \cup_e B'_e$ where the union is taken over all the ends $e$ of $\gamma$ and $B_e$, $B_e'$ are small topological balls containing $e$.

\begin{proposition}\label{p:assemble}

Suppose there exists an $s_1$-admissible section $s$ and, for all $t>0$ small and for each end $e$,
a Lagrangian open solid torus $L_t^e$ in $\pi_{\Delta}^{-1}(B) \cap M^s_t$ 
such that $L_t^e$ is $(B_e \setminus \overline{B}'_e) \cap \gamma$-standard,
and the directions of the meridian and longitude of $L_t^e$ with respect to the integral affine structure 
are as in $L_v$ in Section \ref{ss:LagrangianLift}.
Then, for all $t>0$ sufficiently small, there is a closed Lagrangian  $L_t \subset M^s_t$
such that $L_t$ is diffeomorphic to a Lagrangian lift of $\gamma$ and $\pi_{\Delta}(L_t) \subset N$.
Moreover, we have $w(L_t)=\mult(\gamma)$.

\end{proposition} 

\begin{proof}

We first explain the construction of $L_t$ and proof concept is the same as for Proposition \ref{p:GluingModels}.
Let $D:=\{d_i\}_{i=1}^K \subset \gamma \setminus B$ be a finite collection of points such that it contains all the trivalent points of $\gamma$ and all the points in $\partial \overline{B} \cap \gamma$.
By adding more points to $D$ if necessary, we can assume that every point $x$ on $\gamma$ is contained in the image of a curve $c:[0,1] \to \coDelta \setminus \shA$, for some $\coDelta$,
such that $\Type(c(0)) \le \Type(c(r))=\Type(c(1))$ for all $r \in (0,1]$.
In particular, it implies that the interval between two adjacent points $d,d'$ of $D$ (adjacent with respect to the topology on $\gamma$) is the image of such a curve $c$.
We denote the open (resp. closed) interval between two adjacent points $d,d'$ by $(d,d')$ (resp. $[d,d']$).

By repeatedly applying Lemma \ref{l:gluing2lines}, we get a new $s_1$-admissible section $s'$
such that $M^{s'}$ is $(d,d')$-transition-standard for all adjacent points $d,d' \in D$ and $s'=s$ outside $\pi_{\Delta}^{-1}(N)$.
Moreover, since $M^s$ is standard for points in $(B \setminus \overline{B}') \cap \gamma$ a priori, when we apply Lemma \ref{l:gluing2lines}, we can assume that the outcome $s'$ equals $s$ inside $\pi_{\Delta}^{-1}(B)$.

As a consequence of $M^{s'}$ being $(d,d')$-transition-standard, $M^{s'}$ is $x$-standard for all $x \in D$ (here, we use Corollary~\ref{c:pStandardIndep} and Lemma~\ref{l:pStandardIndepType3} to guarantee that being $x$-standard is independent of corner charts: if $y \in (d,d')$ has $\Type(y)=3$, we apply Lemma~\ref{l:pStandardIndepType3}; 
if $y \in (d,d')$ has $\Type(y)<3$, the assumption of Corollary~\ref{c:pStandardIndep} will be satisfied so we can apply Corollary~\ref{c:pStandardIndep}).
In particular, $M^{s'}$ is $x$-standard for all trivalent points $x$ of $\gamma$.
Let $x$ be a trivalent point of $\gamma$ and  let $d_{i_1}, d_{i_2}, d_{i_3} \in D$
be the three adjacent points of $x$ on the three incident edges of $x$, respectively.
We can apply Lemma \ref{l:Gluing3Legs} at $x$.
The result is a point $b_{i_k} \in (x,d_{i_k})$ for each $k=1,2,3$
such that, for all $t>0$ small, there exists a Lagrangian pairs of pants times circle  $L^x_t \subset M^{s'}$
such that $L^x_t $ is $(x,b_{i_k})$-standard outside the preimage of a small neighborhood $O_x$ of $x$ under $\pi_{\Delta}$.
Since $M^{s'}$ is $(x,d_{i_k})$-transition standard for all $k=1,2,3$, we can apply Lemma \ref{l:standardTransitionModel} to extend $L^x_t$
so that it becomes  $(x,d_{i_k})$-transition standard outside $\pi_{\Delta}^{-1}(O_x)$.

Now, as in the proof of Proposition \ref{p:GluingModels}, for all adjacent $d,d' \in D$ such that $d,d'$ are not trivalent points of $\gamma$, we can also construct Lagrangian local pieces in $M^{s'}$ that are $(d,d')$-transition-standard.
Moreover, we can glue these local pieces together smoothly to get, for all $t>0$ small, a closed Lagrangian $L_t$.

Since $s$ and $s'$ are interpolated by a family of $s_1$-admissible sections that is unchanged outside $\pi_{\Delta}^{-1}(N)$,
we can apply Lemma \ref{l:symplecticIsotopy} to conclude that  $L_t \subset M^{s'}_t$
can be brought back, via a symplectic isotopy, to a closed embedded Lagrangian inside $M^s_t \cap \pi_{\Delta}^{-1}(N)$.

Finally, for the diffeomorphism type and topology of $L_t$, it is clear from the construction that the diffeomorphism type of $L_t$ is governed by $\gamma$ and coincides with Definition \ref{d:LagLift}.
In particular, for a rigid $\gamma$ of genus zero, $L_t$ is a rational homology sphere and $w(L)=\mult(\gamma)$.
\end{proof}

\section{Near the discriminant}\label{s:NearSingLoci}

In this section, we explain the construction of a local Lagrangian solid torus that serves as capping off the Lagrangian 3-folds near the discriminant.
We first explain the case where $\PP_\Delta$ is a toric manifold; subsection \ref{ss:OrbifoldConstruction} reduces the more general orbifold situation to the manifold case.

Let $U$ be a symplectic corner chart.
As explained in Section \ref{ss:toricBasic} (see \eqref{eq:Transition}), we have an explicit diffeomorphism $ \Phi_U:U \to \mathbb{C}^4$ given by
\begin{align}
w_j=\exp(u_j+iv_j)=\Phi_{U,j}(z)=\exp\left(\frac{\partial f_J(p)}{\partial p_j}\right)\exp(iq_j) \label{eq:Phi}
\end{align}
where $(w_1,\dots,w_4) \in \mathbb{C}^4$, $z=(z_1,\dots,z_4)$, $z_j=\sqrt{2p_j}\exp(iq_j)$ and $\Phi_U=(\Phi_{U,1}, \dots,\Phi_{U,4})$.

Let $s_1 \in H^0(\PP_\Delta,\shL)_{\Reg}$ and $M_t:=M_t^{s_1}$ (see \eqref{eq:RegSection}).
For the purpose of capping off the Lagrangian, we will make an assumption on the shape of the discriminant near the ending. %This assumption will automatically be satisfied for all mirror quintics. 
Say the piece of the discriminant that we want to cap off the Lagrangian at is contained in the complex two-dimensional stratum $T=\{w_1=w_2=0,w_3w_4\neq 0\}$.

\begin{assumption}\label{a:singularModel}
In this section, we assume 
\begin{align}\label{eq:singularModel}
M_t \cap U=\Phi_U^{-1}(\{w_1w_2w_3w_4=tg(w)\}) 
\end{align}
where $g(w):=c(b-w_3)+w_1h_1(w)+w_2h_2(w)$ for some polynomial functions $h_1,h_2:\CC^4 \to \mathbb{C}$ and constants $b, c \in \mathbb{C}^*$.
In other words, the restriction of $g$ to $T$ is constant in $w_4$ and degree one in $w_3$.
\end{assumption}

\begin{remark}\label{r:justifyA6}
As a consequence of Lemma~\ref{lem-admissible}, the situation of Assumption~\ref{a:singularModel} is equivalent to the corner chart $U$ being based at a vertex with an adjacent two-cell that deforms to a one-cell in $\Delta_{\check X}$.
Furthermore, $g$ being locally of this form is equivalent to its amoeba $\shA$ locally being one-dimensional. By the admissibility assumption on the tropical curves that we build Lagrangians for, its univalent vertices permit a nearby vertex of $\Delta$ that is contained in a $2$-cell so that the associated chart $\Phi_U$ gives the hypersurface the form of \eqref{eq:singularModel} for $T$ the toric two-stratum associated to the two-cell in $\Delta$ that contains the univalent vertex of $\gamma$.
$\shA$ is locally of dimension one if and only if \eqref{eq:singularModel} holds.
\end{remark}

\begin{remark}[Trivalent vertex]\label{r:trivalentVertex1}
It is natural to ask whether Theorem \ref{t:Construction} can be generalized to tropical curves whose univalent vertices 
end at a codimension one part of $\shA$.
In this case, the local model is 
\begin{align}\label{eq:singularModel2}
M_t \cap U=\Phi_U^{-1}(\{w_1w_2w_3w_4=tg(w)\}) 
\end{align}
and $g(w)=c(b-b_3w_3-b_4w_4)+w_1h_1(w)+w_2h_2(w)$ for some polynomial functions $h_1,h_2:\CC^4 \to \mathbb{C}$ and constants $b, c, b_3,b_4 \in \mathbb{C}^*$.

The key difficulty for this generalization is whether one can straighten the discriminant as in Proposition \ref{p:correctSingModel}.
More details will be explained in Remark \ref{r:trivalentVertex2}.

\end{remark}

Let $g_0:=g|_T$.
Since $g_0(w_3,w_4)=g(0,0,w_3,w_4)=c(b-w_3)$, the discriminant $\Disc(s_1)$ intersected with the stratum $T$ is
\begin{align}
\Disc^0(s_1):=\Disc(s_1) \cap T =\{g_0=0\}&= \{w_3=b\}\cap T \label{eq:disc0}.
\end{align}
Let $\pi:U\ra \coDelta$ be the moment map restricted to $U$ and $\shA:=\pi(\Disc^0(s_1))$.

\begin{lemma}\label{l:fiberBundle}
 $\pi|_{\Disc^0(s_1)}$ is an $S^1$ fiber bundle over $\shA$ and the tangent space of each $S^1$-fiber is generated by $\partial_{q_4}=\partial_{v_4}$.
 Moreover, $\shA$ is an open embedded curve inside the two cell $\{p_1=p_2=0\} \subset \coDelta$
 such that $\shA$ is transverse to the slices $\{p_4=const\}$.
\end{lemma}

\begin{proof}
Since $\Disc(s_1) \cap T$ is connected, so is its projection $\shA$.
Inserting $w_j=e^{u_j+iv_j}$ into $g_0(w_3,w_4)=g(0,0,w_3,w_4)$ and taking logarithm yields that $\Disc(s_1)\cap T$ is given by
$u_3=const$ and $v_3=const$, so it is invariant under the
subtorus action $\{(0,0,0, \vartheta) \in T^4| \vartheta \in S^1\}$. 
This proves the first statement of the lemma.

The curve $\shA$ in $p$-coordinates is found by inserting $u_3=\frac{\partial f_J(p)}{\partial p_3}$ into $u_3=const$ and since $f_J$ is a smooth function,  $\shA$ is a smooth connected curve.
Let $p=(0,0,f_1(r),f_2(r))$ be a parametrization of $\shA$.
Note that 
$\Disc(s_1)\cap T$ is symplectic with tangent space generated by 
$\{f_1'(r) \partial_{p_3}+f'_2(r) \partial_{p_4}, \partial_{q_4}\}$.
It means that $\omega(f_1'(r) \partial_{p_3}+f'_2(r) \partial_{p_4},\partial_{q_4})=f_2'(r) \neq 0$ for all $r$, so 
$\shA$ is transverse to the slices $\{p_4=const\}$.
\end{proof}

\begin{remark}
An alternative proof of Lemma \ref{l:fiberBundle} suggested by an anonymous referee is as follow:
the Hessian of $f_J$ is positive definite so $\frac{\partial^2 f_J(p)}{\partial p_3^2} >0$, and therefore 
 $\shA=\{u_3=const\}$ is a curve as claimed.
\end{remark}

We consider a straight line segment $\gamma(r)=(0,0,r,R) \in \coDelta$ for some fixed $R \in \RR_{>0}$ parametrized by $r\in (r_0, r_1]$, 
inside the $2$-cell $\{p_1=p_2=0\}=\pi(T)$, such that $0<r_0<r_1$ and 
$\gamma(r) \in \shA\iff r=r_1$. (see Figure \ref{fig:TropicalNearSing}).

\begin{figure}[htb]
\begin{tikzpicture}[scale=1.2]
\begin{axis}[
  width=1.8cm,
  height=4cm,
  hide axis,
  axis lines=left,
  clip=false,
  xticklabels=\empty,
  yticklabels=\empty,
]
%\fill[fill=black!10,rounded corners=0.5mm] (axis cs: 1, 1) rectangle (axis cs: 5, 7);
\draw[black!10,fill=black!10] (axis cs: 3.25,4.2) ellipse (1cm and .9cm);
\node [align = center, anchor = west, black!35] at (axis cs: 0.8,7.8) {$N$};
\fill[fill=black!16,rounded corners=0.6mm] (axis cs: 2, 2) rectangle (axis cs: 4.5, 6);
\node [align = center, anchor = west, black!45] at (axis cs: 1.6,6.8) {$N'$};
%\fill[fill=black!22,rounded corners=0.7mm] (axis cs: 2.5, 3) rectangle (axis cs: 3.5, 5);
%\node [align = center, anchor = west, black!55] at (axis cs: 2.4,5.8) {$N''$};

\addplot+[thick,mark=none,samples=200,domain=0:10] ({-1/5*sqrt(4*x^2 - 40*x + 475) + 8},{x});
\node [align = center, anchor = west, blue] at (axis cs: 3.8,8.5) {$\mathcal{A}$};
\draw [draw=black] (axis cs: 0,0) -- (axis cs: 5,0);
\draw [draw=black] (axis cs: 7,5) -- (axis cs: 5,0);
\draw [draw=black] (axis cs: 7,5) -- (axis cs: 5,10);
\draw [draw=black] (axis cs: 0,10) -- (axis cs: 5,10);
\draw [draw=black] (axis cs: 0,10) -- (axis cs: 0,0);
\draw [thick, draw=red] (axis cs: 0,4) -- (axis cs: 4.106415533,4);
\node [align = center, anchor = west, red] at (axis cs: 0,3) {$\gamma$};
\end{axis}
\end{tikzpicture}
\qquad
\begin{tikzpicture}[scale=1.2]
\begin{axis}[
  width=1.8cm,
  height=4cm,
  hide axis,
  axis lines=left,
  clip=false,
  xticklabels=\empty,
  yticklabels=\empty,
]
\draw[black!10,fill=black!10] (axis cs: 3.25,4.2) ellipse (1cm and .9cm);
\node [align = center, anchor = west, black!35] at (axis cs: 0.8,7.8) {$N$};
\fill[fill=black!16,rounded corners=0.6mm] (axis cs: 2, 2) rectangle (axis cs: 4.5, 6);
\node [align = center, anchor = west, black!45] at (axis cs: 1.6,6.8) {$N'$};
%\fill[fill=black!22,rounded corners=0.7mm] (axis cs: 2.5, 3) rectangle (axis cs: 3.5, 5);
%\node [align = center, anchor = west, black!55] at (axis cs: 2.4,5.8) {$N''$};

\addplot+[thick,mark=none,samples=200,domain=0:10] ({-1/5*sqrt(4*x^2 - 40*x + 475) + 8},{x});
\node [align = center, anchor = west, blue] at (axis cs: 3.8,8.5) {$\mathcal{A}$};
\path [fill,fill=black!16] (axis cs: 3.8,2) rectangle (axis cs: 4.4,6);
\draw [thick,draw=blue] (axis cs: 4.106415533,2.6) -- (axis cs: 4.106415533,5.4);
\draw [thick,draw=blue] (axis cs: 4.106415533,2.6) .. controls (axis cs: 4.1,2.45) and (axis cs: 3.9,2.3)  .. (axis cs: 3.94537301345734,2);
\draw [thick,draw=blue] (axis cs: 4.106415533,5.4) .. controls (axis cs: 4.16,5.7) .. (axis cs: 4.10641553321364,6);
\draw [draw=black] (axis cs: 0,0) -- (axis cs: 5,0);
\draw [draw=black] (axis cs: 7,5) -- (axis cs: 5,0);
\draw [draw=black] (axis cs: 7,5) -- (axis cs: 5,10);
\draw [draw=black] (axis cs: 0,10) -- (axis cs: 5,10);
\draw [draw=black] (axis cs: 0,10) -- (axis cs: 0,0);
\draw [thick, draw=red] (axis cs: 0,4) -- (axis cs: 4.106415533,4);
\node [align = center, anchor = west, red] at (axis cs: 0,3) {$\gamma$};
\end{axis}
\end{tikzpicture}
 \caption{The straight line segment $\gamma$ inside the two cell $\{p_1=p_2=0\}$. The transition from the left to the right image is addressed in Section~\ref{ss:Correcting discriminant} below.}
 \label{fig:TropicalNearSing}
\end{figure}
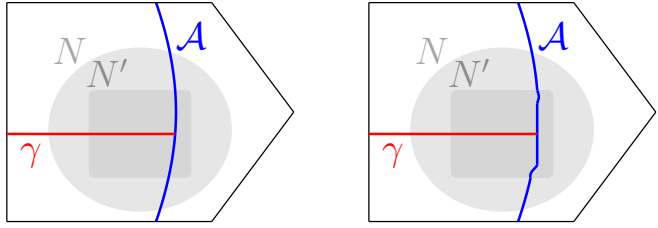

The main result we want to prove in this section is:

\begin{theorem}[Lagrangian solid tori]\label{t:LagrangianSolidTori}
Let $s$ be an $s_1$-admissible section.
 For any neighborhood $N \subset \coDelta$ of $\gamma(r_1)$, there exist $r'<r''<r_1$ with $\gamma([r',r_1])\subset N$,
 and a family of $s_1$-admissible section $(s^u)_{u \in [0,1]}$
 such that $s^0=s$, for all $u$, $s^u=s$ outside $\pi_{\Delta}^{-1}(N)$ and $M^{s^1}$ is $x$-standard with respect to $U$ for all $x \in \gamma([r',r''])$.

Moreover there exists a neighborhood $N' \subset N$ of $\gamma((r',r_1])$
such that $\gamma((r',r_1])$ is proper in $N'$ and there exists
a family of proper Lagrangian open solid tori $L_t \subset (M^{s^1}_t \cap \pi^{-1}(N'))$, for all $t>0$ sufficiently small,
such that $L_t$ is $(\gamma((r',r'')))$-standard (see Definition~\ref{def-Lag-standard} and Terminology \ref{term:SolidTori}). 
\end{theorem}

Note that, in Theorem \ref{t:LagrangianSolidTori}, $L_t$ being $(\gamma((r',r'')))$-standard and proper in $(M^{s^1}_t \cap \pi^{-1}(N'))$
implies that the infinite end of $L_t$ is contained in $\pi^{-1}(\gamma((r',r'')))$.
We will use this property to glue $L_t$ with the standard Lagrangian models constructed in Section \ref{s:AwayFromSing}
to conclude  the proof of Theorem \ref{t:Construction}, eventually.

\subsection{Lagrangian construction near the discriminant under assumptions}\label{ss:CloseUp}

In this section, we give the construction of a Lagrangian solid torus under two additional assumptions on $M_t$ near $\Disc(s_1)\cap T$, and we later show how to reduce the general case to this case.
%This section is devoted to various lemmata building towards the Lagrangian construction near discriminant.
%More precisely, we will construct a Lagrangian solid torus in $(M_t^1)_\beta$ with Legendrian torus boundary (see Proposition \ref{p:LagrangianConstructionSingLoci}).
%Together with Proposition \ref{p:fibrationAfterGoodDeform}, we will have a Lagrangian solid torus inside  $\{\widehat{s}_0=ts^1\}$ 
%whose $\pi_\Delta$-image lies inside a small neighborhood of the tropical curve $\widehat{\gamma}_U$.
We start with some preliminaries about contact geometry and Legendrian submanifolds.

\subsubsection{Digression into contact geometry}
Let $(P,\omega)$ be a compact symplectic manifold with boundary.
A Liouville structure on $(P,\omega)$ is a choice of $\alpha \in \Omega^1(P)$ such that $d\alpha=\omega$
and that the vector field $Z$, $\omega$-dual to $\alpha$ (i.e.~$\iota_Z\omega=\alpha$), points outward along $\partial P$.
The triple $(P,\omega,\alpha)$ is called a Liouville domain.

\begin{example}\label{ex:ball}
 Let $(B^{2n},\sum r_jdr_j \wedge d\theta_j)$ be the standard symplectic closed ball.
 We can pick $\alpha=\sum\frac{r_j^2}{2}d\theta_j$.
 In this case, $Z=\sum\frac{r_j}{2}\partial_{r_j}$ points outward along $\partial B^{2n}$.
\end{example}

Given a Liouville domain $(P,\omega,\alpha)$, $(\partial P, \ker(\alpha|_{\partial P}))$ is a contact manifold (see e.g.~\cite{GeigesBook}, \cite{McDuffSalamon})
and we call it the contact boundary of $(P,\omega,\alpha)$.
The contact boundary of the Liouville domain in Example \ref{ex:ball} is called the standard contact sphere $(S^{2n-1},\xi_{std})$.
In general, there are many contact structures one can put on an odd-dimensional manifold even if one restricts to those that arise as the contact boundary of a Liouville domain.
In contrast, there is a unique contact structure on the $3$-dimensional sphere (up to contactomorphisms) which can be the contact boundary of a Liouville domain, namely, the standard one (see \cite{Eliashberg3Sphere}).

\begin{theorem}[see \cite{Eliashberg90}, and also Theorem $1.7$ of \cite{McDuff90}]\label{t:filling}
 If $(P,\omega,\alpha)$ is a Liouville domain with its contact boundary being the standard contact $3$-sphere,
 then $(P,\omega,\alpha)$ is symplectic deformation equivalent to the standard symplectic closed $4$-ball.
\end{theorem}

A knot $K$ in $(S^{3},\xi_{std})$ is called Legendrian if $T_pK \subset \xi_{std}$ for every point $p \in K$.
A Legendrian unknot is a Legendrian knot such that its underlying smooth knot type is an unknot.

\begin{example}\label{ex:StandardUnknot}
 Let $K \subset (S^3,\xi_{std}) \subset \RR^4$ be the intersection of $(S^3,\xi_{std})$ with a Lagrangian vector subspace of $(\RR^4,\omega_{std})$.
 Then $K$ is a Legendrian unknot and we call it a standard Legendrian unknot.
\end{example}

The Legendrian isotopy type of a Legendrian unknot is classified by its Thurston-Bennequin number and rotation number (see \cite{EliashbergFraser}, and also \cite[Section $5$]{Etnyre05} for more about these background materials).
There is exactly one Legendrian unknot with Thurston-Bennequin number $-1$ up to Legendrian isotopy and it is realized by the standard Legendrian unknot.
By the Thurston-Bennequin inequality, a Legendrian unknot can bound an embedded Lagrangian disk in $(B^4,\omega_{std})$
only if its Thurston-Bennequin number is $-1$.
The converse is also well-known to be true:

\begin{lemma}[Bounding a Lagrangian disk]\label{l:closing}
 Let $(P,\omega,\alpha)$ be a Liouville domain with contact boundary $(S^{3},\xi_{std})$.
 If $\Lambda \subset (\partial P,\ker(\alpha))$ is Legendrian isotopic to the standard Legendrian unknot, then there is an embedded Lagrangian disk $D \subset (P,\omega)$
 such that $\partial D =D \cap \partial P= \Lambda$
\end{lemma}

\begin{proof}
 By Theorem \ref{t:filling}, it suffices to assume that $(P,\omega,\alpha)$ is a star-shaped domain in $(\mathbb{R}^4,\omega_{std})$.
 By \cite[Theorem $1.2$]{Chantraine10}, there is a small Darboux ball $B^4 \subset P$
 and an embedded Lagrangian $L \subset P \setminus \Int(B^4)$ such that $L \cap \partial P=\Lambda$
 and $L \cap \partial B^4$ is a standard Legendrian unknot.
 Moreover, we can assume that $L$ is invariant with respect to radial direction near $\partial B^4$.
 Therefore, we can close up $L$ by a Lagrangian plane in $B^4$ by Example \ref{ex:StandardUnknot}.
 \end{proof}

Let $M_t':=\{(z_1,z_2,z_3)\in \mathbb{C}^3| z_1z_2=tz_3\}$ which is a complex and hence symplectic hypersurface for all $t \neq 0$.
With positive $\epsilon$, let $Y_{\epsilon}:=\{z \in \mathbb{C}^3||z_1|^2+|z_2|^2+|z_3|^2=\epsilon\}$ be the $5$-sphere equipped with the standard contact structure 
and contact form $\alpha|_{Y_{\epsilon}}=\sum \frac{r_i^2}{2}d\theta_i$ (see Example \ref{ex:ball}). 

\begin{lemma}\label{l:3sphere}
 For $t\in\RR_{>0}$, the contact form $\alpha|_{Y_{\epsilon}}$ restricts to a contact form on $Y_{\epsilon,t}:=M_t' \cap Y_{\epsilon}$ 
 such that $(Y_{\epsilon,t}, \ker(\alpha|_{Y_{\epsilon,t}}))$ is contactomorphic to the standard contact $3$-sphere.
\end{lemma}

\begin{proof}
This result is well-known (see Remark \ref{r:link3sphere}) but we still want to give some details.
 Without loss of generality, we assume $t$ is real positive. Note that
 $Y_{\epsilon,t}$ is the union of $Y_{\epsilon,t} \setminus \{z_1=0\}$ and $Y_{\epsilon,t} \setminus \{z_2=0\}$.
 We parametrize $Y_{\epsilon,t} \setminus \{z_1=0\}$ and $Y_{\epsilon,t} \setminus \{z_2=0\}$ by
 \resizebox{\linewidth}{!}{
 \begin{minipage}{\linewidth}
 \begin{align*}
  \Big\{(r_1,\theta_1,r_2,\theta_2,r_3,\theta_3)&=\Big(r,\theta_1,\frac{\rho(\epsilon,t,r)t}{r},\theta_2,\rho(\epsilon,t,r),\theta_1+\theta_2\Big)\Big| r \in (0,\sqrt{\epsilon}], \theta_1,\theta_2 \in \mathbb{R}/2\pi \mathbb{Z}\Big\}, \\
  \Big\{(r_1,\theta_1,r_2,\theta_2,r_3,\theta_3)&=\Big(\frac{\rho(\epsilon,t,r)t}{r},\theta_1,r,\theta_2,\rho(\epsilon,t,r),\theta_1+\theta_2\Big)\Big| r \in (0,\sqrt{\epsilon}], \theta_1,\theta_2 \in \mathbb{R}/2\pi \mathbb{Z}\Big\}
 \end{align*}
 \end{minipage}
}\\
 where $\rho(\epsilon,t,r):=\sqrt{\frac{r^2(\epsilon-r^2)}{r^2+t^2}}$, so when $r=\sqrt{\epsilon}$, we have $\rho(\epsilon,t,\sqrt{\epsilon})=0$
 and the corresponding angular variable (i.e.~$\theta_2$ for the first equation and $\theta_1$ for the second equation) collapses.
 In particular, the parametrizations of $Y_{\epsilon,t} \setminus \{z_1=0\}$ and $Y_{\epsilon,t} \setminus \{z_2=0\}$
 exactly give a Heegaard decomposition of $Y_{\epsilon,t}$.
 The collapsing circles at the ends have intersection pairing one in the Heegaard surface (a $2$-torus) so $Y_{\epsilon,t}=S^3$.
 
 Let $\Phi(s_1,\vartheta_1,s_2,\vartheta_2):=(s_1 \sqrt{t}\exp(i \vartheta_1),s_2 \sqrt{t} \exp(i \vartheta_2),s_1s_2\exp(i(\vartheta_1+\vartheta_2)))$
 be a chart for $M_t'$ and let $\alpha:=\sum \frac{r_j^2}{2}d\theta_j|_{M_t}$ and recall that $\omega=\sum r_jdr_j\wedge\theta_j$.
 Then we have
 \begin{align*}
\Phi^*\alpha&=\frac{s_1^2}{2}(t+s_2^2)d\vartheta_1+\frac{s_2^2}{2}(t+s_1^2)d\vartheta_2,\\
\Phi^*\omega&=s_1(t+s_2^2)ds_1 \wedge d\vartheta_1+s_1^2s_2 ds_2 \wedge d\vartheta_1+s_1s_2^2 ds_1 \wedge d\vartheta_2 +s_2(t+s_1^2)ds_2 \wedge d\vartheta_2, \hbox{ so}\\
 Z_t&=\frac{1}{2(t+s_1^2+s_2^2)}(s_1(t+s_1^2)\partial_{s_1}+s_2(t+s_2^2)\partial_{s_2})
 \end{align*}
is checked to be the dual of $\Phi^*\alpha$ with respect to $\Phi^*\omega|_{M_t}$.
In particular, the Liouville vector field $Z_t$ points outward along $\partial (M_t' \cap \{|z| \le \epsilon\})$. 
Therefore, $M_t' \cap \{|z| \le \epsilon\}$ is a Liouville domain with contact  boundary $(Y_{\epsilon,t}, \ker(\alpha|_{Y_{\epsilon,t}}))$.
Since the standard contact $3$-sphere is the only contact $3$-sphere that arises as the boundary of a Liouville domain, the result follows.
\end{proof}

\begin{remark}\label{r:link3sphere}
 $Y_{\epsilon,t}$ is called the link of the `singularity' of $M'_t$ at the origin.
 Since $M'_t$ is smooth at the origin for $t \neq 0$, the link of the origin is contactomorphic to the standard contact $3$-sphere.
\end{remark}

By translating the $z_3$ coordinate, we know that
 $Y_{a,\epsilon,t}:=\{z \in \mathbb{C}^3| z_1z_2=t(z_3-a)\} \cap \{z \in \CC||z_1|^2+|z_2|^2+|z_3-a|^2=\epsilon\}$ 
 is naturally equipped with a contact structure making it a standard contact $3$-sphere for $a \in \CC$ and $\epsilon>0$.

\begin{lemma}\label{l:LegendrianUnknot} 
With $t\in\RR_{>0}$, the following is a Lagrangian disk in $M_{a,t}':=\{z_1z_2=t(z_3-a)\}$,
\begin{align}
 L:=\{(re^{i\theta},re^{-i\theta},\frac{r^2}{t}+a)|r \in [0,\infty), \theta \in \mathbb{R}/2\pi \mathbb{Z}\} \label{eq:LAGDisk}
\end{align}
Moreover, $\partial L :=L \cap Y_{a,\epsilon,t}$ is a Legendrian and has the Legendrian isotopy type of a standard Legendrian unknot in $Y_{a,\epsilon,t}$.
 
 \end{lemma}

\begin{proof} 
 Being a Lagrangian disk is an easy check.
 Using the chart $\Phi$ for $M_t'$ from above, shifted by $(0,0,a)$, we have 
 \begin{align*}
 \Phi^{-1}(L)=\left\{\left(\frac{r}{\sqrt{t}},\theta,\frac{r}{\sqrt{t}},-\theta\right)\right\}. 
 \end{align*}
 By the proof of Lemma \ref{l:3sphere}, $Z_t=\frac{1}{2(t+s_1^2+s_2^2)}(s_1(t+s_1^2)\partial_{s_1}+s_2(t+s_2^2)\partial_{s_2})$, so we have $Z_t|_x \in T_x(\Phi^{-1}(L))$
 for all $x \in \Phi^{-1}(L)$.
 Therefore, $\partial L$ is a Legendrian.
 The only Legendrian isotopy type that can bound a Lagrangian disk is the standard one so $\partial L$ is Legendrian isotopic to the standard Legendrian unknot.
\end{proof}

\begin{remark}\label{r:FlexibleDisk}
 For every $\phi \in \RR/2\pi \ZZ$, there is a symplectomorphism $M_{a,t}' \to M_{a,t}'$
 given by 
 $z_1 \mapsto e^{i \phi}z_1$, $z_2 \mapsto e^{i \phi}z_2$, $z_3 \mapsto e^{i (2\phi)}(z_3-a)+a$.
 Therefore, if the domain of $r$ in \eqref{eq:LAGDisk} is replaced by $e^{i \phi}[0,\infty)$ for some $\phi \in \RR/2\pi \ZZ$, Lemma \ref{l:LegendrianUnknot} still holds.
\end{remark}
Having reviewed some basic contact geometry,
now we explain the construction of Lagrangian solid tori under the Assumption \ref{a:Disc} and \ref{a:TubularNbhd} below.

\subsubsection{Overview of the construction}
\begin{wrapfigure}[11]{r}{0.26\textwidth}
\begin{center}
{\includegraphics[width=0.22\textwidth]{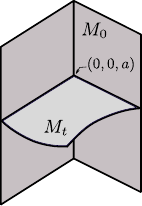}}
\end{center}
%\caption{caption}
%\label{Fig}
\end{wrapfigure}
In one dimension lower like the situation we just considered, 
let $z_1,z_2,z_3$ be symplectic coordinates of $\CC^3$ and suppose we have a family $M_t$ of hypersurfaces in $\CC^3$ with $M_0=\{z_1z_2=0\}$,
$M_t$ a symplectic submanifold for $t\neq 0$, for all $t\neq 0$ the discriminant $M_t\cap\Sing(M_0)$ equals $\{(0,0,a)\}$ for fixed $a\in\CC^*$. Say we have two balls $V, U$ centered at $\{(0,0,a)\}$ with $\bar V\subset U$ so that
\begin{align}
 &M_t \cap (U \setminus V)=\{z_1z_2=t(z_3-a)\} \label{eq:3spheres}, \\
 &M_t \cap U \text{ is a Liouville domain} \label{eq:LD}.
\end{align}
By \eqref{eq:3spheres} and Lemma~\ref{l:3sphere}, we know that $\partial(M_t\cap U)$ is the standard contact $3$-sphere.
By Lemma~\ref{l:LegendrianUnknot}, we have a Legendrian unknot
\begin{align}
 \Lambda_{r}:=\left\{\left.\Big(re^{i\theta},re^{-i\theta},\frac{r^2}{t}+a\Big)\right|\theta \in \mathbb{R}/2\pi \mathbb{Z}\right\}
\end{align}
inside $\partial(M_t\cap U)$ for some appropriate $r$.
Furthermore, by \eqref{eq:LD} and Theorem \ref{t:filling}, we know that $M_t\cap U$ is symplectic deformation equivalent to the standard symplectic ball when $t \neq 0$.
Moreover, by Lemma \ref{l:closing}, we know that we can fill $\Lambda_{r}$ by a Lagrangian disk in $M_t\cap U$.
This Lagrangian disk will generally allow us to construct closed Lagrangian surfaces for a tropical curve ending at the discriminant with such a disk closing up the ending.

In the situation that interests us one dimension higher, the ending needs to be given by a solid 3-torus.
This situation is considerably harder for the following reason.
Ideally, we would like to have a product situation locally.
It means that there is a symplectic annulus $(A,\omega_A)$ such that the family is simply given by $M'_t\times A$ where $M'_t$ is as above, the discriminant is then $\{(0,0,a)\}\times A$, we obtain a Lagrangian disk $D$ in the first factor $M'_t$ as before and then for any circle $C$ in $A$ that generates the fundamental group of $A$, we find $D\times C$ as the desired solid torus in $M'_t\times A$.
However, it is very hard to understand the symplectic form near the discriminant, not to mention to try to deform it to a product situation, so this easy setup won't be achievable for us. The next weaker concept from a product is a fibration which is what we will be using instead, as follows.

Let $z_1,z_2,z_3,z_4$ be symplectic coordinates of $\CC^4$, $M_t$ a family of hypersurfaces in $\CC^4$ with $M_0=\{z_1z_2=0\}$,
$M_t$ a symplectic submanifold for $t\neq 0$, and for all $t\neq 0$ the discriminant $M_t\cap\Sing(M_0)$ equals $\{(0,0,a)\}\times A$ for fixed $a\in\CC^*$ and some $0$-centered annulus $A\subset\CC$. 
Again, say we have two balls $V', U'\subset\CC^3$ centered at $\{(0,0,a)\}$ with $\bar V' \subset U'$ so that setting $U=U'\times A, V=V'\times A$,
\begin{align}
 &M_t \cap (U \setminus V)=\{z_1z_2=t(z_3-a)\} \label{eq2:3spheres}, \\
 &M_t \cap U \text{ is a Liouville domain} \label{eq2:LD}.
\end{align}

We will show below that the restriction of the projection $U \to A$ to $M_t$ gives a ``nice'' exact symplectic fibration $\pi:M_t\cap U \to A$. 
Every fiber of $\pi$ is the lower-dimensional situation as above. 
After symplectic completion, we get an exact symplectic fibration $\Comp(\pi):\Comp(M_t\cap U) \to T^*S^1$
such that fibers are standard symplectic $\RR^4$. 
Since the compactly supported symplectomorphism group of standard $\RR^4$ is trivial, we can find
a compactly supported exact symplectic deformation from $\Comp(M_t\cap U)$ to $\Comp(M_t\cap U)'$ such that $\Comp(\pi)$ is still an exact symplectic fibration
and the symplectic monodromy around a simple loop $C\subset A$ is the identity.
Therefore, we can construct a Lagrangian disk as above in a fiber of a point of $C$ and apply symplectic parallel transport along $C$ to get a Lagrangian solid torus in $\Comp(M_t\cap U)'$. Since $\Comp(M_t\cap U)'$ and $\Comp(M_t\cap U)$ are related by a compactly supported exact symplectic deformation, we get a corresponding Lagrangian torus in $\Comp(M_t\cap U)$ and we can apply the backward Liouville flow to obtain a Lagrangian solid torus in $M_t\cap U$.

In the sections below, we will explain this construction in more details.

\subsubsection{Main construction}
Let $U$ be a symplectic corner chart such that Assumption \ref{a:singularModel} holds. 
Let $s$ be an $s_1$-admissible section and let $T,\shA,\gamma$ and $\pi:U\ra\coDelta$ be the ones from the beginning of the chapter.
By Lemma \ref{l:fiberBundle}, the $q_3$-coordinate of $\Disc(s)\cap T$ is a constant and we denote it simply by $q$.
We are interested in the circle in the discriminant that lies above the point where $\gamma$ hits its amoeba image $\shA$.
Also by Lemma~\ref{l:fiberBundle}, we find $C:=\pi^{-1}(\gamma(r_1)) \cap \Disc(s_1)\cap T$ to be a circle with constant radial coordinate, say given by $|p_4|=R\iff |z_4|=\sqrt{2R}$.
So in $z$-coordinates, by setting $a=\sqrt{2r_1}e^{iq}$, the circle $C$ is given by 
\begin{align}
C=\left\{z=(0,0,a,z_4)\left| |z_4|=\sqrt{2R}\right.\right\}.
\end{align}

We will construct a Lagrangian solid torus inside $M_t \cap U_C$ for an appropriate closed neighborhood $U_C$ of $C$.
%Therefore, we need to first introduce some notations to describe closed neighborhoods of $C$.
\begin{comment}
\begin{notation}\label{n:BAD}
We will use the following notations
\begin{align}
 B^2_r &= \{z \in \mathbb{C}| |z| \le \sqrt{2r}\} \\
 A^2_{a_0,a_1} &= \{(p,q) \in \RR \times \RR/2\pi\ZZ| a_0 \le p \le a_1\} \\
 D^2_{(p_0,q_0), \epsilon} &= \{(p,q) \in \RR \times \RR/2\pi\ZZ| |p-p_0|, |q-q_0| \le \epsilon\}
\end{align}
where $a_0< a_1$, $\epsilon<2\pi$ and $(p_0,q_0) \in \RR \times \RR/2\pi\ZZ$.
We use $B^2$, $A^2$ and $D^2_{(p_0,q_0)}$ respectively
to denote $B^2_r$, $A^2_{a_0,a_1}$ and $D^2_{(p_0,q_0), \epsilon}$ 
for some appropriate choices of $r,a_0,a_1,\epsilon$ that are not specified.
%With these notations, a symplectic corner chart is given by $U=(B^2)^n$ where the $B^2$ in different factors of $U$ can have different sizes. 
\end{notation}
\end{comment}
From now on, every tubular neighborhood of $C$ that we choose will be {\bf closed} and of the form
\begin{align}\label{qe-U_C}
 U_C:=B\times B\times D\times A \subset (\RR^2)^4
%\{((p_1,q_1), \dots,(p_4,q_4)) \in B^2_{r_1} \times B^2_{r_2} \times D^2_{a,r_3} \times A^2_{a_0,a_1}\} \subset U \label{Ubeta0}
\end{align}
for $B$ a $0$-centered disk, $D$ an $a$-centered disk and $A$ a $0$-centered annulus (shrinking and then taking closure of the one we had before) so that the circle of radius $\sqrt{2R}$ is contained in $A$.
We will make two further assumptions for which we will show in later sections how these can be achieved.
The first assumption is that $\Disc(s)\cap T$ depends only on the $z_4$-coordinate near $C$ as illustrated on the right in Figure~\ref{fig:TropicalNearSing}.

\begin{assumption}\label{a:Disc}
 There exists a tubular neighborhood $U_C$ of $C$ such that
\begin{align}
 \Disc(s)\cap T \cap U_C=\{(0,0,a, z_4) \in U_C | z_4\in A\} \label{eq:Disc0Std}
\end{align} 
for $A$ a shrinking of the previous annulus $A$ still containing the circle of radius $\sqrt{2R}$.
In particular, $\shA \cap \pi(U_C)=\{(0,0,r_1)\}\times I$ where $I$ is a straight line segment in the affine $p$-coordinates of $\coDelta$ and $I$ is given by projecting the radial part of $A$.
\end{assumption}

To construct the Lagrangian solid torus, we also need to make an assumption on the restriction of $\pi$ to $U_C\cap M_t$ and for that we introduce the following notion.

\begin{definition}\label{d:exactSymFib0}
Let $(E,\omega_E)$ be a symplectic manifold with corners and $(\Sigma,\omega_\Sigma)$ be a symplectic surface with boundary.
Let  $\pi:E \to \Sigma$ be a symplectic fibration.
The {\bf vertical boundary} of $\pi$ is $\partial^vE:=\pi^{-1}(\partial \Sigma)$.
The {\bf horizontal boundary}  $\partial^hE$ of $\pi$ is the closure of $\partial E \setminus \partial^vE$.
The fibration $\pi$ is called a {\bf smoothly trivial exact symplectic fibration} if 
 \begin{enumerate}
  \item $\pi$ is a smoothly trivial fiber bundle,
  \item there is a one form $\alpha_E$ such that $d\alpha_E=\omega_E$ and the induced Liouville vector field points outward along  $\partial^vE$ and $\partial^hE$,
  \item there exists a neighborhood $N$ of $\partial^hE$ and a symplectic manifold $(F,\omega_F)$ with smooth boundary
  such that there is a symplectomorphism $\Psi:(N,\omega_E|_N) \simeq (F \times \Sigma, \omega_F \oplus \omega_\Sigma)$ 
  and $\pi_\Sigma \circ \Psi=\pi|_N$, where $\pi_\Sigma:F \times \Sigma \to \Sigma$ is the projection to the second factor.
 \end{enumerate}
The last condition is also referred to as $\pi$ being \emph{symplectically trivial near the horizontal boundary}.
\end{definition}

\begin{remark}
 A smoothly trivial exact symplectic fibration is a strictly more restrictive notion than that of an exact symplectic fibration as given in \cite[Section $15$]{SeidelBook}.
\end{remark}

\begin{assumption}\label{a:TubularNbhd}
 There exist tubular neighborhoods 
 \begin{align}
U_C:=B\times B\times D\times A
\quad\text{ and }\quad
V_C:=B'\times B'\times D'\times A  \label{eq:Disc0Std2}
 \end{align}
with $U_C$ as in \eqref{qe-U_C} and $B'\subset B$ a $0$-centered closed disk of smaller radius and $D'\subset D$ an $a$-centered closed disk of smaller radius but $U_C$ and $V_C$ have notably the same $A$-factors such that
 \begin{align}\label{eq:localModel}
  \left\{
  \begin{array}{ll}
   M^s_t \cap (U_C \backslash V_C)= \{z_1z_2=t(z_3-a)\} \text{ and}\\
   \pi: M^s_t \cap U_C \to A \text{ is a smoothly trivial exact symplectic fibration}
  \end{array}
\right.
 \end{align}
where, as usual, $z_j=\sqrt{2 p_j}e^{iq_j}=x_j+iy_j$ for $j=1,2,3,4$.
\end{assumption}

We next carry out the Lagrangian solid torus construction under Assumption~\ref{a:Disc} and \ref{a:TubularNbhd} 
(in fact, we only use Assumption \ref{a:TubularNbhd} for the Lagrangian construction and we will see in Sections \ref{ss:ASymFib}-\ref{ss:aGoodDeformation} that 
Assumption~\ref{a:Disc} is used to obtain Assumption \ref{a:TubularNbhd}).
By Lemma \ref{l:3sphere}, the contact boundary of fibers of the projection $\pi: M_t\cap U_C\ra A$ are contactomorphic to the standard contact $3$-sphere.
It implies that $\pi$ is actually a symplectic $4$-ball bundle over $A$ by Theorem \ref{t:filling}.
Moreover, by Lemma \ref{l:LegendrianUnknot}, for every $z_4 \in A$, there is a unique $r>0$ such that $r^2+a\in\partial D$ for $D$ the third factor of $U_C$. When $t\in\RR_{>0}$ is small,
\begin{align}
\Lambda_{z_4,r}:=\left\{(\left.\sqrt{t}re^{i\theta},\sqrt{t}re^{-i\theta},r^2+a,z_4)\right|\theta \in \mathbb{R}/2\pi \mathbb{Z}\right\} \label{eq:LegUnknot}
\end{align}
is a Legendrian unknot in $\partial (\pi^{-1}(z_4))$. (By Remark \ref{r:FlexibleDisk}, we can also take $r \in \CC^*$ with non-zero argument.)
Recall that $R$ is the $p_4$-coordinate of $\gamma(r_1)$.
Since $\pi$ is assumed to be symplectically trivial near the horizontal boundary, it is clear that 
\begin{align}\label{eq:boundaryLeg}
\Lambda_{C,r}:=\bigcup_{z_4:|z_4|=\sqrt{2R}} \Lambda_{z_4,r}  
\end{align}
is a Legendrian torus in the contact boundary
of $M_t \cap U_C$ (after rounding corners to be able to call $M_t \cap U_C$ a Liouville domain even though $\Lambda_{C,r}$ doesn't meet any corners since it projects to the interior of $A$).

\begin{proposition}\label{p:LagrangianConstructionSingLoci}
The Legendrian torus $\Lambda_{C,r}$ bounds an embedded Lagrangian solid torus $L_{C,r}$ in $M_t \cap U_C$ such that every $\Lambda_{z_4,r}$ is a meridian.
\end{proposition}

\begin{proof}
We use the notation $M'_t:=M_t \cap U_C$ in this proof.
Let $\Comp(M'_t)$ be the symplectic completion of $M'_t$.
In other words,
\begin{align}
 \Comp(M'_t):=M'_t \cup_{\partial M'_t } ([1,\infty) \times \partial M'_t )
\end{align}
and the symplectic form on $([1,\infty) \times \partial M'_t )$ is given by $d(\rho \alpha|_{\partial M'_t})$ for $\rho$ the linear coordinate on $[1,\infty)$ and 
$\alpha$ the one-form on $M'_t$ defining its Liouville structure.
Since $\pi$ is trivial near the horizontal boundary (third item of Definition \ref{d:exactSymFib0}), $\Comp(M'_t)$ can be obtained by first performing symplectic completion along the fibers of $\pi$ and 
then completing along the base direction.
Therefore, we have a symplectic $\RR^4$-bundle $\Comp(\pi):\Comp(M'_t) \to \CC^*$ extended from $\pi$.
We also have a Lagrangian submanifold $[1,\infty) \times \Lambda_{C,r} \subset [1,\infty) \times \partial M'_t \subset \Comp(M'_t)$ fibering over 
the circle $\{|z_4|=\sqrt{2R}\}$ with respect to $\Comp(\pi)$.

Gromov showed that the compactly supported symplectomorphism group of $(\RR^4,\omega_{std})$ is contractible \cite{Gromov}.
Therefore, there exists an exact symplectic deformation $\Comp(M'_t)'$ of $\Comp(M'_t)$ supported inside a compact set 
$K \subset \Comp(M'_t)$ such that after the deformation, 
$\Comp(\pi):\Comp(M'_t)' \to \CC^*$ is still a symplectic $\RR^4$-bundle and
the symplectic monodromy along $\{|z_4|=\sqrt{2R}\}$ defined by symplectic parallel transport becomes the identity (see \cite[Lemma $15.3$]{SeidelBook}).

Pick a point $z_4 \in A$ such that $|z_4|=\sqrt{2R}$.
There exists $\rho_0>1$ sufficiently large such that $[\rho_0,\infty) \times \Lambda_{z_4,r} \subset [1,\infty) \times \partial M'_t$ is disjoint from $K$.
Since $\{\rho_0\} \times \Lambda_{z_4,r}$ is a Legendrian isotopic to the standard Legendrian unknot
in the relevant contact hypersurface $S^3$ of $\Comp(\pi)^{-1}(z_4)=(\RR^4,\omega_{std})$, the proper annulus
$[\rho_0,\infty) \times \Lambda_{z_4,r}$ can be extended to a smooth proper Lagrangian disk $L_{z_4,r}$ in $\Comp(\pi)^{-1}(z_4)$, by 
Lemma \ref{l:closing} (note that when a Legendrian is Lagrangian fillable, one can always perturb the Lagrangian filling near the Legendrian boundary to get another Lagrangian filling that is cylindrical near its Legendrian boundary, therefore the Lagrangian disk $L_{z_4,r}$ can be made to be smooth).
We engage $L_{z_4,r}$ in symplectic parallel transport along $\{|z_4|=\sqrt{2R}\}$.
The fact that the monodromy is the identity implies that the trace of $L_{z_4,r}$ is an embedded proper Lagrangian open solid torus, denoted by $L_{C,r}'$, with a cylindrical end 
$[\rho_0,\infty) \times \Lambda_{C,r}$.
Since $\{\rho_0\} \times \Lambda_{z_4,r}$ bounds a disk in $L_{z_4,r}$, it is a meridian of $L_{C,r}'$.

Finally, since $\Comp(M'_t)'$ is a compactly supported
exact symplectic deformation of $\Comp(M'_t)$, there is also an embedded proper Lagrangian solid torus $L_{C,r}'' \subset \Comp(M'_t)$ with the cylindrical end $[\rho_1,\infty) \times \Lambda_{C,r}$
for some sufficiently large $\rho_1$.
Therefore, one can argue using backward Liouville flow as in the proof of Lemma \ref{l:closing} to rescale $L_{C,r}''$ and make its cylindrical part as long as we like.
We consequently obtain a Lagrangian filling $L_{C,r}$ of $\Lambda_{C,r}$ inside $M'_t$ with the properties required in the proposition.
\end{proof}

\subsubsection{Plan for the remaining part}
In the following subsections, we will generalize Proposition \ref{p:LagrangianConstructionSingLoci}.
In Section \ref{ss:Integral linear transform} and \ref{ss:Correcting discriminant}, we explain how to isotope the discriminant of $s$ so that Assumption \ref{a:Disc} holds.
In Section  \ref{ss:ASymFib}, \ref{ss:Liouville} and \ref{ss:aGoodDeformation}, we construct a smoothly trivial exact symplectic fibration such that Assumption \ref{a:TubularNbhd} holds.
We conclude the proof of Theorem \ref{t:LagrangianSolidTori} in Section \ref{ss:ProofSolidTori}.
The proof of Theorem \ref{t:Construction} is given in Section \ref{ss:Concluding}.

\subsection{Integral linear transform}\label{ss:Integral linear transform}
We go back to the general setup in Theorem \ref{t:LagrangianSolidTori}.
In particular, we have a parametrized straight line segment $\gamma$ with $\gamma(r_1)\in\shA$.
In this section, we want to apply an integral linear transformation to transform the $(p,q)$-coordinates
to obtain new $(\hat{p},\hat{q})$-coordinates so that $\sum_{i=1}^4 q_i=\hat{q}_1+\hat{q}_2$.
This will help us to get rid of the fourth-coordinate in the defining equation of $M_t^s \cap U_C$ later on.
Define
 \[
   \hat{A}:=
  \left[ {\begin{array}{cccc}
   1 & 0 & 0 & 0\\
   0 & 1 & 0 & 0\\
   0 & -1& 1 & 0\\
   0 & -1 & 0 & 1
  \end{array} } \right]\hbox{, so the inverse transpose}\quad
   \hat{A}^{-T}=
  \left[ {\begin{array}{cccc}
   1 & 0 & 0 & 0\\
   0 & 1 & 1 & 1\\
   0 & 0 & 1 & 0\\
   0 & 0 & 0 & 1
  \end{array} } \right].
\]
Consider a change of symplectic coordinates $\Psi^\circ:U^\circ=U \setminus \{p_1\dots p_4=0\} \to \hat{U}^{\circ}:=\im(\Psi^\circ) \subset \RR^4 \times \RR^4/2\pi\ZZ^4$
given by the integral linear transform 
$\Psi^\circ(p,q)=(\hat{p},\hat{q})=(\hat{A}p,\hat{A}^{-T}q)$. 
Note that, for $j=1,2$, we have $\hat{p}_j=p_j$ so we can define $\hat{z}_j:=\sqrt{2 \hat{p}_j}e^{i\hat{q}_j}$
and partially compactify $\hat{U}^{\circ}$ to $\hat{U}$ by  allowing $\hat{z}_j=0$ for $j=1,2$.
We can smoothly extend $\Psi^\circ$ to $\Psi:U \setminus \{p_3p_4=0\} \to \hat{U}=\im(\Psi) \subset \CC^2 \times \RR^2 \times  \RR^2/2\pi\ZZ^2$.
More explicitly, 
\begin{align}
\Psi(z)=
(\sqrt{2p_1}e^{iq_1}, \sqrt{2p_2}e^{i(q_2+q_3+q_4)}, 
-p_2+p_3, -p_2+p_4, q_3,q_4). \label{eq:Psi}
\end{align}
Note that $\Psi|_{\{z_1=z_2=0\}}$ is the identity map.
Therefore, just like before, by Lemma \ref{l:fiberBundle}, $\Disc(s_1)\cap T$ has the constant $\hat{q}_3$-coordinate $\arg(a)$ and $\shA$
is transverse to the slices $\{\hat{p}_4=const\}$.
The straight line $\hat{\gamma}(r):=\Psi(\gamma(r))$ is still given by $(0,0,r,R)$ for $r_0<r \le r_1$, and 
$C=\{\hat{z}=(0,0,a,\hat{z}_4) |\, |\hat{z}_4|=\sqrt{2R}\}$.

For use in the next section, we now apply the transformation to the pencil.
Observe that we achieved $\sum_{i=1}^4 q_i=\hat{q}_1+\hat{q}_2$ and have
\begin{equation} \label{eq-phat-pnohat}
p_3=\hat p_2+\hat p_3, \qquad p_4=\hat p_2+\hat p_4.
\end{equation}
Inserting this and more broadly $z_j=\Psi^{-1}(\hat z_j)$ into Equation~\eqref{eq-nonvanishing-factor-in-p} in Example \ref{ex:Transforming hypersurfaces} yields
\begin{equation}
 2\hat z_1\hat z_2 \sqrt{(\hat p_2+\hat p_3)(\hat p_2+\hat p_4)} = th(\hat p)g(w(\hat p,\hat q)).
\label{Psi-transform-of-pencil}
\end{equation}

\subsection{Straightening the discriminant}\label{ss:Correcting discriminant}
We assume from now until Section \ref{ss:ProofSolidTori} that we have performed the transformation $\Psi$ given in the previous subsection Section \ref{ss:Integral linear transform}.
For better readability, we will use the notation $p$ instead of $\hat{p}$, $z$ for $\hat z$ and so forth.

Our next step is to apply a compactly supported Hamiltonian diffeomorphism to deform $\Disc(s_1)\cap T$ 
such that the ${p}_3$-coordinate of $\Disc(s_1)\cap T$ becomes independent of the ${p}_4$-coordinate near $C$.
In other words, we want that Assumption \ref{a:Disc} holds after deforming $\Disc(s_1)\cap T$.

Similar to \eqref{qe-U_C}, we use a tubular neighborhood of $C$ of the form
\begin{align}
 U_C:=\{(({p}_1,{q}_1), \dots,({p}_4,{q}_4)) \in B\times B\times D \times A \subset {U} \label{Ubeta}
\end{align}
and taken small enough so that ${p}_3$ and ${p}_4$ take positive values in $U_C$ which works by \eqref{eq:Psi}.
It is then sensible to define ${z}_j=\sqrt{2{p}_j}e^{i{q}_j}={x}_j+i{y}_j$ in $U_C$ for $j=1,2,3,4$.

\begin{proposition}\label{p:correctSingModel}
 For any tubular neighborhood $N$ of $C$, there is a Hamiltonian diffeomorphism $\phi_H:\PP_\Delta \to \PP_\Delta$
 supported inside $N$ and a tubular neighborhood $U_C \subset N$ of $C$ given by \eqref{Ubeta} such that 
 \begin{itemize}
  \item $\phi_H$ preserves all the toric strata of $ \PP_\Delta$ setwise,
  \item $\phi_H(\Disc(s_1)\cap T) \cap U_C=\{(0,0,a,{z}_4)|{z}_4 \in A\}$, and
  \item ${z}_j= {z}_j \circ \phi_H^{-1}$ inside $U_C$ for $j=1,2$.
 \end{itemize}
\end{proposition}

After establishing  Proposition \ref{p:correctSingModel}, we push-forward all the data and define
\begin{align}
\hat{J}_\Delta&:=(\phi_H)_*J_\Delta  \label{eq-hat-first}\\
\hat{s}_i&:=s_i \circ \phi_H^{-1} \\
\hat{\shL}&:=(\phi_H^{-1})^*\shL \\
\Disc(\hat{s}_1)&:=(\hat{s}_1)^{-1}(0) \cap (\partial \PP_\Delta)_{\Sing}=\phi_H(\Disc(s_1)) \\
\hat{M}_t^s&:=\{\hat{s}_0=ts\} \text{ for }s \in C^{\infty}(\PP_\Delta,\hat{\shL}) \\
\hat{M}_t&:=\hat{M}_t^{\hat{s}_1}= \phi_H(M_t) \label{eq-hat-last}
\end{align}
In particular, $\hat{J}_\Delta$ is a complex structure on $\PP_\Delta$, and $\hat{s}_i$ are holomorphic sections of the holomorphic bundle $\hat{\shL}$.
Therefore, it makes sense to talk about $\hat{s}_1$-admissible sections (which are the same as $s_1$-admissible sections precomposed by $\phi_H^{-1}$). 
Most notably, by the second bullet of Proposition \ref{p:correctSingModel}, $\Disc(\hat{s}_1)\cap T \cap U_C$ satisfies Assumption~\ref{a:Disc} in $({p}, {q})$-coordinates.

\begin{remark}[Trivalent vertex]\label{r:trivalentVertex2}
As mentioned in Remark \ref{r:trivalentVertex1}, the key difficulty to generalize Theorem \ref{t:Construction}
to tropical curves with ends on a codimension one part of $\shA$ is whether one can establish the corresponding result of Proposition \ref{p:correctSingModel}.

More precisely, supposed we are given the local model \eqref{eq:singularModel2}
and a straight line segment $\gamma(r)=(0,0,r,R)$ parametrized by $r\in (r_0,r_1]$
such that $\gamma(r) \in \shA$ if and only if $r=r_1$.
Let $(0,0,a,b) \in \Disc(s_1)$ such that $\pi_{\Delta}(0,0,a,b)=\gamma(r_1)$
and let $C=\{(0,0,a,z_4):|z_4|=|b|\}$. We define neighborhood $U_C$ of $C$ as above.
If Proposition \ref{p:correctSingModel} is true in this setup, which means that it is
true for all the ends of a tropical curve, then the Lagrangian construction in Theorem \ref{t:Construction} applies to the tropical curve.

With that said, it is tempting to try to mimic the proof of Proposition \ref{p:correctSingModel} below to make 
$\shA$ to be very close to a trivalent graph and if $\gamma(r_1)$ is not the trivalent point of the graph, we would be able
to get a Hamiltonian $\phi_H$ satisfying all the three bullets of  Proposition \ref{p:correctSingModel}.
However, such a $\phi_H$ is not supported inside $N$.
For $\phi_H$ to be supported inside $N$, we can only perturb $\Disc(s_1)$ in $N$ and hence cannot shrink $\shA$ to a trivalent graph.

If one uses a $\phi_H$ that is not supported inside $N$ to run the rest of the argument, one can still get a closed 
Lagrangian that is diffeomorphic to a Lagrangian lift of the tropical curve but one cannot control the $\pi_{\Delta}$-image of the Lagrangian to be in a small neighborhood of the tropical curve.

It is very possible that Proposition \ref{p:correctSingModel} for appropriate $\gamma(r)$ is true in this setup.
Even though it is a very explicit local question, we are not able to write down a clean condition 
on $\gamma$ for it to work, especially when $\gamma(r_1)$ is very close to `the trivalent point of $\shA$'.

\end{remark}

Before giving the proof of Proposition \ref{p:correctSingModel}, we first conclude the resulting local model of $\hat{M}_t \cap U_C$. We remind the reader that $\hat{z}$ and $\hat{p}$ in the previous section are denoted by $z$ and $p$ in this section.

\begin{lemma}\label{l:localModelAfterCorrection}
 Let $\phi_H$ and $U_C$ be chosen as in Proposition \ref{p:correctSingModel}.
 Then we have
 \begin{align*}
  \hat{M}_t \cap U_C=\{{z}_1{z}_2=t{g}_U({z})\}
 \end{align*}
for some smooth function ${g}_U:U_C \to \mathbb{C}$ such that
\begin{itemize}
\item ${g}_U=g$ up to a change of coordinates and a multiplication by a non-vanishing function: more precisely,
${g}_U=\rho({z})g(\Phi_U \circ \Psi^{-1}(\phi_H^{-1}({z})))$ for some  $\rho({z}):U_C \to \mathbb{C}^*$,
 \item the zero locus of $({g}_U)_0:={g}_U|_{\{{z}_1={z}_2=0\}}$ is given by $\Disc(\hat{s}_1)\cap T \cap U_C$,
 \item $({g}_U)_0$ is submersive (i.e.~$D({g}_U)_0$ surjective) near $\Disc(\hat{s}_1)\cap T \cap U_C$,
 \item $({g}_U)_0|_{(D \setminus \{a\}) \times A}$ is homotopic to $({z}_3,{z}_4)  \mapsto {z}_3-a$ as $\CC^*$-valued functions.
 \end{itemize}
\end{lemma}

\begin{proof}
By Assumption \ref{a:singularModel} and Equation~\eqref{Psi-transform-of-pencil}, we have
\begin{align*}
M_t \cap U_C =\{2z_1z_2\sqrt{( p_2+ p_3)( p_2+p_4)} = t h(p)g(w(p,q))\}.
\end{align*}
Since $(p_2+p_3)(p_2+p_4)>0$, we can rearrange the terms to get
 \begin{align*}
  M_t \cap U_C =\{z_1z_2 = t\tilde h(p)g(w(p,q))\}
 \end{align*}
where $\tilde h(p)=\frac{h(p)}{2\sqrt{( p_2+ p_3)( p_2+p_4)}}$.
Note that $\tilde h(p)$ is a non-vanishing positive function because the numerator and denominator are both positive.
On the other hand, by tracing back the definitions, we have $w(p,q)=\Phi_U \circ \Psi^{-1}(p,q)$.
Applying $\phi_H$ to $M_t$ corresponds to precomposing the coordinates in the defining equation by $\phi_H^{-1}$ so we have
\begin{align}
 \hat{M}_t \cap U_C =\{\tilde{z}_1\tilde{z}_2=t\tilde h(\phi_H^{-1}({z}))g(\Phi_U \circ \Psi^{-1}(\phi_H^{-1}({z})))\} \label{eq:locG}
\end{align}
where $\tilde{z}_i:={z}_i \circ \phi_H^{-1}$ for $i=1,2$.
If we define $\rho( z):=\tilde h(\phi_H^{-1}({z})) $, then by the third bullet of Proposition \ref{p:correctSingModel}, we get the first bullet of this lemma.

In $U_C$, from the first bullet of Proposition \ref{p:correctSingModel} and the discussion above, it is clear that $({g}_U)_0=g|_T$ up to a 
 change of coordinates and a multiplication by a non-vanishing function.
Therefore, $({g}_U)_0^{-1}(0)=\phi_H \circ \Psi \circ \Phi_U^{-1}((g|_T)^{-1}(0)) \cap U_C = \Disc(\hat{s}_1)\cap T \cap U_C$ which is exactly the second bullet.

We now consider the third bullet.
Since $\Phi_U$, $\Psi$ and $\phi_H$ are diffeomorphisms, it suffices to check that 
$D(g|_T)$ is submersive near $\Disc(s_1)\cap T$.
We can check it in the complex chart where $g|_T=a(b-w_3)$ and $\Disc(\hat{s}_1)\cap T=\{b=w_3,w_1=w_2=0,w_4 \neq 0\}$.
Therefore, the third bullet follows.

Finally, since $\phi_H$ is isotopic to the identity, in order to understand the homotopy class of $({g}_U)_0|_{(D\backslash \{a\}) \times A}$, in view of \eqref{eq:locG},
it suffices to understand the homotopy class of
\begin{align}
 \tilde h(p) g(\Phi_U \circ \Psi^{-1}({z}))|_{\{0\} \times \{0\} \times (D \backslash \{a\}) \times A} \label{eq:LongExp}
\end{align}
It is clear that $\tilde h(p)$ is null-homotopic
because on one hand, it is well-defined and non-vanishing on the whole $D$ factor, and on the other it is independent of the ${q}_4$-coordinate.
The homotopy class of the remaining term, $g(\Phi_U \circ \Psi^{-1}({z}))$, can be understood
by combining the fact that, away from the zero locus, $g|_T$ is homotopic to $(w_3,w_4) \mapsto w_3-b \in \CC^*$
and $q_3= v_3$ is preserved under $\Psi$ (see \eqref{eq:Psi}).
\end{proof}

\subsubsection{Proof of Proposition \ref{p:correctSingModel}}
Let $N\subset U_C$ be a tubular neighborhood of $C$ of a similar form as $U_C$. Under abuse of repeating notation, $N$ is thus given by
$$
 N=\{(({p}_1,{q}_1), \dots,({p}_4,{q}_4)) \in B \times B \times D \times A\} 
$$
Let $a_0,a_1$ be the radii of $A$ in the $p_4$-coordinate with $a_0<R<a_1$ for $R$ the radius of $C$.
By Lemma \ref{l:fiberBundle}, there exists a smooth $p_{\Disc}:[a_0,a_1] \to \RR_{>0}$ and a constant $q_{\Disc} \in [0,2\pi)$ such that 
$$
 \Disc(s_1)\cap T \cap N =\{{z}=(0,0,\sqrt{2p_{\Disc}({p}_4)}e^{i q_{\Disc}},{z}_4)|{z}_4 \in A\}
$$
In particular, by $C \subset \Disc(s_1)\cap T \cap N $, we have $\sqrt{2p_{\Disc}(R)}=|a|$ and $q_{\Disc}=\arg(a)$.

By ignoring the first two factors, we can view $\Disc(s_1)\cap T \cap N$ as a symplectic section of the projection $\pi:D \times A \to A$.
In the following lemma, we explain how to deform this symplectic section (denoted by $Z$ in the lemma) to another symplectic section that 
is locally constant near $\pi(C)$.
After that, we will explain in Lemma \ref{l:correctSingModel4cell} how to thicken this Hamiltonian isotopy inside $T$ to be a Hamiltonian isotopy
in $U_C$ to achieve Proposition \ref{p:correctSingModel}.

\begin{lemma}\label{l:correctSingModel2cell}
 Let $Z\subset D\times A$ be the image of the symplectic section 
$$A\ra D\times A, \quad z_4\mapsto (\sqrt{2p_{\Disc}(p_4)}e^{iq_{\Disc}},z_4)$$
of $\pi$.
There exists a Hamiltonian $H:D \times A \to \mathbb{R}$, 
 supported inside the interior of the domain $D \times A$, and a neighborhood $W$ of 
 $\pi(C)=\{p_4=R\}$ in $A$
 such that $\pi^{-1}(W) \cap \phi_H(Z)=\{(a,{z}_4)| {z}_4 \in W\}$ (see Figure \ref{fig:deformSection}).
\end{lemma}

\begin{proof}
  Consider the Hamiltonian 
$$
  H=(p_{\Disc}({p}_4)-p_{\Disc}(R))({q}_3-q_{\Disc}) \in C^{\infty}(D \times A).
$$
This is well-defined because the $q_3$-coordinate on $D$ is bounded in an interval for $D$ is $a$-centered with $a\neq 0$ and $D$ doesn't meet $z_3=0$.
The corresponding Hamiltonian vector field is (with the sign convention $dH=-\iota_{X_H} \omega$) given by
$$
 X_H= -(p_{\Disc}({p}_4)-p_{\Disc}(R)) \partial_{{p}_3}+ p_{\Disc}'({p}_4)({q}_3-q_{\Disc}) \partial_{{q}_4}.
$$
In particular, $X_H|_{C}=0$ and when the time $1$ flow $\phi_H$ is well-defined, we have
$$
 \phi_H({p}_3,{q}_3,{p}_4,{q}_4)=({p}_3+p_{\Disc}(R)-p_{\Disc}({p}_4),{q}_3,{p}_4,{q}_4+p_{\Disc}'({p}_4)({q}_3-q_{\Disc}))
$$
so $\phi_H(C)=C$
 and 
$\phi_H(Z)$ is a section over $A$ with $({p}_3,{q}_3)$-coordinates equal to $(p_{\Disc}(R),q_{\Disc})$.

Note that $H$ is not compactly supported (and $\phi_H$ is not everywhere well-defined).
In order to get a compactly supported Hamiltonian $\tilde H$, we need to multiply a cutoff function to $H$
of the form $\rho_1({p}_4) \rho_2({p}_3,{q}_3)$ such that $\rho_1({p}_4):A \to \RR$ equals to $1$ near $R$ and 
$\rho_2({p}_3,{q}_3): D \to \RR$ equals to $1$ near $(p_{\Disc}(R),q_{\Disc})$.
Now, for $\tilde H=\rho_1({p}_4) \rho_2({p}_3,{q}_3)H$,
it follows that for a sufficiently small neighborhood  $W \subset A$ of $\{p_4=R\}$, we will get
$\pi^{-1}(W) \cap \phi_{\tilde H}(Z)=\{(a,{z}_4)| {z}_4 \in W\}$.
\end{proof}

\begin{figure}
 \includegraphics{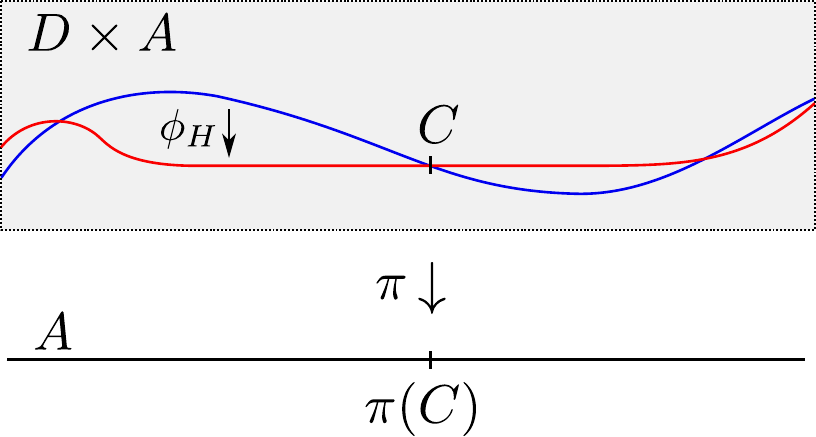}
 \caption{The symplectic section $Z$ (blue) is deformed to a section locally constant near $\pi(C)$ (red) after a compactly supported Hamiltonian diffeomorphism.}
\label{fig:deformSection}
 \end{figure}

We can thicken the constructed Hamiltonian as follows.

\begin{lemma}\label{l:correctSingModel4cell}
 As before, except with two extra ball factors $B$, let $Z:=\Disc(s_1) \cap N$ be the symplectic section of the fiber bundle $\pi:N \to A$ given by projection and $C=\{(0,0,a,{z}_4))||{z}_4|=\sqrt{2R}\}$.
 There exists a Hamiltonian $H:N \to \mathbb{R}$, supported inside the interior of $N$,
 and a neighborhood $W$ of  $\pi(C)$ in $A$
 such that $\pi^{-1}(W) \cap \phi_H(Z)=\{(0,0,a,{z}_4)| {z}_4 \in W\}$.
 Moreover, $\phi_H$ preserves $\{0\} \times B\times D \times A$,
$B \times \{0\}\times D \times A$ and $\{0\} \times \{0\}\times D \times A$
setwise.
\end{lemma}

\begin{proof}
 Let $h:B \to \mathbb{R}$ be a function supported inside the interior of $B$ such that $h\equiv 1$ near the origin.
 Let $H^0:D \times A\ra\RR$ be the Hamiltonian obtained via Lemma~\ref{l:correctSingModel2cell}.
 We define $H=h({z}_1)h({z}_2)H^0({z}_3,{z}_4)$, so $H$ is supported inside the interior of $N$.
 Moreover, the Hamiltonian vector field satisfies
 \begin{align}
  X_H|_{\{0\} \times \{0\} \times D \times A}&=X_{H^0} \label{eq:HamVectorField} \\
  X_H|_{\{0\} \times B \times D \times A}&=H^0({z}_3,{z}_4)X_{h_2}+h({z}_2)X_{H^0} \label{eq:HamVectorField2}\\
  X_H|_{B \times \{0\} \times D \times A}&=H^0({z}_3,{z}_4)X_{h_1}+h({z}_1)X_{H^0} \label{eq:HamVectorField3}
 \end{align}
where $X_{h_i}$ denotes the unique vector field that pushes down to $X_h$ in the $i$th $B$-factor and trivial to the other factors. We conclude the assertion.
\end{proof}

\begin{proof}[Proof of Proposition \ref{p:correctSingModel}]
Given a tubular neighborhood $N$ of $C$, we can apply Lemma \ref{l:correctSingModel4cell} to get a Hamiltonian diffeomorphism $\phi_H:\PP_\Delta \to \PP_\Delta$ supported inside $N$
such that $\phi_H$ preserves all tori strata setwise (so the first bullet of Proposition \ref{p:correctSingModel} holds).

If $U_C$ is a small tubular neighborhood of $C$ such that $\pi(U_C) \subset W$, where $W$ is obtained in Lemma \ref{l:correctSingModel4cell}, 
then the second bullet of Proposition \ref{p:correctSingModel} holds.

Finally, a simple but crucial observation is that Equation \eqref{eq:HamVectorField} is true \emph{near} $\{0\} \times \{0\} \times D \times A$.
Therefore, $\tilde{z}_i={z}_i $ near $\{0\} \times \{0\} \times D \times A$ for $i=1,2$.
By shrinking $U_C$, we obtain the third bullet of Proposition \ref{p:correctSingModel}. 
\end{proof}

\subsection{A symplectic fibration}\label{ss:ASymFib}

We assume from now until Section \ref{ss:ProofSolidTori} that we have applied the diffeomorphism $\phi_H$ given in
Proposition \ref{p:correctSingModel}.
For better readability, we will drop the `hat' notations.

In this subsection and the next two, we want to equip $\pi:{M}_t^s \cap U_C \to A$ with a smoothly trivial exact symplectic fibration structure for some appropriate
${s}_1$-admissible section $s$. After that, Assumption \ref{a:TubularNbhd} will be justified and we can apply Proposition~\ref{p:LagrangianConstructionSingLoci}
to get some Lagrangian solid torus.
As a first step towards this, we consider $s={s}_1$ and equip ${M}_t \cap U_C$ with a 
symplectic fiber bundle structure over $A$ (see Proposition \ref{p:fibration} below).
The main tool is the following linear algebra observation first made by Simon Donaldson (and known by the slogan ``almost holomorphic implies symplectic'').

\begin{proposition}[\cite{Donaldson}, Proposition $3$]\label{p:Donaldson}
 Let $\alpha:\mathbb{C}^n \to \mathbb{C}$ be an $\mathbb{R}$-linear map.
 Let $\alpha^{1,0}$ and $\alpha^{0,1}$ be the complex linear and the anti-complex linear parts of $\alpha$ respectively.
 If $|\alpha^{0,1}|<|\alpha^{1,0}|$, then $\ker(\alpha)$ is symplectic of rank $2n-2$ in $\mathbb{C}^n$.
\end{proposition}

Let 
%\begin{align}
$ G^t:={z}_1{z}_2-t{g}_U({z}): U_C \to \mathbb{C}$
%\end{align}
where ${g}_U$ and $U_C$ are obtained in Lemma \ref{l:localModelAfterCorrection} and Proposition \ref{p:correctSingModel}.
Since the tangent space of ${M}_t \cap U_C$ is given by $\ker(DG^t)$ for $t\neq 0$, analyzing $DG^t$
and how it is related to the projection $\pi$ will be the heart of this subsection.

In $({x},{y})$-coordinates (see the paragraph after \eqref{Ubeta}), we have 
\begin{align*}
 D{g}_U=[\partial_{{x}_1} {g}_U({x},{y}), \partial_{{y}_1} {g}_U({x},{y})
 , \dots, \partial_{{x}_4} {g}_U({x},{y}), \partial_{{y}_4} {g}_U({x},{y})]
\end{align*}
Let $D_3{g}_U:=[0,\dots,0,\partial_{{x}_3} {g}_U(z), \partial_{{y}_3} {g}_U(z),0,0]$ and
$D_4{g}_U:=[0,\dots,0,\partial_{{x}_4} {g}_U(z), \partial_{{y}_4} {g}_U(z)]$, which taken together form a $2 \times 8$ real matrix-valued function on $U_C$.
For $t\neq 0$, we know that ${M}_t \cap U_C$ is symplectic, or equivalently, $\ker(DG^t)$ is symplectic at all points
where $G^t=0$, because ${M}_t$ is a holomorphic submanifold.
If $D_4{g}_U\equiv 0$, then $G^t$ is independent of the ${z}_4$-coordinate, so factors as $U_C\ra B\times B\times D \stackrel{(G^t)'}{\longrightarrow} \CC$. If $N_t$ denotes the fiber of $(G^t)'$ over $0$ then
$M_t\cap U_C=N_t\times A$, a symplectic product.
While there is no reason to have $D_4{g}_U\equiv 0$, we are in fact going to show that if we ``remove'' the term $tD_4{g}_U$ from $DG^t$, then 
$\ker (DG^t+tD_4{g}_U)$ is still symplectic near $C$, and we show this implies that $\pi$ is a symplectic fiber bundle for $U_C$ sufficiently small.

\begin{lemma}\label{l:TiltingTangent}
 There exists a  tubular neighborhood $U_C' \subset U_C$ of $C$
 such that $\ker ((DG^t+tD_4{g}_U)|_{T_{{z}} U_C'})$ is symplectic of rank $6$ for all ${z} \in {M}_t\cap U_C'$ and all $t >0$.
\end{lemma}

\begin{proof}
 With the transformation after Proposition \ref{p:correctSingModel} implicit, we denote $\hat J_\Delta$ just by $J_\Delta$ etc. in the following. 
We have ${M}_t= \{{s}_0=t{s}_1\}$ and both ${s}_0$ and ${s}_1$ are holomorphic section and thus
 \begin{align*}
  DG^t \circ {J}_\Delta|_{T_{{z}} U_C}= J_{\CC} \circ DG^t|_{T_{{z}} U_C}
 \end{align*}
 for all ${z} \in {M}_t$ and all $t>0$, where $J_{\CC}$ is the standard complex structure of $\CC$.
 
 As a result, for ${z} \in {M}_t$ and any $\RR$-linear matrix
 $A: T_{{z}}U_C \to \mathbb{C}$, we have
 \begin{align*}
  (DG^t+A)^{1,0}&=DG^t+A^{1,0} \\
  (DG^t+A)^{0,1}&=A^{0,1}
 \end{align*}
 where superscripts $(1,0)$ and $(0,1)$ are the $({J}_\Delta,J_{\CC})$ complex linear part and anti-complex linear part, respectively, so
\begin{equation}
2A^{1,0}=A-J_{\CC}AJ_{\Delta}\quad\hbox{ and }\quad 2A^{0,1}=A+J_{\CC}AJ_{\Delta}
\label{eq-10-01}
\end{equation}
 Using the fact that $\|{J}_\Delta\|$ is uniformly bounded, applying triangle inequality to \eqref{eq-10-01} gives a $c_0>0$ such that 
 $\|A^{1,0}\|, \|A^{0,1}\| <c_0 \|A\|$ for every $\RR$-linear matrix $A$.
 Now assume additionally that $\|A\|<c_1\|DG^t({z})\|$ for some $c_1 >0$, then 
 \begin{align*}
  \| (DG^t+A)^{1,0}\| & \ge \| DG^t\| -\|A^{1,0}\|  \\
  &> \| DG^t\| -c_0\|A\| \\
  &> \| DG^t\| -c_0c_1\| DG^t\| \\
  &> \frac{(1-c_0c_1)}{c_0c_1}\|A^{0,1}\|  \\
  &=\frac{(1-c_0c_1)}{c_0c_1}\|(DG^t+A)^{0,1}\|
 \end{align*}
Hence, given $c_1>0$ (independent of $t$) such that $c_0c_1<\frac{1}{2}$, we have for all $A$ satisfying $\|A\|<c_1\|DG^t({z})\|$ for all $t$ that
$\| (DG^t+A)^{1,0}\| >\|(DG^t+A)^{0,1}\|$. In this case, $\ker(DG^t+A)$ is symplectic of rank $6$ for all $t$ by Proposition \ref{p:Donaldson}.
 
  By the second and third bullet of Lemma \ref{l:localModelAfterCorrection}, we know that $\partial_{{x}_4} {g}_U({z})=\partial_{{y}_4} {g}_U({z})=0$,
 and $\partial_{{x}_3} {g}_U({z}),\partial_{{y}_3} {g}_U({z}) \neq 0$ for ${z} \in \Disc({s}_1)$.
 Therefore, for any $c_1>0$ such that $c_0c_1<\frac{1}{2}$, there exist small neighborhood $U_C'$ of $C \subset \Disc({s}_1)$ such that
  \begin{align*}
  \|D_4{g}_U\| < c_1\|D_3{g}_U\|
 \end{align*}
 for all ${z} \in U_C'$.
 As a result, we have $c_1\|DG^t({z})\| \ge c_1\|tD_3{g}_U\| > \|tD_4{g}_U\|$ so $\ker(DG^t+tD_4{g}_U)$ is symplectic
 for all ${z} \in {M}_t\cap U_C'$ and for all $t > 0$.
 \end{proof}

 By shrinking the $U_C$ we chose in Proposition \ref{p:correctSingModel} if necessary, we can assume $U_C$ is small enough such that Lemma \ref{l:TiltingTangent} is satisfied and we will do so in the following.
 
\begin{proposition}\label{p:fibration}
 Let $C$ and $U_C$ be as before and let $\pi:{M}_t \cap U_C \to A$ be the restriction of the projection $\pi_4:U_C \to A$.
 We find that $\pi$ is a symplectic fibration without singularities.
\end{proposition}

\begin{proof}
Let $H^t({z})=(G^t({z}),\pi_4({z})): U_C\ra \CC\times A$. For all ${z}_4 \in A$ we get $F_{{z}_4}:=(H^t)^{-1}(0,{z}_4)=\pi^{-1}({z}_4)$.
Along $F_{{z}_4}$, we have 
$$\ker (DH^t|_{F_{{z}_4}})=\{v \in \ker(DG^t|_{F_{{z}_4}})\,|\,v_7=v_8=0\} \subset \ker((DG^t+tD_4{g}_U)|_{F_{{z}_4}})$$ 
where $v_7,v_8$ are the $7^{th}$ and $8^{th}$ entries of the vector $v$ respectively.
 Notice that $\ker((DG^t+tD_4{g}_U)_{F_{{z}_4}})=\ker (DH^t|_{F_{{z}_4}}) \oplus \mathbb{R}\langle v_7,v_8 \rangle$ 
and the left hand side has rank 6 by Lemma~\ref{l:TiltingTangent}, hence $\dim_\RR(\ker (DH^t|_{F_{{z}_4}}))=4$ and therefore $\pi$ has smooth fibers.
 
 Moreover, $\ker((DG^t+tD_4{g}_U)_{F_{{z}_4}})$ is symplectic by Lemma \ref{l:TiltingTangent}.
 It is clear that $\mathbb{R}\langle v_7,v_8 \rangle$ is symplectic and its symplectic orthogonal complement is $\ker (DH^t|_{F_{{z}_4}})$ so $\ker (DH^t|_{F_{{z}_4}})$ is also symplectic.
 \end{proof}

 \subsection{Liouville vector field}\label{ss:Liouville}
 
 We recall that $U_C=B \times B \times D \times A$.
 For $j=1,2$, let $(\varrho_j, \vartheta_j)=(\sqrt{2p_j},q_j)$ be the polar coordinates of the first two factors respectively.
 Using ${z}_3=\sqrt{2{p}_3}e^{i{q}_3}$, we can symplectically identify $D$ 
 with a closed disk in $\CC$ centered at $a$. 
 Translating by $a$, the polar coordinates on $\CC$ induces a polar coordinates $(\varrho_3, \vartheta_3)$ on $D$ with $\{\varrho_3=0\}=\{a\}$.
 We identify $A$ with a $S^1$-equivariant neighborhood of the zero section in $T^*S^1$ such that 
 $\{|{z}_4|=R\}$ is mapped to the zero section.
 Let $\varrho_4$ and $\vartheta_4$ be the fiber and base coordinates of $T^*S^1$ and hence coordinates on $A$.
 With these new notations, the symplectic form on $U_C$ can be re-written as $d \alpha$, where 
 \begin{align*}
  \alpha:=\sum_{i=1}^3 \frac{\varrho_i^2}{2}d\vartheta_i+\varrho_4 d\vartheta_4.
 \end{align*}
 We also have a Liouville vector field (see Subsection \ref{ss:CloseUp} for some background)
 \begin{align*}
  Z_{U_C}:=\sum_{i=1}^3 \frac{\varrho_i}{2}\partial_{\varrho_i}+\varrho_4 \partial_{\varrho_4 }
 \end{align*}
 pointing outward along $\partial U_C$ making $U_C$ a convex exact symplectic manifold (or equivalently, a Liouville domain).
 The restriction of $\alpha$ to ${M}_t \cap U_C$ induces a Liouville vector field $Z_{M_t}$ on it.
 We want to show that $Z_{M_t}$ points outward along the vertical boundary $\pi^{-1}(\partial A)$ of ${M}_t \cap U_C$.

 \begin{proposition}\label{p:verticalB} 
  Given $U'_C$ as in Proposition~\ref{p:fibration}. There exists a shrinking of the $B \times B \times D$-factor of $U'_C$ to obtain an open set $U_C$ such that $Z_{M_t}$ points outward along the
  vertical boundary of the fibration $\pi:{M}_t \cap U_C \to A$.
 \end{proposition}

 \begin{proof}
The Liouville vector field $Z_{U'_C}$ decomposes with respect 
 to $T{M}_t\oplus (T{M}_t)^{\omega}$ in say $Z_1+Z_2$.
For $v\in T{M}_t$, we have 
$$\alpha|_{M_t}(v)=\alpha(v,0)=\omega_{U'_C}(Z_1+Z_2,(v,0))=\omega_{U'_C}(Z_1,(v,0))=\omega_{M_t}(Z_1,v)$$
since $\omega_{U'_C}(Z_2,(v,0))=0$ by ${M}_t$ being symplectic in $U'_C$. This being true for all $v$, we conclude that $Z_{M_t}=Z_1$.
 
% We want to apply Lemma~\ref{l:omegaOrthogonal} for $a=DG^t$ and $e=tD_4{g}_U$. 
% Let $\eps>0$, so that the lemma spits out some $\delta$ now. 
% By the last part of to the proof of Lemma~\ref{l:TiltingTangent}, after shrinking $U_C$ if needed, we can assume that $\|tD_4{g}_U\| < \delta \|DG^t\|$ for all $t \neq 0$.
% Since $\varrho_4 \partial_{\varrho_4 } \in \ker(DG^t+tD_4{g}_U)$, by Lemma~\ref{l:omegaOrthogonal}, the $\ker(DG^t)^{\omega}$-component of $\varrho_4 \partial_{\varrho_4 }$, denoted by $(\varrho_4 \partial_{\varrho_4 })'$, satisfies 
% $|(\varrho_4 \partial_{\varrho_4 })'|< \epsilon |\varrho_4 \partial_{\varrho_4 }|$ for all $t\neq 0$.
 
% \red{By shrinking, we can assume that $|\frac{\varrho_i}{2}\partial_{\varrho_i}| \ll |\varrho_4 \partial_{\varrho_4 }|$ along the vertical boundary.
% As a result, $\frac{|Z_{U_C}'|}{|Z_{U_C}|}$ can be made  arbitrarily small so 
% $Z_{M_t}=Z_{U_C}-Z_{U_C}'$  is dominated by $\varrho_4 \partial_{\varrho_4 }$ and hence 
% points outward along the vertical boundary, where $Z_{U_C}'$ is the $\ker(DG^t)^{\omega}$-component of $Z_{U_C}$.}

Let $A_0:=\{\varrho_1=\varrho_2=\varrho_3=0\} \times A$ which lies inside $M_t$ for all $t$.
Note that $Z_{U'_C}=\varrho_4 \partial_{\varrho_4}$ on $A_0$ so it points outward along $\partial A_0$.
% Note also that $\{\varrho_1=\varrho_2=\varrho_3=0\} \times A \subset \Disc \subset M_t$ for all $t$.
% Now, by taking $\epsilon$ to be sufficiently small, $Z_{M_t}$ on $M_t \cap (\{\varrho_1=\varrho_2=\varrho_3=0\} \times A)$ also points outward along 
% $M_t \cap (\{\varrho_1=\varrho_2=\varrho_3=0\} \times \partial A)$  (actually, with respect to the standard metric, taking $\epsilon=\frac{1}{2}$ suffices because $\varrho_4 \partial_{\varrho_4 }$ is perpendicular to $\partial A$).
Note also that $Z_{U'_C}|_{A_0} \in TM_t$ so the $(T{M}_t)^{\omega}$-component of $Z_{U'_C}|_{A_0}$ is $0$, which in turn implies that 
$Z_{M_t}|_{A_0}$ points outward along $\partial A_0$.
 Since pointing outward along $M_t \cap (B \times B \times D \times \partial A)$ is an open condition, by shrinking the $B\times B \times D$ factor, 
 we can ensure that $Z_{M_t}$ points outward along $M_t \cap (B \times B \times D \times \partial A)$.
 
 \end{proof}

\subsection{A good deformation}\label{ss:aGoodDeformation}

We are going to construct a smoothly trivial exact symplectic fibration and justify Assumption \ref{a:TubularNbhd} in this subsection.
Ideally, we would like the symplectic fibration $\pi: {M}_t \cap U_C \to A$ to be a smoothly trivial exact symplectic fibration 
but it is not true in general that $\pi$ is trivial near the horizontal boundary even if we assume $U_C$ to be very small.
However, we can show that it is true after appropriately deforming ${s}_1$ to another ${s}_1$-admissible section which has been the whole purpose of introducing the notion of admissible sections.

\begin{proposition}[Homotoping into Assumption \ref{a:TubularNbhd}]\label{p:fibrationAfterGoodDeform}
 For any open neighborhood $N$ of $C$, there are tubular neighborhoods $U_{C},V_C \subset N$ as in \eqref{eq:Disc0Std2} so that $V_C \subsetneq U_C$ is a closed neighborhood of $\Disc({s}_1) \cap U_C$. 
The neighborhood
$U_C$ satisfies Proposition \ref{p:correctSingModel}, \ref{p:fibration} and \ref{p:verticalB}. 
There is also a smooth family $(s^u)_{u\in [0,1]}$ of ${s}_1$-admissible sections with $s^0={s}_1$ and for all $u$, $s^u={s}_1$ outside $N$ and
\begin{align}
 {M}_t^{s^1} \cap (U_C \backslash V_C)=\{{z}_1{z}_2= t({z}_3-a)\}. \label{eq:assumptionHorizontalTrivial}
\end{align}
Moreover, the projection to $A$, $\pi:{M}_t^{s^1} \cap U_C\ra A$  is a smoothly trivial exact symplectic fibration for all $t>0$ small.
\end{proposition}

Recall that every ${s}_1$-admissible section equals ${s}_1$ near $\Disc({s}_1)$ (hence is more messy than \eqref{eq:assumptionHorizontalTrivial}) and recall that $\Disc({s}_1) \cap U_C=\{0\} \times \{0\} \times \{a\} \times A$, so 
we cannot hope for \eqref{eq:assumptionHorizontalTrivial} to be true if $V_C \subset U_C$ is 
not a neighborhood of $\Disc({s}_1) \cap U_C$ which is why the $A$-factors of $V_C$ and $U_C$ agree in \eqref{eq:Disc0Std2}.

The proof of Proposition \ref{p:fibrationAfterGoodDeform} is divided into two steps.
The first step is the construction of $(s^u)_{u\in [0,1]}$ and $V_C$, and the second step is to justify that 
$\pi$ is a smoothly trivial exact symplectic fibration for all $t>0$ small.

\begin{proof}[Proof of Proposition \ref{p:fibrationAfterGoodDeform}: Step one]
Pick $U'_C\subset N$ sufficiently small such that Proposition \ref{p:correctSingModel}, \ref{p:fibration} and \ref{p:verticalB} are satisfied.
We shrink the A-factor of $U'_C$ to obtain an open set $U''_C$ that still satisfies Proposition \ref{p:correctSingModel} and \ref{p:fibration}. Finally,
apply Proposition~\ref{p:verticalB} to $U''_C$ to shrink its $B\times B\times D$-factor and arrive at an open set $U_C$ that also satisfies all three propositions like $U'_C$ and furthermore $U_C$ is contained in the interior of $U'_C$ which we will need later.

We work on $U_C$ now. By the last bullet of Lemma \ref{l:localModelAfterCorrection}, we know that
 $({g}_U)_0|_{ (D \setminus \{a\} ) \times A} \to \mathbb{C}^*$ is homotopic to ${z}_3-a:(D \setminus \{a\} ) \times A \to \mathbb{C}^*$.
 Therefore, there
  is no obstruction to constructing a smooth family of functions $(h^u)_{u \in [0,1]}: D \times A \to \mathbb{C}$ such that 
  \begin{itemize}
   \item $h^0=({g}_U)_0$,
   \item  $h^u$ is independent of $u$ near the discriminant $\{a\} \times A$,
   \item $(h^u)^{-1}(0) =\{a\} \times A$ for all $u$, and
   \item  there is a neighborhood of $V_0$ of $\{a\} \times A$ inside $D \times A$ such that $h^1|_{(D \times A) \setminus V_0}={z}_3-a$.
  \end{itemize}
  The second and third bullet above correspond to admissibility of sections, and the last bullet corresponds to \eqref{eq:assumptionHorizontalTrivial}.
  
  After $h^u$ is constructed, there is no obstruction to extend it to $g^u:U_C \to \mathbb{C}$ such that $g^0={g}_U$, for all $u$,
  $g^u|_{\{{z}_1={z}_2=0\}}=h^u$, $g^u$ is independent of $u$ near the discriminant $\Disc({s}_1) \cap U_C$
  and there exists a closed neighborhood $V_C \subset U_C$ of $\Disc({s}_1) \cap U_C$ such that $g^1|_{U_C \backslash V_C}={z}_3-a$.
  Indeed, note that we permit $g^u$ to take value $0$ outside $\{{z}_1={z}_2=0\} \cap U_C$.
This is because $\{{z}_1={z}_2=0\} \cap U_C$ is exactly the intersection between 
the $2$-dimensional toric strata and $U_C$ so
even if $g^u$ is $0$ somewhere in $U_C \setminus \{{z}_1={z}_2=0\}$, it will not create new discriminant (cf. the proof of Corollary \ref{c:firstApproximation}).
  
  With this understood, we can extend the isotopy $(g^u)_{u \in [0,1]}$ from $U_C$ to $U_C'$ so that it equals $g_U$ for all $u$ near the boundary of $U'_C$ as well as near the discriminant. 
Recall that ${M}_t \cap U_C'=\{{s}_0=t{s}_1\} \cap U_C'=\{{z}_1{z}_2=t{g}_U\} \cap U_C'$. 
We can patch $g^u$ with ${g}_U$ outside $U_{C}'$ to obtain a family of ${s}_1$-admissible sections $(s^u)_{u\in [0,1]}$
 such that $s^0={s}_1$, for all $u$, $s^u={s}_1$ outside $N$ and \eqref{eq:assumptionHorizontalTrivial} is satisfied on $U_C$.  
\end{proof}

\begin{proof}[Proof of Proposition \ref{p:fibrationAfterGoodDeform}: Step two]
 
 Now, we want to address why $\pi$ is a smoothly trivial exact symplectic fibration for all $t>0$ small. 
Let $\pi_4:U_C\ra A$ be the obvious projection
(note the difference with the $\pi$ above, namely, $\pi$ is the restriction of $\pi_4$ to $M_t^s \cap U_C$ for some $s$ but $\pi_4$ is defined on the entire $U_C$).
 We use the notation $M_t^u:={M}_t^{s^u} \cap U_C$ in this proof.
 We will choose a subset $V_C'\subset V_C$ (for $V_C$ defined in step one), so the vertical boundary of the fibration $\pi_4|_{M_t^u}$ is divided into two parts, namely a) $\pi_4|_{M_t^u \cap V_C'}$, 
 b) $\pi_4|_{M_t^u \cap (U_C \setminus V'_C)}$ where we use different arguments. We first choose $V_C'$.
 
Let $G^{t,u}:={z}_1{z}_2-tg^u({z})$ for $g^u:U_C \to \CC$ constructed in step one.
In particular, we have $M_t^u=(G^{t,u})^{-1}(0)$.
Near $\Disc({s}_1) \cap U_C$, $G^{t,u}$ is independent of $u$ so by Proposition \ref{p:fibration}, there exists a neighborhood 
$V_C' \subset V_C$ of $\Disc({s}_1) \cap U_C$
such that $\pi_4|_{M_t^u \cap V_C'}$ is a symplectic fibration without singularity for all $0< |t| < \delta$ and all $u$.
Moreover, by Proposition \ref{p:verticalB}, $Z_{M_t^u}$ points outward along vertical boundary of $\pi_4|_{M_t^u \cap V_C'}$ so we are done with a).

For b), as argued in the proof of  Lemma \ref{l:GoodDeformation}, $D({z}_1{z}_2)$ dominates $tDg^u({z})$ outside $V_C'$ when $t$ is small.
Therefore, $\ker(G^{t,u})$ is an arbitrarily small perturbation of $\ker(D({z}_1{z}_2))$ outside $V_C'$ for all $u$ when $t$ is small.
%By the similar argument as in the proof of Lemma~\ref{l:TiltingTangent} and Proposition~\ref{p:fibration}, \ref{p:verticalB}, 
More precisely, $\ker(G^{t,u})|_{M_t^u \setminus V_C'}$ converges uniformly to $\ker(D({z}_1{z}_2))$ outside $V_C'$ for all $u$ when $t$ goes to $0$.
%By the same spirit as in the proof of Lemma~\ref{l:TiltingTangent} and Proposition~\ref{p:fibration}, 
Since $\ker(D({z}_1{z}_2)) \cap \ker(D(z_4))$ is symplectic, the fact that $\ker(G^{t,u})|_{M_t^u \setminus V_C'}$ converges uniformly to $\ker(D({z}_1{z}_2))$ implies that
for $t$ is small,
$\pi_4|_{M_t^u \setminus V_C'}$ is a symplectic fibration without singularity.
On the other hand, by the local model in Lemma \ref{l:3sphere},
we also know that
$\ker(G^{t,u})|_{M_t^u \setminus V_C'}$ converges uniformly to $\ker(D({z}_1{z}_2))$ implies that
for $t$ is small, $Z_{M_t^u}$ points outward along the vertical boundary of $\pi_4|_{M_t^u \setminus V_C'}$.
We conclude that there exists $\delta>0$ such that $\pi_4|_{M_t^u}$ is a symplectic fibration without singularity, and 
$Z_{M_t^u}$ points outward 
 along vertical boundary of the fibration $\pi_4|_{M_t^u}$, for all $0< |t| < \delta$ and all $u$.
 
Finally, we need to deal with the outward-pointing along the horizontal boundary. We have $g^1|_{U_C \setminus V_C}={z}_3-a$ so we 
have $M_t^1 \setminus V_C=\{{z}_1{z}_2= t({z}_3-a)\}$. 
Since the horizontal boundary of $\pi: M_t^1 \to A$ lies inside $M_t^1  \setminus V_C$
and $M_t^1  \setminus V_C$ is independent of the ${z}_4$-coordinate so $\pi$ is symplectically trivial near the horizontal boundary.
By Lemma \ref{l:3sphere}, we know that $Z_{M_t^u}$ also points outward 
 along the horizontal boundary of $\pi$. 
\end{proof}

As a consequence of Proposition \ref{p:LagrangianConstructionSingLoci}, we get the following corollary.

\begin{corollary}\label{c:FinalSolidTori}
Under Proposition \ref{p:fibrationAfterGoodDeform}, there exist a family of proper Lagrangian solid tori $L_t \subset {M}_t^{s^1} \cap U_C$, for $t>0$ small,
such that $L_t \subset {M}_t^{s^1} \cap (U_C \backslash V_C)$ is a
cylindrical Lagrangian over the Legendrian \eqref{eq:boundaryLeg}.

\end{corollary}

\subsection{Proof of Theorem \ref{t:LagrangianSolidTori}}\label{ss:ProofSolidTori}
Recall our convention to write $\hat p_j$ as $p_j$, etc. We undo this convention now to distinguish between the two sets of coordinates. 
The last section used the $\hat p_j$-coordinates. Recall the transformation $\Psi$ between the two sets of coordinates from \eqref{eq:Psi}.
In this section, we apply $\Psi^{-1}$ to transform the Lagrangian solid tori obtained in Corollary \ref{c:FinalSolidTori}
back to $(p,q)$-coordinates. 
After that, we will conclude the proof of Theorem \ref{t:LagrangianSolidTori}.

We start with the situation as in Proposition~\ref{p:fibrationAfterGoodDeform}, so have $U_C,V_C$ of the form \eqref{Ubeta} and a family $(s^u)_{u \in [0,1]}$ of ${s}_1$-admissible sections.
Recall that $\Psi$ is merely a change of coordinates and that it preserves the coordinates on $\{z_1=z_2=0\}=\{\hat z_1=\hat z_2=0\}$.
\begin{comment}
We also have a neighborhood $U_{C}$ given by \eqref{Ubeta}, the discriminant in $U_C$ is given by
$\Disc({s}_1) \cap U_{C}=\{0\} \times \{0\} \times \{a\} \times A_{a_0,a_1}$ 
($2^{\text{nd}}$ bullet of Proposition \ref{p:correctSingModel})
and $V_C$ can be taken to be
\begin{align}
 V_C=\{(({p}_1,{q}_1), \dots,({p}_4,{q}_4)) \in B_{r_1'} \times B_{r_2'} \times D_{a,r_3'} \times A_{a_0,a_1}\} 
\end{align}
for some $0<r_j'<r_j$ are sufficiently small.
Most importantly, ${M}_t^{s^1} \cap (U_C\setminus V_C)$ satisfies \eqref{eq:assumptionHorizontalTrivial}.
We apply $\Psi^{-1}$ to $U_C$, $\Disc({s}_1) \cap U_C$, $V_C$ and ${M}_t^{s^1}$ so that we get a description of them in $(p,q)$-coordinates 
instead of $({p},{q})$-coordinates.
Notice that $\Psi^{-1}|_{\{{p}_1={p}_2=0\}}$ is the identity (see \eqref{eq:Psi}) so
\begin{align}
 \Psi^{-1}(U_C) \cap \{p_1=p_2=0\} &= D_{a,r_3} \times A_{a_0,a_1} \\ 
 \Psi^{-1}(\Disc({s}_1) \cap U_C) \cap \{p_1=p_2=0\} &= \{a\} \times A_{a_0,a_1} \\ 
 \Psi^{-1}(V_C) \cap \{p_1=p_2=0\} &= D_{a,r_3'} \times A_{a_0,a_1}
\end{align}
where $D_{a,r}$ and $A_{a_0,a_1}$ are now defined with respect to $(p_3,q_3)$ and $(p_4,q_4)$-coordinates, respectively.
\end{comment}
By applying $\Psi^{-1}$ to ${M}_t^{s^1}$, that is inserting \eqref{eq:Psi}, we get the following
\begin{align}
 &{M}_t^{s^1} \cap (U_C\setminus V_C)=\{\hat{z}_1\hat{z}_2=t(\hat{z}_3-a)\} \nonumber\\
 =&\{\sqrt{4p_1p_2}e^{i(q_1+q_2+q_3+q_4)} =t(\sqrt{2(p_3-p_2)}e^{iq_3}- a)\}. \label{eq:SympEq}
\end{align}

Recall that the third factor of $U_C$ and $V_C$ are disks centered at $a$, say $D_U$ and $D_V$ respectively ($D_V\subsetneq D_U$).
Recall also that we have a straight line $\gamma(r)=(0,0,r,R)$ in $p$-coordinates for $r_0<r \le r_1$ (see the paragraph before Theorem \ref{t:LagrangianSolidTori}) and $a=\sqrt{2r_1}e^{iq}$, so $\gamma$ ends at $a$.
We choose $r_0 <r'<r''<r'''<r_1$ such that if $A_1$ is the $0$-centered annulus with radii $r',r''$ and $A_2$ is the $0$-centered annulus with radii $r'',r'''$ then
$$A_1\cap D_U=\emptyset,\quad A_2\cap D_V=\emptyset,\quad \hbox{ but }\quad A_2\cap D_U\neq\emptyset,$$
see Figure~\ref{fig:p3q3}.
We want to perform an additional symplectic isotopy for ${M}_t^{s^1}$
so that the new symplectic hypersurface is $x$-standard for all $x\in \gamma([r',r''])$ (see Definition \ref{d:xStandard}), as explained in the following lemma.
For $\epsilon>0$, let $B_\epsilon$ denote the closed $0$-centered $\epsilon$-ball in $\RR^2$. We set
$$W_{1,\epsilon}=B_\epsilon\times B_\epsilon\times A_1\times A,$$ 
$$W_{2,\epsilon}=B_\epsilon\times B_\epsilon\times A_2\times A.$$

\begin{figure}
 \includegraphics[width=4.4cm]{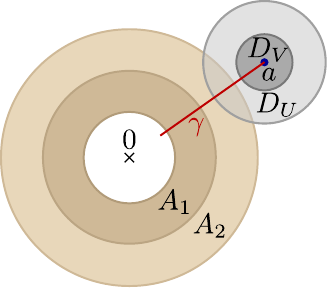} 
 \caption{Disks and annuli in the $z_3$-plane.}
 \label{fig:p3q3}
\end{figure}

\begin{lemma}\label{l:tran}
 Let $N \subset \PP_\Delta$ be an open set such that there exists $\epsilon>0$ with
 $U_C,W_{1,\epsilon},W_{2,\epsilon} \subset N$.
 Then there exists a smooth family $(s^u_{\tran})_{u \in [0,1]}$ of ${s}_1$-admissible sections such that
 $s^0_{\tran}=s^1$ and, for all $u$, $s^u_{\tran}=s^1$ inside $U_C$ and outside $N$. Furthermore,
 \begin{align}
  {M}^{s^1_{\tran}}_t \cap (W_{1,\epsilon}\cup W_{2,\epsilon}) =\{\sqrt{p_1p_2}e^{i(q_1+\dots+q_4)}=tg_{\tran}(p_3-p_2,q_3)\} \label{eq:gtran}
 \end{align}
for some  $g_{\tran}(p_3-p_2,q_3) \in C^{\infty}(W_{1,\epsilon}\cup W_{2,\epsilon} ,\CC)$ such that $g_{\tran}|_{W_{1,\epsilon}}$ is a non-zero constant.

Note that, $g_{\tran}|_{W_{1,\epsilon}}$ being a non-zero constant implies that ${M}^{s^1_{\tran}}_t$ is $x$-standard for $x \in \gamma([r',r''])$.
\end{lemma}

\begin{proof}
 This proof is very similar to the proof of Lemma~\ref{l:ExistenceStandardSymp} and step one of the proof of Proposition~\ref{p:fibrationAfterGoodDeform}.
 We know that ${M}^{s^1}_t \cap (W_{1,\epsilon}\cup W_{2,\epsilon})$ is given by
 \begin{align}
  \sqrt{p_1p_2p_3p_4}e^{i(q_1+\dots+q_4)}=tf
 \end{align}
for some $f \in C^{\infty}(W_{1,\epsilon}\cup W_{2,\epsilon} ,\CC)$ (cf. \eqref{eq:CornerForm}). 
Let $B_A$ denote the smallest $0$-centered ball containing $A$ (i.e.~of radius the larger radius of $A$).
By the fact that $s^1$ is ${s}_1$-admissible, we have $\im(f|_{\{p_1=p_2=0\}}) \subset \CC^*$.
Moreover, since there is an open subset $G$ of $\big(\{0\} \times \{0\} \times B_{r'''} \times B_A\big)$
which is homeomorphic to a ball, contains both the origin and $(A_1\cup A_2) \times A$
and such that $G \cap \Disc({s}_1)= \emptyset$, we have that 
\begin{align}
 (f|_{\{p_1=p_2=0\}})_*: \pi_1((A_1\cup A_2) \times A) \to \pi_1(\CC^*) 
\end{align}
is the zero map (cf. \eqref{eq:homotopyClass}).
Thus, there is no obstruction to constructing a smooth family $f_{\tran,0}^u:(A_1\cup A_2) \times A \to \CC^*$, for $u \in [0,1]$,
such that $f_{\tran,0}^0=f|_{\{p_1=p_2=0\}}$,
$f_{\tran,0}^u$ is independent of $u$ inside $((A_1\cup A_2)\cap D_U) \times A$, 
$f_{\tran,0}^1=\sqrt{p_3p_4}g_{\tran,0}(p_3,q_3)$ for some $g_{\tran,0}:(A_1\cup A_2) \times A \to \CC^*$
and $g_{\tran,0}|_{A_1 \times A}$ is a non-zero constant.

Finally, as in the step one of the proof of Proposition \ref{p:fibrationAfterGoodDeform}, we can extend 
 $f_{\tran,0}^u$ and $g_{\tran,0}$ to $f_{\tran}^u$ and $g_{\tran}$ which are defined over the whole $W_{\epsilon, 1}\cup W_{\epsilon, 2}$ such that, by patching,
$f_{\tran}^u$ induces a family $(s^u_{\tran})_{u \in [0,1]}$ of ${s}_1$-admissible sections with all the properties listed in the proposition satisfied.
In particular, $g_{\tran}$ satisfies \eqref{eq:gtran}.
\end{proof}

Next, we want to describe the family (for $t>0$ small) of proper Lagrangian solid tori $L_t \subset {M}_t^{s^1} \cap U_C$
in Corollary \ref{c:FinalSolidTori} in $(p,q)$-coordinates.
Recall that $a=\sqrt{2r_1}e^{iq}$ and that Remark~\ref{r:FlexibleDisk} permits us to choose any argument for the Legendrian. 
For our purpose, we pick $r$ to be the map $r\mapsto re^{i\phi}$ with $\phi=\frac{\pi+q}{2}$. 
This way, $r^2+a$ parametrizes a curve that starts at $a$ and moves straight towards the origin.
We find $L_t \cap (U_C\setminus V_C)$ in $\hat{z}$-coordinates (as in \eqref{eq:LAGDisk}) be given by
\begin{align*}
\left\{\hat{z}=\left(re^{i\theta_1},re^{-i\theta_1},\frac{r^2}t+a,\sqrt{2R}e^{i\theta_2}\right) \in {M}_t^{s^1} \cap (U_C\setminus V_C)\left|{r \in 
 e^{i\frac{\pi+q}{2}}(0,\infty),}\atop{\theta_1,\theta_2 \in \mathbb{R}/2\pi \mathbb{Z}}\right.\right\}
\end{align*}
and in the $(p,q)$-coordinates (applying \eqref{eq-phat-pnohat} alias inserting \eqref{eq:Psi}) this is described by the following equations
\begin{align}
p_1=p_2=\frac{|r|^2}{2},\qquad p_3=\frac{\big(\sqrt{2r_1}-\frac{|r|^2}t\big)^2}{2}+\frac{|r|^2}{2},\qquad p_4=R+\frac{|r|^2}{2},\label{eq-p-of-Lagrangian}\\
q_1=\frac{\pi+q}{2}+\theta_1,\qquad q_2=\frac{\pi+q}{2}-\theta_1-\theta_2-q,\qquad q_3=q,\qquad q_4=\theta_2.
\label{eq-q-of-Lagrangian}
\end{align}
The tropical curve $\gamma$ is contained in the line through $(0,0,1,R)$ and $(0,0,0,R)$, so in view of Proposition~\ref{p:standardLagModel}, we define
$W$ to be the affine $2$-plane in $\RR^4$ containing the points $(1,1,1,1+R)$, $(0,0,1,R)$ and $(0,0,0,R)$, so $\gamma \subseteq W \cap \partial \coDelta$.
By inspecting \eqref{eq-p-of-Lagrangian}, we see that
$p\in W$ for all $(p,q) \in L_t \cap (U_C\setminus V_C)$
and, by deriving \eqref{eq-q-of-Lagrangian}, we find $\partial_{\theta_1}, \partial_{\theta_2} \in W^{\perp}$.

The following lemma gives a family (for $t>0$ small) of Lagrangian solid tori (with boundary) in ${M}_t^{s^1_{\tran}}$ 
that are $\gamma((r',r''))$-standard (see Definition \ref{def-Lag-standard}).

\begin{lemma}\label{l:LAGtran}
For $s^1_{\tran}$ in Lemma \ref{l:tran}, 
there is a family of Lagrangian solid tori with boundary, for $t>0$ sufficiently small,
$L_t^{\tran} \subset {M}_t^{s^1_{\tran}} \cap (W_{1,\epsilon}\cup W_{2,\epsilon} \cup U_C)$
such that 
\begin{enumerate}
 \item $p \in W \cup \pi_{\Delta}(V_C)$ for all $(p,q) \in L_t^{\tran}$, and 
 \item $W^\perp \subset T_{(p,q)} L_t$ for all $(p,q) \in L_t^{\tran}$ satisfying $p \in W\backslash \pi_{\Delta}(V_C)$, and
 \item the $p_3$-coordinate of all points in the torus boundary $\partial L_t^{\tran}$ is $r'$.
\end{enumerate}

\end{lemma}

\begin{proof} 
 By the construction in Lemma \ref{l:tran}, $s^1_{\tran}|_{U_C}=s^1|_{U_C}$
 so the $L_t$ constructed in Corollary \ref{c:FinalSolidTori} are Lagrangian inside ${M}_t^{s^1_{\tran}}\cap U_C$.
 Inspecting \eqref{eq:boundaryLeg} and \eqref{eq:LegUnknot}, 
 for a fixed $t>0$ sufficiently small, by Proposition~\ref{p:LagrangianConstructionSingLoci} and the paragraph before, the $(p_3,q_3)$-coordinates of all the points in $\partial L_t$ are the same and they lie in $\partial D_U$.
 Therefore, we need to explain how to `extend' $L_t$ to $L_t^{\tran}\subset {M}_t^{s^1_{\tran}} \cap (W_{1,\epsilon}\cup W_{2,\epsilon} \cup U_C)$
 so that, in particular, the $p_3$-coordinate of all the points in the torus boundary $\partial L_t^{\tran}$ equals $r'$.
 
 The proof strategy is the same as Proposition \ref{p:standardLagModel} and Lemma \ref{l:standardTransitionModel}.
 By Lemma \ref{l:tran}, ${M}^{s^1_{\tran}}_t \cap (W_{1,\epsilon}\cup W_{2,\epsilon})$ is given by
  \begin{align} \label{eq-nearby-fibre-Lagtran}
  \sqrt{p_1p_2}e^{i(q_1+\dots+q_4)}=tg_{\tran}(p_3-p_2,q_3).
 \end{align}
 We want to construct a Lagrangian in ${M}^{s^1_{\tran}}_t \cap (W_{1,\epsilon}\cup W_{2,\epsilon})$ such that $q_3=q$ is a constant. We move $\gamma$ inside $W$ from the toric boundary to the nearby fibers as follows:
Consider the function $\rho:=\frac{p_1p_2}{|g_{\tran}(p_3-p_2,q)|^2}$ on $\pi_\Delta(W_{1,\epsilon}\cup W_{2,\epsilon})$, so $\rho=t^2$ is the moment map image of the hypersurface \eqref{eq-nearby-fibre-Lagtran}.
 Let $\nu:=\sum_{j=1}^4 \partial_{p_j}$, so $\nu(g_{\tran}(p_3-p_2,q_3))=0$. 
This implies $\nu(\rho)>0$ for all $p \in \pi_\Delta(W_{1,\epsilon}\cup W_{2,\epsilon}) \backslash \{p_1=p_2=0\}$, so $\rho$ is strictly increasing in the direction $(1,1,1,1)$ and zero on the boundary $\{p_1=0\}\cup\{p_2=0\}$.
 For small $t>0$, for all $r \in (r',r''')$, there exists a unique $\lambda$ such that $p=\gamma(r)+\lambda(1,1,1,1)$ satisfies $\sqrt{p_1p_2}=t|g_{\tran}(p_3-p_2,q)|$.
 Therefore, by the reasoning in Proposition \ref{p:standardLagModel} and Lemma \ref{l:standardTransitionModel}, 
 we get a family of Lagrangians in ${M}^{s^1_{\tran}}_t \cap (W_{1,\epsilon}\cup W_{2,\epsilon})$
 that is $\gamma((r',r''))$-standard.
 By choosing $q_1,q_2,q_4$-coordinates appropriately (cf. \eqref{eq:xStandardL}), this family can be smoothly attached to $L_t$ to give $L_t^{\tran}$ as desired.
\end{proof}

Now recall that we applied a Hamiltonian isotopy $\phi$ in Section \ref{ss:Correcting discriminant} to modify $s_1$ so that the discriminant became straight in the sense of Proposition~\ref{p:correctSingModel} at the endpoint of the tropical curve.
We will account for this step in the following and conclude the proof of Theorem \ref{t:LagrangianSolidTori} where we carefully distinguish between $s_1$ and $\hat s_1$, etc., see \eqref{eq-hat-first}-\eqref{eq-hat-last}.

\begin{proof}[Proof of Theorem \ref{t:LagrangianSolidTori}]
Let $s$ be an $s_1$-admissible section and $N$ be a neighborhood of $\gamma(r_1)$.
Let $r'< r'' <r_1$ be such that $\gamma([r',r_1]) \subset N$.
Consequently, $W_{1,\epsilon} \subset \pi^{-1}_\Delta(N)$ for $\epsilon>0$ small (indeed, recall that $W_{1,\epsilon} $ is defined with respect to $(p,q)$-coordinates).
Now, after applying the integral linear transformation $\Psi$, we let $N_1 \subset \pi_\Delta^{-1}(N)$ be a tubular neighborhood of $C$ such that $W_{1,\epsilon} \cap N_1=\emptyset$, where $N_1$ is defined in $(\hat{p}, \hat{q})$-coordinates.
We apply Proposition \ref{p:correctSingModel} to $N_1$ so that we get a Hamiltonian isotopy $\phi_H$ supported inside $N_1$ to straighten the discriminant.
By Corollary~\ref{c:FinalSolidTori}, there exist neighborhoods $V_C \subset U_C \subset N_1$ of $C$, $\hat{s_1}$-admissible sections $(s^u)_{u \in [0,1]}$ and for all $t>0$ small, Lagrangian solid tori $L_t \subset \hat{M}^{s^1}_t \cap U_C$.
Let $r'''<r_1$ such that the corresponding $W_{2,\epsilon}$ satisfies
$W_{2,\epsilon} \cap U_C \neq \emptyset$ and $W_{2,\epsilon} \cap V_C = \emptyset$ as before.
We can now apply Lemma \ref{l:LAGtran} to obtain $L_t^{\tran} \subset \hat{M}_t^{s^1_{\tran}} \cap (W_{1,\epsilon}\cup W_{2,\epsilon} \cup U_C)$.

Finally, we apply the inverse of the Hamiltonian isotopy $\phi_H$ to get $\phi_H^{-1}(L_t^{\tran}) \subset \phi_H^{-1}(\hat{M}_t^{s^1_{\tran}}) \cap \phi_H^{-1}(W_{1,\epsilon}\cup W_{2,\epsilon} \cup U_C)$.
First note that $\phi_H^{-1}(\hat{M}_t^{s^1_{\tran}})={M}_t^{s^1_{\tran} \circ \phi_H}$ and $s^1_{\tran} \circ \phi_H$ is $s_1$-admissible.
By definition, $\phi_H^{-1}$ is the identity outside $N_1$.
As a result, $\phi_H^{-1}(\hat{M}_t^{s^1_{\tran}})={M}_t^{s^1_{\tran} \circ \phi_H}$ remains $x$ standard in $W_{1,\epsilon}$ (because $W_{1,\epsilon} \cap N_1=\emptyset$).
Moreover, $\phi_H^{-1}(L_t^{\tran}) $ remains $\gamma((r',r''))$-standard for the same reason.
This finishes the proof.
\end{proof}

\subsection{Concluding the proof of Theorem \ref{t:Construction}}\label{ss:Concluding}

\begin{proof}[Proof of Theorem \ref{t:Construction}]
Let $\gamma$ be a tropical curve satisfying the assumptions of Theorem \ref{t:Construction}.
Let $N$ be a neighborhood of $\gamma$.
We can apply Theorem \ref{t:LagrangianSolidTori} to construct open Lagrangian solid tori for the endings of the tropical curve near the discriminant
such that the non-compact ends of the tori are standard with respect to an open subset of $\gamma$.
Therefore, we can apply Proposition \ref{p:assemble} to obtain, for all $t>0$ small, a closed Lagrangian $L_t \subset M^s_t$
such that $\pi_{\Delta}(L_t) \subset N$.

Again, as explained in the proof of Proposition \ref{p:assemble}, we can assume the families of $s_1$-admissible sections we have constructed are constant outside $\pi_{\Delta}^{-1}(N)$.
Therefore, we can apply Lemma \ref{l:symplecticIsotopy} to conclude that  $L_t \subset M^s_t$
can be brought back, via a symplectic isotopy, to a closed embedded Lagrangian inside $M_t \cap \pi_{\Delta}^{-1}(N)$.

The statement regarding multiplicity is proved in Proposition \ref{p:multiplicity}.
\end{proof}

\subsubsection{Orbifold case}\label{ss:OrbifoldConstruction}

When $\PP_\Delta$ is a toric orbifold, the proof of Theorem~\ref{t:Construction} goes very similar. 
First, by Lemma \ref{l:unbranched}, the cover $\CC^n \to \CC^n/K$ is unbranched away from the origin.
It means that if $U/K$ is a symplectic corner chart for $\PP_\Delta$, then $U \backslash \{0\} \to (U \backslash \{0\})/K$
is an unbranched cyclic covering.
Note that near the discriminant, the tropical curve is in the direction $(0,0,1,0)$ with respect to the symplectic corner chart $U$.
Since the cyclic group $K$ is generated by an element in $(\mathbb{R}/2\pi\mathbb{Z})^4$ with non-zero components (otherwise, the orbifold points will not be isolated),
it implies that we are necessarily in the case $K \cap W^\perp_T=\{0\}$ in Proposition \ref{p:standardLagModel}.
Therefore, if we denote the lift of $\Disc \subset U/G$ to $U$ by $\widetilde{\Disc}$, then we can apply Theorem \ref{t:LagrangianSolidTori}
in $U$ to get a family of solid Lagrangian tori $L_t$ near $\widetilde{\Disc}$ and its $K$-orbit is a disjoint 
union of $|K|$ solid Lagrangian tori.
Since $L_t$ are away from the origin in $U$, it descends to a family of solid Lagrangian tori in $U/K$ near $\Disc$.
%The same covering and descending method applies for the construction of Lagrangians away from the discriminant.
Therefore, the result follows.

{\small
\bibliography{ToricLens}
\bibliographystyle{plain}
}
\end{document}